\documentclass{article}
\usepackage{amsmath,amssymb,amsthm,fullpage}
\usepackage{graphicx,subfigure,float,url,color}
\usepackage{mathrsfs}
\usepackage[colorlinks=true]{hyperref}
\usepackage{enumitem}
\usepackage{pdfsync}

\def\R{\mathrm{I\kern-0.21emR}}
\def\N{\mathrm{I\kern-0.21emN}}

\renewcommand{\geq}{\geqslant}
\renewcommand{\leq}{\leqslant}

\newtheorem{theorem}{Theorem}
\newtheorem{proposition}{Proposition}

\newtheorem{lemma}{Lemma}
\theoremstyle{definition}
\theoremstyle{definition}\newtheorem{remark}{Remark}

\newcommand{\Fmax}{F_{\textrm{max}}}

\begin{document}

\title{Linear turnpike theorem}

\author{Emmanuel Tr\'elat\footnote{Sorbonne Universit\'e, CNRS, Universit\'e de Paris, Inria, Laboratoire Jacques-Louis Lions (LJLL), F-75005 Paris, France (\texttt{emmanuel.trelat@sorbonne-universite.fr}).}
}

\date{}
\maketitle
\begin{abstract}
The turnpike phenomenon stipulates that the solution of an optimal control problem in large time remains essentially close to a steady-state of the dynamics, itself being the optimal solution of an associated static optimal control problem. Under general assumptions, it is known that not only the optimal state and the optimal control, but also the adjoint state coming from the application of the Pontryagin maximum principle, are exponentially close to that optimal steady-state, except at the beginning and at the end of the time frame. In such a result, the turnpike set is a singleton, which is a steady-state.

In this paper, we establish a turnpike result for finite-dimensional optimal control problems in which some of the coordinates evolve in a monotone way, and some others are partial steady-states of the dynamics. We prove that the discrepancy between the optimal trajectory and the turnpike set is linear, but not exponential: we thus speak of a linear turnpike theorem.

 \end{abstract}

\section{Introduction and main results}


\subsection{Reminders on the exponential turnpike phenomenon}\label{sec_reminders}
Let $n,m\in\N^*$ and let $T>0$ be arbitrary. We consider a general nonlinear optimal control problem in $\R^n$, in fixed final time $T$: 
\begin{align}
& \dot x(t) = f(x(t),u(t)) \label{syst} \\[1mm]
& x(0)\in M_0,\quad x(T)\in M_1 \label{terminalconditions} \\
&  u(t)\in\Omega \label{syst_Omega} \\
& \min \int_0^T f^0(x(t),u(t))\, dt \label{mincost}
\end{align}
where $f:\R^n\times\R^m\rightarrow\R^n$ and $f^0:\R^n\times\R^m\rightarrow\R$ are mappings of class $C^2$, 
$\Omega$ is a measurable subset of $\R^m$,
$M_0$ and $M_1$ are bounded subsets of $\R^n$, and $u\in L^\infty([0,T],\Omega)$ is the control.
For simplicity of the exposition, we assume that there exists $T_0>0$ such that, for every $T\geq T_0$, there exists a unique optimal solution of \eqref{syst}-\eqref{terminalconditions}-\eqref{syst_Omega}-\eqref{mincost}, denoted by $(x^T(\cdot),u^T(\cdot))$ (sufficient conditions ensuring existence are standard, see, e.g., \cite{trelat_book}; uniqueness is ensured under differentiability properties of the value function and of the Hamiltonian, see, e.g., \cite[Theorem 7.3.9]{CannarsaSinestrari}, and thus is generic in some sense). 
By the Pontryagin maximum principle (see \cite{Pontryagin}), there exist $p^0\leq 0$ and an absolutely continuous mapping $p_x^T(\cdot):[0,T]\rightarrow\R^n$ (called adjoint vector), satisfying $(p_x^T(\cdot),p^0)\neq(0,0)$, such that
\begin{equation}\label{extremal_syst}
\begin{split}
& \dot x^T(t) = \frac{\partial H}{\partial p_x}(x^T(t),p_x^T(t),p^0,u^T(t)),
\qquad \dot p_x^T(t) = -\frac{\partial H}{\partial x}(x^T(t),p_x^T(t),p^0,u^T(t)) \\
&\qquad\qquad\qquad u^T(t) = \underset{v\in\Omega}{\mathrm{argmax}}\, H(x^T(t),p_x^T(t),p^0,v) 
\end{split}
\end{equation}
for almost every $t\in[0,T]$, where $H(x,p_x,p^0,u) = \langle p_x,f(x,u)\rangle + p^0 f^0(x,u)$
is the Hamiltonian of the optimal control problem \eqref{syst}-\eqref{terminalconditions}-\eqref{syst_Omega}-\eqref{mincost} (here and throughout, $\langle\cdot,\cdot\rangle$ is the Euclidean scalar product in $\R^n$, and $\Vert\cdot\Vert$ is the corresponding norm). We have moreover the transversality conditions on the adjoint vector at initial and final time: 
\begin{equation}\label{transversality}
p_x^T(0)\perp T_{x^T(0)}M_0,\qquad p_x^T(T)\perp T_{x^T(T)}M_1,
\end{equation}
provided that the tangent space $T_{x^T(0)}M_0$ to $M_0$ at $x^T(0)$ and the tangent space $T_{x^T(T)}M_1$ to $M_1$ at $x^T(T)$ exist (these conditions are empty when the initial and final points are fixed in the optimal control problem).
We assume that the extremal lift $(x^T(\cdot),p_x^T(\cdot),p^0,u^T(\cdot))$ is not abnormal and thus that we can take $p^0=-1$. We also assume, for the simplicity of the exposition, that the extremal lift is unique (these two conditions are generic for large classes of optimal control problems, see \cite{CJT_JDG2006, CJT_SICON2008}).

Note that the last condition in \eqref{extremal_syst} gives
$$
\frac{\partial H}{\partial u}(x^T(t),p_x^T(t),p^0,u^T(t)) = 0
$$
at every Lebesgue time $t\in[0,T]$ at which $u^T(\cdot)$ is approximately continuous and $u^T(t)\in\mathring{\Omega}$ (interior of $\Omega$).

\medskip

The turnpike phenomenon stipulates that, when the final time $T$ is large, the optimal solution $(x^T(\cdot),p_x^T(\cdot),u^T(\cdot))$ remains, along the time interval $[0,T]$ and except around the initial time $t=0$ and around the final time $t=T$, ``essentially close" to some static point $(\bar x,\bar p_x,\bar u)\in\R^n\times\R^n\times\Omega$, i.e., roughly speaking, 
\begin{equation}\label{roughturnpike}
x^T(t) \simeq\bar x,\quad p_x^T(t) \simeq\bar p_x,\quad u^T(t) \simeq\bar u \qquad \forall t\in[\eta,T-\eta]
\end{equation}
(for some $\eta>0$). Moreover, $(\bar x,\bar p_x,\bar u)$ is the solution of a minimization problem called \textit{static} optimal control problem (not depending on $T$), which is:
\begin{equation}\label{staticpb}
\min \left\{ f^0(x,u) \ \mid\ (x,u)\in \R^n\times\Omega,\ f(x,u)=0 \right\}
\end{equation}
i.e., in other words, the problem of minimizing the instantaneous cost $f^0(x,u)$ over all possibles equilibrium points of the dynamics.
Note that, assuming that $\bar u\in\mathring{\Omega}$ (interior of $\Omega$), by the Lagrange multiplier rule, expressed in a Hamiltonian form, applied to the static optimal control problem, we have
\begin{equation}\label{staticextr}
\frac{\partial H}{\partial p_x}(\bar x,\bar p_x,p^0,\bar u)=0,\qquad
\frac{\partial H}{\partial x}(\bar x,\bar p_x,p^0,\bar u)=0,\qquad
\frac{\partial H}{\partial u}(\bar x,\bar p_x,p^0,\bar u)=0 .
\end{equation}
Here we assume as well, for the simplicity of the exposition, that the minimizer $(\bar x,\bar u)\in\R^n\times\mathring{\Omega}$ is unique, and also that the quadruple $(\bar x,\bar p_x,p^0,\bar u)$ is not abnormal, so that one can take $p^0=-1$ (usual qualification condition in optimization theory) and that the Lagrange multiplier is unique (the latter two conditions are generic).

\medskip

We note then that $(\bar x,\bar p_x,-1,\bar u)$ is an equilibrium point of \eqref{extremal_syst}. Hence, the informal property \eqref{roughturnpike} stipulates that the optimal triple $(x(\cdot),p_x(\cdot),u(\cdot))$, solution of \eqref{syst}-\eqref{terminalconditions}-\eqref{syst_Omega}-\eqref{mincost}-\eqref{extremal_syst}-\eqref{transversality} remains essentially close to an \emph{equilibrium point} of the dynamics \eqref{extremal_syst}, this equilibrium being itself an optimal solution of the static problem \eqref{staticpb}. This is the turnpike phenomenon.

The intuition explaining why such a property is to be expected is the following. Since $T$ is assumed to be large, we set $\varepsilon=\frac{1}{T}$, which is considered as a small parameter. Now, we make a time reparametrization, by setting $s=\varepsilon t$, so that, when $t$ ranges over $[0,T]$, $s$ ranges over $[0,1]$. Setting $\mathrm{x}_\varepsilon(s)=x^T(t)$, $\mathrm{p}_\varepsilon(s)=p_x^T(t)$ and $\mathrm{u}_\varepsilon(s)=u^T(t)$, we have, still in an informal way,
\begin{equation}\label{informal}
\begin{split}
& \frac{\partial H}{\partial p_x}(\mathrm{x}_\varepsilon(s),\mathrm{p}_\varepsilon(s),-1,\mathrm{u}_\varepsilon(s)) = \varepsilon\mathrm{x}_\varepsilon'(s) \simeq 0 \\
& \frac{\partial H}{\partial x}(\mathrm{x}_\varepsilon(s),\mathrm{p}_\varepsilon(s),-1,\mathrm{u}_\varepsilon(s)) = -\varepsilon\mathrm{p}_\varepsilon'(s) \simeq 0 \\
& \frac{\partial H}{\partial u}(\mathrm{x}_\varepsilon(s),\mathrm{p}_\varepsilon(s),-1,\mathrm{u}_\varepsilon(s)) = 0
\end{split}
\end{equation}
(the latter condition, assuming that the control is in the interior of $\Omega$)
which gives credibility to \eqref{roughturnpike}, at least when $(\bar x,\bar p_x,\bar u)$ is the unique solution of \eqref{staticextr} (and under nondegenerate assumptions on the system of equations \eqref{staticextr}). But this is not a proof and this remains informal: indeed, to justify \eqref{informal} one would need that the derivatives $\mathrm{x}_\varepsilon'(s)$ and $\mathrm{p}_\varepsilon'(s)$ be bounded (except around $t=0$ and $t=T$), uniformly with respect to $\varepsilon$, i.e., that the derivatives $T\dot x^T(t)$ and $T\dot p_x^T(t)$ be bounded (except around $t=0$ and $t=T$), uniformly with respect to $T$, which is precisely the hard part in establishing a turnpike result.

\medskip

The main result of \cite{TZ} states an \emph{exponential turnpike property} which is of a local nature (see Remark \ref{rem_local}). Hereafter, we combine this result with a general result of \cite{TrelatZhang_MCSS2018} in order to obtain a global exponential turnpike property. We use the notion of \emph{strict dissipativity}, introduced in \cite{Willems} and already used in \cite{Faulwasser_TAC2017, Grune1, GruneMuller_SCL2016, GruneGuglielmi_SICON2018} to derive turnpike properties. The family of optimal control problems \eqref{syst}-\eqref{syst_Omega}-\eqref{mincost}, indexed by $T>0$, is said to be \emph{strictly dissipative} at the optimal static point $(\bar x, \bar u)$ with respect to the \emph{supply rate function} 
\begin{equation}\label{supply}
w(x,u)= f^0(x,u)-f^0(\bar x,\bar u)
\end{equation}
if there exists a locally bounded function $S:\R^n\rightarrow\R$, called \emph{storage function}, such that
\begin{equation}\label{dissip}
S(x(t_0))+\int_{t_0}^{t_1} w(x(t),u(t))\,dt\geq S(x(t_1)) +  \int_{t_0}^{t_1} \alpha\left( \Vert x(t)-\bar x\Vert, \Vert u(t)-\bar u\Vert \right) dt 
\end{equation}
(strict dissipativity inequality)
for all $t_0<t_1$, for any pair $(x(\cdot),u(\cdot))$ solution of \eqref{syst} and \eqref{syst_Omega}, for some function $\alpha$  of class $\mathcal{K}$, i.e., $\alpha:[0,+\infty)\rightarrow[0,+\infty)$ continuous, increasing, and $\alpha(0)=0$. Comments on dissipativity can be found in \cite{Brogliato,TrelatZhang_MCSS2018,Willems} (see also further).

We use hereafter the shorter notations
$$
\bar H_{\star\#} = \frac{\partial^2 H}{\partial\star\partial\#}(\bar x,\bar p_x,-1,\bar u)
$$
where $\star$ and $\#$ will be replaced with such or such variable. 
We set
$$
\bar A_1 = \bar H_{p_x x} = \frac{\partial f}{\partial x}(\bar x,\bar u),\qquad
\bar B_1 = \bar H_{p_x u} = \frac{\partial f}{\partial u}(\bar x,\bar u), \qquad \bar U=-\bar H_{uu}
$$
and, since we assume hereafter that $\bar U$ is invertible,
$$
A_1 = \bar A_1 + \bar B_1 \bar U^{-1} \bar H_{ux}, \qquad \bar W = -\bar H_{xx} - \bar H_{xu}\bar U^{-1}\bar H_{ux} .
$$

\begin{theorem}\label{thm_TZ}
We make the following assumptions:\\
Assumptions of global nature:
\begin{enumerate}[label=(\roman*)]
\item\label{H1_unique1} There exist $K>0$ and $T_0>0$ such that, for every $T\geq T_0$, there exists a unique optimal triple $(x^T(\cdot),p_x^T(\cdot),u^T(\cdot))$ solution of \eqref{syst}-\eqref{terminalconditions}-\eqref{syst_Omega}-\eqref{mincost}-\eqref{extremal_syst}-\eqref{transversality} (assumed to be normal); 
moreover, $\Vert x^T(t)\Vert+\Vert u^T(t)\Vert\leq K$ for every $t\in[0,T]$.
\item\label{H1_unique2} There exists a unique optimal triple $(\bar x,\bar p_x,\bar u)$ solution of \eqref{staticpb}-\eqref{staticextr}, assumed to be normal; 
moreover, $\bar u\in\mathring{\Omega}$. 
\item\label{H1_dissip} The family of optimal control problems \eqref{syst}-\eqref{syst_Omega}-\eqref{mincost}, indexed by $T>0$, is strictly dissipative at the optimal static point $(\bar x, \bar u)$ with respect to the supply rate function $w$ defined by \eqref{supply}, with a storage function $S$ (satisfying \eqref{dissip}) that is bounded on $M_0$ and $M_1$.
\item\label{H1_cont} For $i=1,2$, there exist $t_i\geq 0$, a control $u_i(\cdot)\in L^\infty([0,t_i],\Omega)$ and a trajectory $x_i(\cdot)$ solution of $\dot x_i(t)=f(x_i(t),u_i(t))$ on $[0,t_i]$, satisfying $\Vert x_i(t)\Vert+\Vert u_i(t)\Vert\leq K$ for every $t\in[0,t_i]$ and the terminal conditions $x_1(0)\in M_0$, $x_1(t_1)=\bar x$, and $x_2(0)=\bar x$, $x_2(t_2)\in M_1$. 
\end{enumerate}
Assumptions of local nature:
\begin{enumerate}[label=(\roman*), start=5]
\item\label{H1_U} $\bar U=-\bar H_{uu}=-\frac{\partial^2 H}{\partial u^2}(\bar x,\bar p_x,-1,\bar u)$ is a positive definite symmetric matrix.
\item\label{H1_W} $\bar W$ is a positive definite symmetric matrix.
\item\label{H1_Kalman} The pair $(\bar A_1,\bar B_1)$ satisfies the Kalman condition\footnote{\label{footnote1}Equivalently, the pair $(A_1,\bar B_1)$ satisfies the Kalman condition.
}, i.e., $\mathrm{rank}(\bar B_1,\bar A_1\bar B_1,\ldots,\bar A_1^{n-1}\bar B_1)=n$. In other words, the linearized system at $(\bar x,\bar u)$ is controllable.
\end{enumerate}
Then, there exist $C>0$ and $\nu>0$ such that if $T\geq T_0$
then
\begin{equation}\label{turnpikeexp}
\boxed{
\Vert x^T(t)-\bar x\Vert + \Vert p_x^T(t)-\bar p_x\Vert + \Vert u^T(t)-\bar u\Vert \leq C \left( e^{-\nu t} + e^{-\nu(T-t)}\right) \qquad \forall t\in [0,T] 
}
\end{equation}
\end{theorem}

In other words, except around $t=0$ and $t=T$, the optimal triple $(x^T(\cdot),p_x^T(\cdot),u^T(\cdot))$ solution of \eqref{syst}-\eqref{terminalconditions}-\eqref{syst_Omega}-\eqref{mincost}-\eqref{extremal_syst}-\eqref{transversality} is exponentially close to the optimal triple $(\bar x,\bar p_x,\bar u)$ solution of \eqref{staticpb}-\eqref{staticextr}.
The constants $C$ and $\nu$ (not depending on $T$) can even be made explicit by solving a Riccati equation (see \cite{TZ} for details).

The local assumptions \ref{H1_U}-\ref{H1_W}-\ref{H1_Kalman} done in this theorem, which concern the linearization at the turnpike, have been chiefly discussed and commented in \cite{TZ}. 

The assumption \ref{H1_cont} is a global controllability assumption: there exists a trajectory steering in finite time the control system \eqref{syst} from $M_0$ to the turnpike $\bar x$, and from the turnpike $\bar x$ to the final set $M_1$.
Global controllability is usually established under appropriate Lie bracket assumptions (see, e.g., \cite{Jurdjevic} for control-affine systems, and \cite{AgrachevSachkov} for general nonlinear systems).
Assumption \ref{H1_cont} can also follow from the local controllability at the turnpike (Assumption \ref{H1_Kalman}) combined with the (weak) assumption of stabilizability of the system to the turnpike. The latter assumption is a weak one and is related to the dissipativity assumption \ref{H1_dissip}.

In Assumption \ref{H1_unique1}, the property of uniqueness of the solution of the optimal control problem and normality of the extremal triple is actually generic, in some sense: it is satisfied if the Hamiltonian system is sufficiently regular and if the value function corresponding to the optimal control problem is differentiable at the terminal points under consideration (see \cite[Theorem 7.3.9]{CannarsaSinestrari}). The $K$-boundedness assumption (uniform with respect to $T$) precludes any trajectory blow-up. Note that the boundedness of the adjoint vector can be deduced from the normality assumption for a large class of optimal control problems (see \cite{trelat_JDCS2000}); here, we will deduce it from the dissipativity assumption and from the normality of the static problem.
Assumption \ref{H1_unique2} is generic as well, and is related to the qualification of a usual optimization problem under constraints.
But, such comments are completely general in optimal control theory and we thus do not comment more on the generality of our assumptions.

Assumption \ref{H1_dissip} on dissipativity is certainly the most difficult to check in practice. When a control system is dissipative, the question of determining a storage function is similar to the one of finding an appropriate Lyapunov function in order to establish an asymptotic stability domain for a given equilibrium point of a dynamical system. Like in Lyapunov theory, storage functions can be determined by solving a nonlinear partial differential equation (see, e.g., \cite{Brogliato}).

It is proved in \cite{TrelatZhang_MCSS2018} that ``strong duality implies dissipativity". To be more precise, we recall that the static problem \eqref{staticpb} has the \emph{strong duality property} if $(\bar x,\bar u)$  minimizes the \emph{Lagrangian function} $L(\cdot,\cdot,\bar p_x): \R^n\times \R^m\times\R^n\rightarrow \R$ defined by $L(x,u,\bar p_x)= f^0(x,u)-\langle \bar p_x, f(x,u)\rangle$ (the sign minus is due to the fact that we took $p^0=-1$ for the Lagrange multiplier associated with the cost). This notion is well known in classical optimization theory, in relation with primal and dual problems. It is proved in \cite[Theorem 3]{TrelatZhang_MCSS2018} that the strong duality property implies the dissipativity property, with the storage function $S(x)=\langle \bar p_x,x\rangle$.

Note that strong duality is in some sense an infinitesimal version of the dissipativity inequality, at least when the storage function is continuously differentiable (see also \cite{Faulwasser_TAC2017, Grune2}).
Note also that the infinitesimal form of the (non-strict) dissipativity inequality \eqref{dissip}, with $\alpha=0$, is the Hamilton-Jacobi inequality 
$$
H_1(x,\nabla S(x),u)\leq -f^0(\bar x,\bar u) ,
$$
where $H_1$ is the maximized normal Hamiltonian (indeed, divide by $t_1-t_0>0$ and take the limit $t_1-t_0\rightarrow 0$). The existence of $C^1$ solutions is therefore related to the so-called weak KAM theory (see \cite{Fathi}). In this context, the singleton $\{(\bar x,\bar u)\}$ is the Aubry set and $f^0(\bar x,\bar u)$ is the Ma{\~{n}}\'e critical value. 

\begin{remark}
Under \ref{H1_U}-\ref{H1_W}-\ref{H1_Kalman}, 
Assumptions \ref{H1_unique1}-\ref{H1_unique2}-\ref{H1_dissip}-\ref{H1_cont} are automatically satisfied in the linear-quadratic case  where
$$
f(x,u) = \bar A_1 x+\bar B_1 u, \qquad f^0(x,u) = \frac{1}{2} (x-x_d)^\top\bar W(x-x_d)+ \frac{1}{2} (u-u_d)^\top\bar U(u-u_d)
$$
for some $x_d\in\R^n$ and $u_d\in\R^m$, and $\Omega=\R^m$.
\end{remark}

\begin{remark}\label{rem_local}
Theorem \ref{thm_TZ} does not appear as such in the literature. Compared with \cite{TZ} where the main result is established under the assumptions of local nature only (i.e., \ref{H1_U}-\ref{H1_W}-\ref{H1_Kalman}), we have added here assumptions of global nature (i.e., \ref{H1_unique1}-\ref{H1_unique2}-\ref{H1_dissip}-\ref{H1_cont}): compactness and uniqueness of minimizers, strict dissipativity, global controllability to and from the turnpike set, and then we obtain a \emph{global} turnpike property. 

When the minimizers are not unique or when the system is not dissipative, we obtain only a \emph{local} turnpike property, in the sense that the turnpike estimate \eqref{turnpikeexp} can only be expected to be satisfied in a neighborhood of an extremal steady-state $(\bar x,\bar p_x,\bar u)$ (see \cite{TZ}).
This is so, because the dynamics is nonlinear. In dynamical systems theory, when an equilibrium is not degenerate, the solutions of the nonlinear system resemble the solutions of the linearized system at the equilibrium, only locally around the equilibrium. The above turnpike theorem is of this nature. 

When the static problem has several local minimizers, then one has local turnpike properties. When the static problem has, for instance, two global minimizers, then there is a region of the state space in which globally optimal solutions of the optimal control problems have a turnpike property corresponding to the first global minimizer, and another region in which the turnpike property corresponds to the second global minimizer. We refer to Appendix \ref{sec_app} for a discussion on these topics and for some numerical simulations.
\end{remark}

The turnpike result of \cite{TZ} has been extended to the infinite-dimensional setting in \cite{TrelatZhang_MCSS2018, TrelatZhangZuazua_SICON2018} and in \cite{Faulwasser_arxiv2020, GruneSchallerSchiela_JDE2020}. One can find many turnpike results in the literature (see \cite{CarlsonHaurie, Faulwasser_TAC2017, GruneMuller_SCL2016, GruneGuglielmi_SICON2018, PorrettaZuazua_2013, Rapaport_2004, Zaslavski_2006, Zaslavski_2015}, and see historical references in \cite{FaulwasserGrune_handbook,TZ}), but the specificity of the above result is that the turnpike property is established as well for the adjoint vector: as explained in \cite{TZ}, this is particularly important in view of deriving appropriate initializations in numerical computation methods (direct methods or shooting methods, see \cite{Trelat_JOTA2012}).

\begin{remark}\label{rem_state_constraints}
In the optimal control problem \eqref{syst}-\eqref{terminalconditions}-\eqref{syst_Omega}-\eqref{mincost}, we have put some control constraints $u(t)\in\Omega$. This is a (slight) novelty with respect to \cite{TZ} where there is no control constraint. We say the novelty is slight, because we however assume that $\bar u\in\mathring{\Omega}$, so that, in the long middle part of the time interval, the optimal control, which is close to $\bar u$, does not saturate the constraints. Bang arcs can only occur at the beginning and at the end of the time frame.

We can as well add some state constraints: $x(t)\in M$ for every $t\in[0,T]$, where $M$ is a subset of $\R^n$. This does not lead to significant changes, provided that, in Assumption \ref{H1_unique2}, we assume not only that $\bar u\in\mathring{\Omega}$ but also that $\bar x\in\mathring{M}$.
In presence of state constraints, the extremal system \eqref{extremal_syst} may change: along a boundary arc, there is an additional adjoint vector $\mu_x^T(\cdot)$, standing for the Lagrange multiplier of the state constraint, which may live in a measure space (see, e.g., \cite{Clarke}).
Except the additional assumption $\bar x\in\mathring{M}$, the rest of Theorem \ref{thm_TZ} is unchanged. Indeed, the optimal trajectory remains exponentially close to the optimal steady-state, on an interval $[\varepsilon,T-\varepsilon]$, hence $x^T(\cdot)\in\mathring{M}$ along this interval and thus $\mu_x^T(\cdot)=0$ there along. This is only at the beginning and at the end that the multiplier $\mu_x^T(\cdot)$ may be nontrivial.

In presence of control and/or state constraints, if $\bar u\in\partial\Omega$ and/or $\bar x\in\partial M$, i.e., if the optimal steady-state saturates the contraints, it is not clear how to derive an exponential turnpike property by exploiting the Pontryagin extremal system, because our approach strongly relies on a linearization of this first-order optimality system at the optimal steady-state and requires smoothness properties of the extremal control. In such cases, the approach by dissipativity, which is softer, is more appropriate but leads to weaker turnpike results, like the ``measure-turnpike property" (see \cite{Faulwasser_TAC2017,  GruneMuller_SCL2016, TrelatZhang_MCSS2018}). 
\end{remark}

Until now, we have discussed a turnpike phenomenon around an \emph{equilibrium point} of the dynamics.
In the next section, we state a turnpike theorem for systems having some coordinates that evolve in a monotone way and thus have no equilibrium point. 
Theorem \ref{thm_TZ} (and its proof) will thus be a particular case of Theorem \ref{turnpike_thm}.

\subsection{A new result: linear turnpike phenomenon}
\subsubsection{Framework}
We keep the notations introduced in Section \ref{sec_reminders}, and we add to the control system \eqref{syst}-\eqref{syst_Omega}-\eqref{mincost} $p$ additional differential equations, for some $p\in\N^*$, as follows. 
Given any $T>0$, we consider the general optimal control problem in $\R^{n+p}$ in fixed final time $T$:
\begin{align}
&\displaystyle \dot x(t) = f(x(t),u(t))   \label{syst1} \\[1mm]
&\displaystyle \dot y(t) = g(x(t),u(t)) \label{syst2} \\[1mm]
&\displaystyle x(0)\in M_0, \quad x(T)\in M_1, \qquad y(0)=y_0, \quad y(T)=y_1^T  \label{terminalconditions_12} \\[1mm]
&\displaystyle u(t)\in\Omega \label{syst12_Omega} \\
&\displaystyle \min \int_0^T f^0(x(t),u(t))\, dt \label{mincost1}
\end{align}
where $y_0,y_1^T\in\R^p$ are fixed and $g:\R^n\times\R^m\rightarrow\R^p$ is of class $C^2$ (the rest of the assumptions is unchanged, with respect to Section \ref{sec_reminders}). We assume that $m\geq p$.

Compared with \eqref{syst}-\eqref{terminalconditions}-\eqref{syst_Omega}-\eqref{mincost}, we have added the $p$ differential equations $\dot y(t) = g(x(t),u(t))$, where $g$ does not depend on the new coordinate $y\in\R^p$. The instantaneous cost function $f^0$ also, still does not depend on $y$. This is important to derive the result hereafter. 

As in the previous section, the final time $T$ is fixed in the optimal control problem \eqref{syst1}-\eqref{syst2}-\eqref{terminalconditions_12}-\eqref{syst12_Omega}-\eqref{mincost1} and the linear turnpike property will describe the behavior of optimal solutions when $T$ is large.

In \eqref{terminalconditions_12}, the data $M_0$, $M_1$ and $y_0$ do not depend on $T$, but the final condition $y_1^T$ depends on $T$: we have thus put a superscript $T$ in order to underline this important dependence. We have indeed in mind to consider cases where $\Vert y_1^T\Vert$ is large as $T$ is large: in such cases, we expect that the components of $y(t)$ should evolve in a monotone way, i.e., $g(x(t),u(t))\neq 0$ along the optimal trajectory: we have then no equilibrium in the $y$ components. 

We are going to derive a turnpike result in which the ``turnpike set" consists of a \emph{partial equilibrium}, i.e., the turnpike is a path $t\mapsto(\bar x,\bar y(t),\bar u)\in\R^n\times\R^p\times\R^m$ where $f(\bar x,\bar u)=0$. 

\begin{remark}
Following \cite{FaulwasserOberBlobaum_2019, FaulwasserOberBlobaum_arxiv2020}, it is worth noting that such problems \eqref{syst1}-\eqref{syst2}-\eqref{syst12_Omega}-\eqref{mincost1} frequently appear in applications, for instance in mechanical problems where, denoting by $\mathrm{x}(t)$ a position variable, by $\mathrm{v}(t)=\dot{\mathrm{x}}(t)$ the speed variable, one has a system of the form
$$
\dot{\mathrm{x}}(t) = \mathrm{v}(t), \qquad \dot{\mathrm{v}}(t) = f(\mathrm{v}(t),\mathrm{u}(t)) .
$$
The result that we are going to derive allows one to obtain a turnpike phenomenon along a trajectory $\mathrm{x}(t) = \overline{\mathrm{v}} t$, where $\overline{\mathrm{v}}\in\R^n$ is such that $f(\overline{\mathrm{v}},\overline{\mathrm{u}})=0$. Here, $(\overline{\mathrm{v}},\overline{\mathrm{u}})$ is an equilibrium of $f$, but $(\overline{\mathrm{v}} t,\overline{\mathrm{v}},\overline{\mathrm{u}})$ is not an equilibrium of the complete dynamics. 
In \cite{FaulwasserOberBlobaum_2019, FaulwasserOberBlobaum_arxiv2020}, this phenomenon is called a ``velocity turnpike". Because of the estimate that we are going to establish in Theorem \ref{turnpike_thm} hereafter, we rather speak of a ``linear turnpike" phenomenon.
\end{remark}

Note that, in \eqref{syst1}-\eqref{syst2}-\eqref{terminalconditions_12}-\eqref{syst12_Omega}-\eqref{mincost1}, we have fixed the initial and final conditions on $y(t)$. Indeed, otherwise, if either $y(0)$ or $y(T)$ is let free, then there is nothing new, because $f$ and $f^0$ do not depend on $y$ and then \eqref{syst1}-\eqref{syst2}-\eqref{terminalconditions_12}-\eqref{syst12_Omega}-\eqref{mincost1} is then equivalent to \eqref{syst}-\eqref{terminalconditions}-\eqref{syst_Omega}-\eqref{mincost}, with the coordinate $y(t)$ evolving without having any impact on the optimal control problem \eqref{syst}-\eqref{terminalconditions}-\eqref{syst_Omega}-\eqref{mincost}.
But here, we impose $y(0)=y_0$ and $y(T)=y_1^T$, which creates additional constraints with respect to \eqref{syst}-\eqref{terminalconditions}-\eqref{syst_Omega}-\eqref{mincost}, and then the turnpike result stated in Theorem \ref{thm_TZ} cannot be applied.

\subsubsection{A motivating example}
In order to shed light on what can be expected, let us consider the very simple example in $\R^2$:
\begin{equation}\label{simple_example}
\begin{array}{lll}
\displaystyle \dot x(t) = u(t), &  x(0)=1, & x(T)=2 \\[2mm]
\displaystyle \dot y(t) = x(t), \quad & y(0)=0, & y(T)=y_1^T \\[2mm]
\displaystyle \min \int_0^T (x(t)^2+u(t)^2)\, dt &&
\end{array}
\end{equation}
for which the optimal trajectory can be computed explicitly. Here, $y_1^T\in\R$ is a smooth function of $T$ that is sublinear at infinity. For instance, one can take $y_1^T=\alpha T$ or $y_1^T=\alpha T\sin(T)$ for some $\alpha\in\R$.
We also assume that $\Omega=\R$, i.e., there is no control constraint.

In a first step, if we ignore the dynamics in $y$ in \eqref{simple_example}, then we are in the framework of Theorem \ref{thm_TZ}. The static optimal control problem consists of minimizing $x^2+u^2$ under the constraint $u=0$, which gives the optimal solution $\bar x=0$, $\bar u=0$ with the Lagrange multiplier $\bar p_x=0$ (which is normal; here, we take $p^0=-1/2$).
By Theorem \ref{thm_TZ}, we obtain $\vert x^T(t)\vert +\vert p_x^T(t)\vert
+\vert u^T(t)\vert \leq C ( e^{-\nu t} + e^{-\nu(T-t)} )$ for every $t\in[0,T]$. In other words, roughly, $x(t)\simeq 0$ and $u(t)\simeq 0$, except around $t=0$ and $t=T$.

Now, let us take into account the additional dynamics in $y$ with the terminal conditions on $y$. 
We are not anymore in the framework of Theorem \ref{thm_TZ}, because there is no term in $y(t)^2$ in the cost and therefore the matrix $W$ is not positive definite. 
We define the ``turnpike-static" optimal control problem (depending on $T$) as
\begin{equation}\label{stpbex}
\min \{ x^2+u^2 \ \mid\ (x,u)\in\R\times\R,\quad u=0,\quad Tx=y_1^T \}
\end{equation}
i.e., with respect to the previous static problem, we add the terminal constraints coming from the dynamics in $y$: the unique solution of the Cauchy problem $\dot y(t)=x$ (with $x$ constant), $y(0)=0$, must satisfy $y(T)=y_1^T$. 

It is easy to check that the (unique) optimal solution of \eqref{stpbex} is $\bar x^T=\frac{y_1^T}{T}$, $\bar u^T=0$, with Lagrange multipliers $\bar p_x^T=0$ associated with the constraint $u=0$, and $\bar p_y^T=\frac{y_1^T}{T}$ associated with the constraint $x=\frac{y_1^T}{T}$ (which are normal; here, we take $p^0=-1/2$). 
The corresponding trajectory in $y$ is $\bar y^T(t)=\frac{y_1^T}{T}t$ (and we have $\bar y^T(T)=y_1^T$ as expected). For this example, we expect that, if $T>0$ is large, then $x^T(t)\simeq 0$, $u^T(t)\simeq 0$ and $y^T(t)\simeq\bar y^T(t)$ for $t\in[0,T]$, where the order of this approximation is to be determined, what we do hereafter by explicit calculations.

Note that, now, the solution of the ``turnpike-static" optimal control problem, denoted with upper bars, depends on $T$.

Let us apply the Pontryagin maximum principle to \eqref{simple_example}.
The Hamiltonian is $H = p_x u + p_y x - \frac{1}{2} (x^2+u^2)$ (there is no abnormal minimizer). Then $u^T(t)=p_x^T(t)$ and the extremal system is 
$$
\dot x^T(t) = p_x^T(t),\quad \dot y^T(t)=x^T(t),\quad \dot p_x^T(t)=x^T(t)-p_y^T
$$
with $p_y^T$ constant.
We have then
\begin{align*}
\frac{d}{dt} \begin{pmatrix} x^T(t)-p_y^T \\ p_x^T(t) \end{pmatrix}
&= 
\begin{pmatrix}
0 & 1 \\
1 & 0 
\end{pmatrix}
\begin{pmatrix} x^T(t)-p_y^T \\ p_x^T(t) \end{pmatrix}
\\
\frac{d}{dt} y^T(t) &= x^T(t) 
\end{align*}
In particular, the dynamics in $(x^T-p_y^T,p_x^T)$ is hyperbolic: setting $v^T(t)=x^T(t)-p_y^T+p_x^T(t)$ and $w^T(t)=x^T(t)-p_y^T-p_x^T(t)$, we have
$$
\dot v^T(t) = v^T(t),\quad \dot w^T(t)=-w^T(t)
$$
with terminal conditions $v^T(0)+w^T(0)=2-2p_y^T$, $v^T(T)+w^T(T)=4-2p_y^T$. 
The solutions of this differential system are given by the familiar phase portrait around a saddle point, and we infer that 
$$
\vert v^T(t)\vert+\vert w^T(t)\vert = \mathrm{O}(e^{-t}+e^{-(T-t)})+\mathrm{O}(\vert p_y^T\vert (e^{-t}+e^{-(T-t)}))
$$
for every $t\in[0,T]$, where the $\mathrm{O}(\cdot)$ is uniform with respect to $T$, 
and thus 
$$
\vert x^T(t)-p_y^T\vert+\vert p_x^T(t)\vert =\mathrm{O}(e^{-t}+e^{-(T-t)})+\mathrm{O}(\vert p_y^T\vert (e^{-t}+e^{-(T-t)})) .
$$
In particular, 
$$
x^T(t) = p_y^T + \mathrm{O}(e^{-t}+e^{-(T-t)})+\mathrm{O}(\vert p_y^T\vert (e^{-t}+e^{-(T-t)}))
$$
and, integrating $\dot y^T(t)=x^T(t)$ we obtain $y^T(t) = p_y^T t + \mathrm{O}(1)+\mathrm{O}(\vert p_y^T\vert)$ for every $t\in[0,T]$. Since $y^T(T)=y_1^T$, we obtain that 
$$
p_y^T = \frac{y_1^T}{T} + \mathrm{O}\left( \frac{1}{T}\right)+\mathrm{O}\left( \frac{\vert p_y^T\vert}{T} \right),
$$
and since $y_1^T$ is sublinear at infinity, we conclude that 
$$
p_y^T 
= \bar p_y^T + \mathrm{O}\left( \frac{1}{T}\right) .
$$
Therefore 
$$
x^T(t) = \frac{y_1^T}{T} + \mathrm{O}\left( \frac{1}{T}+e^{-t}+e^{-(T-t)}\right) = \bar x^T + \mathrm{O}\left( \frac{1}{T}+e^{-t}+e^{-(T-t)}\right)
$$
and
$$
y^T(t)=\frac{y_1^T}{T}t+\mathrm{O}(1)=\bar y^T(t)+\mathrm{O}(1)
$$
for every $t\in[0,T]$. In turn, we have also $p_x^T(t) = \mathrm{O}(e^{-t}+e^{-(T-t)})$.

We conclude that there exists $C>0$ (independent of $T$) such that
\begin{equation}\label{conclusion_simple_example}
\boxed{
\begin{split}
& \vert x^T(t)-\bar x^T\vert\leq C\left( \frac{1}{T}+e^{-t}+e^{-(T-t)}\right) \\
& \vert u^T(t)-\bar u^T\vert + \vert p_x^T(t)-\bar p_x^T\vert \leq C\left( e^{-t}+e^{-(T-t)}\right) \\
& \vert y^T(t) - \bar y^T(t)\vert \leq C \\
& \vert p_y^T-\bar p_y^T\vert \leq \frac{C}{T}
\end{split}
}
\end{equation}
for every $t\in[0,T]$ (with $\bar p_x^T=0$ and $\bar p_y^T=\frac{y_1^T}{T}$).
Comparing with Theorem \ref{thm_TZ}, we see on this example that we do not have an exponential turnpike phenomenon for the component $x^T(t)$: the above inequality is weaker; it says that, except around $t=0$ and $t=T$, $x^T(t)$ is bounded by $\frac{1}{T}$ (which is small as $T$ is large), while the component $y^T(t)$ remains in a bounded neighborhood of $\bar y^T(t)=\frac{y_1^T}{T}$.

It is interesting to note that, in this example, if we take $y_1^T=T\sin(T)$ (as mentioned at the beginning) then $\bar x^T=\sin(T)$ and thus $\bar x^T$ oscillates as $T$ varies and has no limit as $T\rightarrow+\infty$. Despite of that, we have the estimates \eqref{conclusion_simple_example}.

In what follows, we are going to generalize the above example.

\subsubsection{Main result}
\paragraph{Pontryagin maximum principle.}
As in Section \ref{sec_reminders}, we assume that there exists a unique optimal solution $(x^T(\cdot),y^T(\cdot),u^T(\cdot))$ of \eqref{syst1}-\eqref{syst2}-\eqref{terminalconditions_12}-\eqref{syst12_Omega}-\eqref{mincost1}.
Compared with \eqref{syst}-\eqref{terminalconditions}-\eqref{syst_Omega}-\eqref{mincost}, the Hamiltonian of \eqref{syst1}-\eqref{syst2}-\eqref{terminalconditions_12}-\eqref{syst12_Omega}-\eqref{mincost1} is now
$$
H(x,p_x,p_y,p^0,u) = \langle p_x,f(x,u)\rangle + \langle p_y , g(x,u)\rangle + p^0 f^0(x,u) .
$$
Note that $H$ does not depend on the variable $y$.
Assuming as well that, for every $T\geq T_0$, the optimal solution has a unique extremal lift, which is not abnormal (so that we take $p^0=-1$), the application of the Pontryagin maximum principle leads to the existence of an absolutely continuous mapping $p_x^T(\cdot):[0,T]\rightarrow\R^n$ and of a constant $p_y^T\in\R^p$ such that
\begin{equation}\label{extremal_syst1}
\begin{split}
\dot x^T(t) &= \frac{\partial H}{\partial p_x}(x^T(t),p_x^T(t),p_y^T,-1,u^T(t)) = f(x^T(t),u^T(t)) \\
\dot y^T(t) &= \frac{\partial H}{\partial p_y}(x^T(t),p_x^T(t),p_y^T,-1,u^T(t)) = g(x^T(t),u^T(t)) \\
\dot p_x^T(t) &= - \frac{\partial H}{\partial x}(x^T(t),p_x^T(t),p_y^T,-1,u^T(t)) \\
&= - \frac{\partial f}{\partial x}(x^T(t),u^T(t))^\top p_x^T(t) - \frac{\partial g}{\partial x}(x^T(t),u^T(t))^\top p_y^T + \frac{\partial f^0}{\partial x}(x^T(t),u^T(t))^\top 
\end{split}
\end{equation}
($p_y^T$ is constant because $\frac{\partial H}{\partial y}=0$) and
\begin{equation}\label{extremal_syst1_max1}
u^T(t) = \underset{v\in\Omega}{\mathrm{argmax}}\ H(x^T(t),p_x^T(t),p_y^T,-1,v)
\end{equation}
for almost every $t\in[0,T]$. 

Note that \eqref{extremal_syst1_max1} gives
\begin{multline}\label{extremal_syst1_max2}
\frac{\partial H}{\partial u}(x^T(t),p_x^T(t),p_y^T,-1,u^T(t)) \\ 
= \frac{\partial f}{\partial u}(x^T(t),u^T(t))^\top p_x^T(t) + \frac{\partial g}{\partial u}(x^T(t),u^T(t))^\top p_y^T - \frac{\partial f^0}{\partial x}(x^T(t),u^T(t))^\top
= 0
\end{multline}
at every Lebesgue time $t\in[0,T]$ at which $u^T(\cdot)$ is approximately continuous and $u^T(t)\in\mathring{\Omega}$.

\paragraph{``Turnpike-static" optimal control problem.}
With respect to \eqref{staticpb}, the ``turnpike-static" optimal control problem (depending on $T$) is now the following one:
\begin{equation}\label{staticpb1}
\begin{split}
\min \Big\{ f^0(x,u) \ \mid\ (x,u)\in\R^n\times\Omega, \quad & f(x,u)=0 , \\
& \dot y(t)=g(x,u),\quad y(0)=y_0,\quad y(T)=y_1^T \Big\}
\end{split}
\end{equation}
where we have added the constraint coming from the dynamics and terminal conditions in $y$.
In contrast to the static optimal control problem \eqref{staticpb}, which does not depend on $T$, here, we have a \emph{family} of optimization problems indexed by $T$. We call it ``turnpike-static" because $(x,u)$ is still searched as an equilibrium point, while the component $y(\cdot)$ is time-varying.

We assume that, for every $T\geq T_0$, there exists a unique optimal solution $(\bar x^T,\bar y^T(\cdot),\bar u^T)$ of \eqref{staticpb1} (depending on $T$), with $u^T\in\mathring{\Omega}$, and that it has a unique extremal (Lagrange multiplier) lift, which is not abnormal. 
Integrating the differential equation in \eqref{staticpb1}, noting that $(x,u)$ is constant, the turnpike-static optimal control problem \eqref{staticpb1} is equivalent to
\begin{equation}\label{staticpb1_wellposed}
\min \left\{ f^0(x,u) \ \mid\ (x,u)\in\R^n\times\Omega, \quad f(x,u)=0 , \quad  g(x,u) = \frac{y_1^T-y_0}{T} \right\}
\end{equation}
which is a classical minimization problem under $n+p$ equality constraints (note that its optimal solution $(\bar x^T,\bar u^T)$ is assumed to be such that $\bar u^T\in\mathring{\Omega}$ for every $T\geq T_0$).
The Lagrange multiplier rule implies (since the optimal solution is moreover assumed not to be abnormal) that $df^0(\bar x^T,\bar u^T)$ can be expressed linearly in function of $df(\bar x^T,\bar u^T)$ and $dg(\bar x^T,\bar u^T)$.
In the next lemma, we write these classical first-order optimality conditions in an equivalent Hamiltonian form, which seems more complicated but is exactly devised to be comparable to \eqref{extremal_syst1}.

\begin{lemma}\label{lem_static1}
There exist $\bar p_x^T\in\R^n$ and $\bar p_y^T\in\R^p$ such that, for every $t\in[0,T]$,
\begin{equation*}
\begin{split}
0 &= \frac{\partial H}{\partial p_x}(\bar x^T,\bar p_x^T,\bar p_y^T,-1,\bar u^T) = f(\bar x^T,\bar u^T)  \\
\dot{\bar y}(t) &= \frac{\partial H}{\partial p_y}(\bar x^T,\bar p_x^T,\bar p_y^T,-1,\bar u^T) = g(\bar x^T,\bar u^T),\qquad \bar y^T(0)=y_0,\quad \bar y^T(T)=y_1^T \\
0 &= -\frac{\partial H}{\partial x}\left( \bar x^T,\bar p_x^T,\bar p_y^T,-1,\bar u^T\right) = -\frac{\partial f}{\partial x}(\bar x^T,\bar u^T)^\top \bar p_x^T - \frac{\partial g}{\partial x}(\bar x^T,\bar u^T)^\top \bar p_y^T + \frac{\partial f^0}{\partial x}(\bar x^T,\bar u^T) \\
0 &= - \frac{\partial H}{\partial y}(\bar x^T,\bar p_x^T,\bar p_y^T,-1,\bar u^T) 
\end{split}
\end{equation*}
and
\begin{equation*}
\frac{\partial H}{\partial u}\left( \bar x^T,\bar p_x^T,\bar p_y^T,-1,\bar u^T\right) \\
= \frac{\partial f}{\partial u}(\bar x^T,\bar u^T)^\top \bar p_x^T + \frac{\partial g}{\partial u}(\bar x^T,\bar u^T)^\top \bar p_y^T - \frac{\partial f^0}{\partial u}(\bar x^T,\bar u^T)  =  0  .
\end{equation*}
\end{lemma}

This lemma follows by applying the Lagrange multiplier rule to the optimization problem \eqref{staticpb1_wellposed} (see also Lemma \ref{lem_static_asympt} in Section \ref{sec_proof_turnpike_thm}).
The vector $\bar p_x^T$ is the Lagrange multiplier associated with the constraint $f(x,u)=0$, while the vector $\bar p_y^T$ is the Lagrange multiplier associated with the constraint $g(x,u)=\frac{y_1^T-y_0}{T}$.

We observe that
\begin{equation}\label{ybar}
\bar y^T(t)= y_0 + t g(\bar x^T,\bar u^T) \qquad\forall t\in[0,T]
\end{equation}
and therefore 
\begin{equation}\label{y1T}
y_0 + T g(\bar x^T,\bar u^T)=y_1^T .
\end{equation}
When $g(\bar x^T,\bar u^T)\neq 0$, this is only possible if the norm of $y_1^T-y_0$ is large as $T$ is large. In the proof of the turnpike theorem hereafter, we are going to consider $T$ as a parameter, tending to $+\infty$. This is why we have underlined the dependence of $y_1^T$ with respect to $T$: when $T$ becomes larger, it may be required that $y_1^T-y_0$ becomes larger as well (otherwise, the controllability problem may have no solution).

\paragraph{Linear turnpike theorem.}
We use the same notations as before (note anyway that, now, the matrices $A_1^T$,$\bar A_1^T$, $\bar B_1^T$, $\bar U^T$ and $\bar W^T$ depend on $T$ because $(\bar x^T,\bar u^T)$ depends on $T$), and additionally, we set
$$
\bar A_2^T = \bar H_{p_y x} = \frac{\partial g}{\partial x}(\bar x^T,\bar u^T),\qquad
\bar B_2^T = \bar H_{p_y u} = \frac{\partial g}{\partial u}(\bar x^T,\bar u^T),\qquad A_2^T = \bar A_2^T + \bar B_2^T (\bar U^T)^{-1} \bar H_{ux},
$$
$$
A^T = \begin{pmatrix} A_1^T & 0  \\ A_2^T & 0 \end{pmatrix} , \qquad
\bar A^T = \begin{pmatrix} \bar A_1^T & 0  \\ \bar A_2^T & 0 \end{pmatrix} 
= \begin{pmatrix}
\frac{\partial f}{\partial x}(\bar x^T,\bar u^T) & 0  \\
\frac{\partial g}{\partial x}(\bar x^T,\bar u^T) & 0 
\end{pmatrix}, \qquad
\bar B^T = \begin{pmatrix} \bar B_1^T \\ \bar B_2^T \end{pmatrix} 
= \begin{pmatrix}
\frac{\partial f}{\partial u}(\bar x^T,\bar u^T) \\
\frac{\partial g}{\partial u}(\bar x^T,\bar u^T)
\end{pmatrix} .
$$
Here, $A^T$ and $\bar A^T$ are square matrices of size $n+p$ and $\bar B^T$ is a matrix of size $(n+p)\times m$.
In our main result hereafter, we make all assumptions done in Theorem \ref{thm_TZ} and we adapt them in an obvious way to the present context where the turnpike set is not anymore the singleton $\{\bar x\}$, but is now the path $t\mapsto (\bar x^T,\bar y^T(t))$ that we call the \emph{turnpike trajectory}. For instance, the controllability assumption \ref{H1_cont} now stipulates that there exists a trajectory steering in finite time some point of the initial set $M_0$ to some point of the turnpike trajectory, and that any point of the turnpike trajectory can be steered in finite time to the final set $M_1$.

We say that the family of optimal control problems \eqref{syst1}-\eqref{syst2}-\eqref{syst12_Omega}-\eqref{mincost1}, indexed by $T>0$, is \emph{partially $x$-strictly dissipative} at $(\bar x^T, \bar u^T)$ with respect to the supply rate function \eqref{supply}, if the control system \eqref{syst1}, in the $x$ variable, is strictly dissipative with respect to the supply rate function \eqref{supply}. In other words, here, we do not require any dissipativity with respect to $y$.

\begin{theorem}\label{turnpike_thm}
There exist $\delta>0$, $K>0$ and $T_0>0$ such that, for every $T\geq T_0$:\\
-- Assumptions of global nature:
\begin{enumerate}[label=(\roman*)]
\item\label{H2_unique1} There exists a unique optimal tuple $(x^T(\cdot),y^T(\cdot),p_x^T(\cdot),p_y^T,u^T(\cdot))$ solution of \eqref{syst1}-\eqref{syst2}-\eqref{terminalconditions_12}-\eqref{syst12_Omega}-\eqref{mincost1}-\eqref{extremal_syst1}-\eqref{extremal_syst1_max1}-\eqref{transversality} (assumed to be normal); 
moreover, $\Vert x^T(t)\Vert+\Vert u^T(t)\Vert\leq K$ for every $t\in[0,T]$.
\item\label{H2_unique2} There exists a unique optimal tuple $(\bar x^T,\bar y^T(\cdot),\bar p_x^T,\bar p_y^T,\bar u^T)$ solution of \eqref{staticpb1} (with $\bar y^T(\cdot)$ given by \eqref{ybar}) and of the first-order optimality system stated in Lemma \ref{lem_static1}, assumed to be normal; 
moreover, $\Vert \bar x^T\Vert+\Vert \bar u^T\Vert\leq K$ and $\bar u^T\in\Omega_1$ where $\Omega_1$ is a closed subset of $\mathring{\Omega}$ not depending on $T$.
\item\label{H2_dissip} The family of optimal control problems \eqref{syst1}-\eqref{syst2}-\eqref{syst12_Omega}-\eqref{mincost1}, indexed by $T>0$, is partially $x$-strictly dissipative at $(\bar x^T, \bar u^T)$ with respect to the supply rate function $w(x,u)=f^0(x,u)-f^0(\bar x^T,\bar u^T)$, with a storage function $S^T$ (satisfying \eqref{dissip}, with $\alpha$ not depending on $T$) that is bounded on $M_0$ and $M_1$ uniformly with respect to $T$. 
\item\label{H2_cont} There exists a trajectory $(\tilde x^T(\cdot),\tilde y^T(\cdot),\tilde u^T(\cdot))$ solution of \eqref{syst1}-\eqref{syst2} on $[0,T]$, such that $\Vert\tilde x^T(t)\Vert+\Vert\tilde u^T(t)\Vert\leq K$ for every $t\in[0,T]$ and $\tilde x^T(t)=\bar x^T$ and $\tilde u^T(t)=\bar u^T$ for every $t\in[t_1^T,T-t_2^T]$, for some $t_1^T,t_2^T\in[0,T_0]$.
\end{enumerate}
-- Assumptions of local nature:
\begin{enumerate}[label=(\roman*), start=5]
\item\label{H2_U} $\bar U^T=-\bar H_{uu}=-\frac{\partial^2 H}{\partial u^2}(\bar x^T,\bar p_x^T,\bar p_y^T,-1,\bar u^T)$ is a positive definite symmetric matrix, uniformly with respect to $T$, i.e., $\bar U^T\geq \delta I_m$ for every $T\geq T_0$.
\item\label{H2_W} $\bar W^T = -\bar H_{xx} - \bar H_{xu}(\bar U^T)^{-1}\bar H_{ux}$ is a positive definite symmetric matrix, uniformly with respect to $T$, i.e., $\bar W^T\geq \delta I_n$ for every $T\geq T_0$.
\item\label{H2_Kalman} The pair $(\bar A^T,\bar B^T)$ (equivalently, the pair $(A^T,\bar B^T)$) satisfies the Kalman condition in a uniform way with respect to $T$, i.e., denoting $\bar K^T=(\bar B^T,\bar A^T\bar B^T,\ldots,(\bar A^T)^{n+p-1}\bar B^T)$ the Kalman matrix, there exists $\delta>0$ such that $\det ( K^T (K^T)^\top ) \geq \delta$ for every $T\geq T_0$.
\end{enumerate}
In addition, we assume that there exists $C_1>0$ such that:
\begin{enumerate}[label=(\roman*), start=8]
\item\label{H2_y1T} $\Vert y_1^T\Vert\leq C_1T$ for every $T\geq T_0$.
\item\label{inftyqualified} There exists $C_2>C_1$ such that, for every $\theta\in\R^p$ satisfying $\Vert\theta\Vert\leq C_2$,  the set $\{ (x,u)\in\R^n\times\Omega\ \mid\ f(x,u)=0,\quad g(x,u)=\theta\}$ is a ($C^2$) submanifold of $\R^n\times\R^m$.
\end{enumerate} 
Then, there exist $C>0$ and $\nu>0$ such that if $T\geq T_0$ then
\begin{equation}\label{turnpikelin}
\boxed{
\begin{split}
& \Vert x^T(t)-\bar x^T\Vert + \Vert p_x^T(t)-\bar p_x^T\Vert + \Vert u^T(t)-\bar u^T\Vert \leq C \left( \frac{1}{T} + e^{-\nu t} + e^{-\nu (T-t)} \right) \\
& \Vert y^T(t)-\bar y^T(t)\Vert \leq C , \qquad
\Vert p_y^T-\bar p_y^T\Vert\leq \frac{C}{T} 
\end{split}
}
\end{equation}
for every $t\in[0,T]$.
\end{theorem}

The first estimate in \eqref{turnpikelin} states that, except around $t=0$ and $t=T$, the discrepancies $x^T(t)-\bar x^T$, $p_x^T(t)-\bar p_x^T$ and $u^T(t)-\bar u^T$ are bounded above by $\frac{1}{T}$, which is small as $T$ is large, but not exponentially small: it is weaker than \eqref{turnpikeexp}, and we speak of a \emph{linear turnpike} estimate.
The second estimate in \eqref{turnpikelin} gives a bound (uniform with respect to $T$) on the discrepancy $y^T(t)-\bar y^T(t)$ along the whole interval $[0,T]$: it says that $y^T(t)$ remains at a uniform (wrt $T$) distance of $\bar y^T(t) = y_0+t g(\bar x^T,\bar u^T)$ as $t\in[0,T]$.
Finally, the third estimate in \eqref{turnpikelin} says that $p_y^T-\bar p_y^T$ is (linearly) small as $T$ is large.

On specific examples, it may happen that some of the components of the extremal triple $(x^T(\cdot),p_x^T(\cdot),u^T(\cdot))$ enjoy the exponential turnpike property, i.e., an estimate that is stronger than \eqref{turnpikelin}, without the term $\frac{1}{T}$ at the right-hand side. This is the case in the example \eqref{simple_example}: one can see in \eqref{conclusion_simple_example} that the component $x^T$ satisfies the linear turnpike estimate \eqref{turnpikelin}, while the components $p_x^T$ and $u^T$ satisfy the stronger exponential turnpike estimate \eqref{turnpikeexp}. 

\begin{remark}
Under \ref{H2_U}-\ref{H2_W}-\ref{H2_Kalman}-\ref{H2_y1T}, 
Assumptions \ref{H2_unique1}-\ref{H2_unique2}-\ref{H2_dissip}-\ref{H2_cont} and \ref{inftyqualified} are automatically satisfied in the linear-quadratic (LQ) case
\begin{equation}\label{optcont1_LQ}
\begin{split}
& \dot x(t) = \bar A_1^Tx(t)+\bar B_1^Tu(t), \qquad\qquad  x(0)\in M_0, \quad x(T)\in M_1 \\
& \dot y(t) = \bar A_2^Tx(t)+\bar B_2^Tu(t), \qquad\qquad  y(0)=y_0, \quad\ y(T)=y_1^T \\
& \min \int_0^T \left( (x(t)-x_d)^\top \bar W^T (x(t)-x_d)+(u(t)-u_d)^\top \bar U^T (u(t)-u_d)\right) dt 
\end{split}
\end{equation}
for some $x_d\in\R^n$ and $u_d\in\R^m$, with $\Omega=\R^m$.
Actually, \ref{inftyqualified} follows from \ref{H2_Kalman}, because, using the Hautus test, \ref{H2_Kalman} implies that the $(n+p)\times(n+m)$ matrix $(\bar A^T \  \bar B^T )$ is surjective.
\end{remark}

\begin{remark}
As said earlier, to simplify the exposition we have assumed that, for every $T\geq T_0$, the optimal control problem \eqref{syst1}-\eqref{syst2}-\eqref{terminalconditions_12}-\eqref{syst12_Omega}-\eqref{mincost1} has a unique solution, whose extremal lift is unique and normal, and that the static optimal control problem \eqref{staticpb1_wellposed} has a unique solution, whose Lagrange multiplier is unique and normal. When the optimal solution is not unique, or when the extremal lift (still assumed to be normal) is not unique, the statement of Theorem \ref{turnpike_thm} must be adapted as in \cite{TZ}: we obtain in this case a result that is satisfied \emph{locally} around the extremal triple.
\end{remark}

\begin{remark}
As in Remark \ref{rem_state_constraints}, we can add state constraints in the optimal control problem \eqref{syst1}-\eqref{syst2}-\eqref{terminalconditions_12}-\eqref{syst12_Omega}-\eqref{mincost1}. This does not change the statement of Theorem \ref{turnpike_thm}, provided that, in Assumption \ref{H2_unique2}, it is also assumed that $(\bar x^T,\bar y^T(\cdot))$ lies in a closed subset of $\mathring{M}$ not depending on $T$.
\end{remark}

\begin{remark}
Theorem \ref{thm_TZ} (and its proof) is a particular case of Theorem \ref{turnpike_thm}, by ignoring the variables $y$ and $p_y$ (i.e., formally, $p=0$), and the term $\frac{1}{T}$ in the turnpike estimate \eqref{turnpikelin}.
\end{remark}

Theorem \ref{turnpike_thm} is proved in Section \ref{sec_proof_turnpike_thm}.
In Section \ref{sec_examples} we give several optimal control problems that illustrate our main result and we provide numerical simulations showing evidence of the linear turnpike phenomenon.

\section{Proof of Theorem \ref{turnpike_thm}}\label{sec_proof_turnpike_thm}
The proof of Theorem \ref{turnpike_thm} is quite long and involved in some parts. Let us describe shortly the strategy.

First, we prove in Section \ref{sec_prelim_static}, Lemma \ref{lem_static_asympt}, that the family $(\bar x^T,\frac{y_1^T}{T},\bar u^T,\bar p_x^T,\bar p_y^T)_{T\geq T_0}$, solution of the turnpike-static optimal control problem \eqref{staticpb1}, is of class $C^1$ and remains bounded as $T\rightarrow+\infty$. Then, throughout the section, we consider an arbitrary closure point $(\bar x^\infty,\bar d^\infty,\bar u^\infty,\bar p_x^\infty,\bar p_y^\infty)$ of that bounded family and an arbitrary sequence $T_k\rightarrow+\infty$ of positive real numbers such that $(\bar x^{T_k},\frac{y_1^{T_k}}{T_k},\bar u^{T_k},\bar p_x^{T_k},\bar p_y^{T_k})$ converges to $(\bar x^\infty,\bar d^\infty,\bar u^\infty,\bar p_x^\infty,\bar p_y^\infty)$ as $k\rightarrow+\infty$.

In Section \ref{sec_auxiliary}, we study the auxiliary optimal control problem $\mathcal{P}(t_0,t_1,x_0,x_1,\eta_0,\eta_1)$, consisting of the same control system and cost functional as in \eqref{syst1}-\eqref{syst2}-\eqref{terminalconditions_12}-\eqref{syst12_Omega}-\eqref{mincost1}, but on a fixed time interval $[t_0,t_1]$ and with fixed terminal conditions $x(t_0)=x_0$, $y(t_0)=t_0\eta_0$, $x(t_1)=x_1$, $y(t_1)=t_1\eta_1$. Noting in Lemma \ref{leminftydyn} that the trivial trajectory $t\mapsto (x(t),y(t),u(t))=(\bar x^\infty,t\bar d^\infty,\bar u^\infty)$ is an optimal solution of $\mathcal{P}(t_0,t_1,\bar x^\infty,\bar x^\infty,\bar d^\infty,\bar d^\infty)$, with normal extremal lift $(\bar x^\infty,t\bar d^\infty,\bar p_x^\infty,\bar p_y^\infty,-1,\bar u^\infty)$, we then establish in Proposition \ref{prop_sensitivity_dyn} (the main result of the section) that, for any $x_0,x_1$ near $\bar x^\infty$ and for any $\eta_0,\eta_1$ near $\bar d^\infty$, the optimal control problem $\mathcal{P}(t_0,t_1,x_0,x_1,\eta_0,\eta_1)$ has a unique locally optimal solution, with a normal extremal lift $(x(\cdot),y(\cdot),p_x(\cdot),p_y,-1,u(\cdot))$ that is close to $(\bar x^\infty,t\bar d^\infty,\bar p_x^\infty,\bar p_y^\infty,-1,\bar u^\infty)$, and moreover their discrepancy satisfies linear turnpike estimates that prefigure the estimates \eqref{turnpikelin} to be established in Theorem \ref{turnpike_thm}. In some sense, in that section, we pass from static to dynamic.

The proof of Proposition \ref{prop_sensitivity_dyn} itself is quite long and is decomposed in several steps. Since its strategy is described in Section \ref{sec_auxiliary}, we do not reproduce it here. We just mention that the proof relies on sensitivity analysis and conjugate point theory (Section \ref{sec-sentiv}), on the analysis of a dynamical system near an hyperbolic singular point (Section \ref{sec_localturn}), and on the careful analysis of the shooting problem. 
The most delicate point is to obtain a neighborhood that does not depend on $t_0,t_1$ for $t_1-t_0$ large enough: this is done in Section \ref{sec-unifneighb} by solving the shooting problem while exploiting and bootstrapping the estimates obtained until that step.

In Section \ref{sec_exploit_dissip}, we show how to pass from local to global by exploiting the dissipativity property. We first show in Lemma \ref{lem_strict1} that the optimal trajectory $t\mapsto (x^{T_k}(t),\frac{y^{T_k}(t)}{\max(t,1)})$, starting at $t=0$ (resp., starting back in time at $t=T_k$), enters a neighborhood of $(\bar x^\infty,\bar d^\infty)$ within a time that converges to $+\infty$ as $k\rightarrow+\infty$ but can be chosen as a $\mathrm{o}(T_k)$. Thanks to Proposition \ref{prop_sensitivity_dyn} and by local uniqueness, we first derive some estimates that can be bootstrapped to finally show that there exists a finite entering time. The proof of Theorem \ref{turnpike_thm} can then be concluded.

\subsection{Preliminaries on the turnpike-static optimal control problem}\label{sec_prelim_static}
We consider the turnpike-static optimal control problem \eqref{staticpb1}, written in the form of the optimization problem \eqref{staticpb1_wellposed}. 
Recall that, by Assumption \ref{H2_unique2}, there exists a unique minimizer $(\bar x^T,\bar u^T)$ of \eqref{staticpb1_wellposed} and that $\bar p_x^T\in\R^n$ and $\bar p_y^T\in\R^p$ are the (normal) Lagrange multipliers respectively associated to the two constraints (see Lemma \ref{lem_static1}). 
For every $T\geq T_0$, we have (with $df = \begin{pmatrix} \frac{\partial f}{\partial x} & \frac{\partial f}{\partial u} \end{pmatrix}$)
\begin{equation}\label{lagbarT}
\begin{split}
& df(\bar x^T,\bar u^T)^\top\bar p_x^T + dg(\bar x^T,\bar u^T)^\top\bar p_y^T  - df^0(\bar x^T,\bar u^T)^\top = 0 , \\
& f(\bar x^T,\bar u^T) =0, \qquad g(\bar x^T,\bar u^T) = \frac{y_1^T - y_0}{T} .
\end{split}
\end{equation}

\begin{lemma}\label{lem_static_asympt}
For every $T\geq T_0$,
the mapping $T\mapsto(\bar x^T,\bar u^T,\bar p_x^T,\bar p_y^T)$ is of class $C^1$ and
\begin{equation}\label{lem_static_asympt_grandO}
(\bar x^T,\bar u^T,\bar p_x^T,\bar p_y^T) = \mathrm{O}(1)
\end{equation}
as $T\rightarrow +\infty$.
\end{lemma}

\begin{proof}
Let us first prove the $C^1$ regularity property.
The optimality system \eqref{lagbarT} is written as
$$
F(\bar x^T,\bar u^T,\bar p_x^T,\bar p_y^T) = \begin{pmatrix}
df(\bar x^T,\bar u^T)^\top\bar p_x^T + dg(\bar x^T,\bar u^T)^\top\bar p_y^T  - df^0(\bar x^T,\bar u^T)^\top \\
f(\bar x^T,\bar u^T) \\
g(\bar x^T,\bar u^T) 
\end{pmatrix}
= \begin{pmatrix}
0 \\ 0 \\ \frac{y_1^T - y_0}{T}
\end{pmatrix}
$$
which is a system of $2n+m+p$ equations with the $2n+m+p$ unknowns $(x,u,p_x,p_y)$, where $F:\R^{2n+m+p}\rightarrow\R^{2n+m+p}$ is a $C^1$ mapping. The Jacobian of $F$ at the point $(\bar x^T,\bar u^T,\bar p_x^T,\bar p_y^T)$, which is the usual sensitivity matrix in optimization, is
$$
dF(\bar x^T,\bar u^T,\bar p_x^T,\bar p_{y_1}^T,\bar p_{y_2}^T)
= \left( \begin{array}{cc|cc}
\bar H_{xx} & \bar H_{xu} & \bar H_{xp_x} & \bar H_{xp_y} \\
\bar H_{ux} & \bar H_{uu} & \bar H_{up_x} & \bar H_{up_y} \\ \hline
\bar H_{p_xx} & \bar H_{p_xu} & 0 & 0 \\
\bar H_{p_yx} & \bar H_{p_yu} & 0 & 0 \\
\end{array} \right)
= \left( \begin{array}{c|c}
E_1 & E_2^\top \\ \hline
E_2 & 0 \\
\end{array} \right) .
$$
Noting that 
$$
E_2 = \begin{pmatrix}
\frac{\partial f}{\partial x}(\bar x^T,\bar u^T) & \frac{\partial f}{\partial u}(\bar x^T,\bar u^T) \\[1mm]
\frac{\partial g}{\partial x}(\bar x^T,\bar u^T) & \frac{\partial g}{\partial u}(\bar x^T,\bar u^T)
\end{pmatrix} = \begin{pmatrix}
\bar A^T & \bar B^T
\end{pmatrix} ,
$$
Using the Hautus test, Assumption \ref{H2_Kalman}, which says in particular that the pair $(\bar A^T,\bar B^T)$ satisfies the Kalman assumption, implies that the $(n+p)\times(n+p+m)$ matrix $E_2$ is surjective. 

Besides, by the Schur complement lemma (see, e.g., \cite[Appendix A.5.5]{BoydVandenberghe}), the symmetric matrix $E_1$ is negative definite, because by Assumption \ref{H2_U} the matrix $\bar H_{uu}=-\bar U^T$ is negative definite and by Assumption \ref{H2_W} the matrix $\bar H_{xx}-\bar H_{xu} \bar H_{uu}^{-1} \bar H_{ux}=-\bar W^T$ (Schur complement) is negative definite.

Now, since $E_1$ is symmetric negative definite and $E_2$ is surjective, we conclude that the sensitivity matrix $dF(\bar x^T,\bar u^T,\bar p_x^T,\bar p_y^T)$ is invertible.\footnote{Indeed, let $(Y_1,Y_2)\in\R^{n+m}\times\R^{n+p}$ be such that $dF.\begin{pmatrix}Y_1\\Y_2\end{pmatrix}=\begin{pmatrix}0\\0\end{pmatrix}$, i.e., $E_1Y_1+E_2^\top Y_2=0$ and $E_2Y_1=0$. Then $0 = Y_1^\top E_1 Y_1+Y_1^\top E_2^\top Y_2 = Y_1^\top E_1 Y_1$ because $Y_1^\top E_2^\top=0$, and thus $Y_1=0$ because $E_1$ is negative definite. Then, $E_2^\top Y_2=0$ implies $Y_2=0$ because $E_2^\top$ is injective.}
By Assumption \ref{H2_y1T}, we have $\frac{y_1^T - y_0}{T}   = \mathrm{O}(1)$ as $T\rightarrow+\infty$. Hence, we have to solve the nonlinear system of equations $F(\bar x^T,\bar u^T,\bar p_x^T,\bar p_y^T) = \mathrm{O}(1)$. The $C^1$ regularity property follows by applying the implicit function theorem.

\medskip

Let us now establish \eqref{lem_static_asympt_grandO}.
We already know that $(\bar x^T,\bar u^T) = \mathrm{O}(1)$ as $T\rightarrow +\infty$ by the $K$-boundedness property in Assumption \ref{H2_unique2}. Hence, it remains to prove that $\bar p_x^T$ and $\bar p_y^T$ are uniformly bounded as $T\rightarrow+\infty$.
By contradiction, let us assume that there exists a sequence $T_k\rightarrow+\infty$ such that $\Vert\bar p_x^{T_k}\Vert+\Vert\bar p_y^{T_k}\Vert\rightarrow+\infty$ as $k\rightarrow+\infty$.

First, let us show how to pass to the limit in \eqref{lagbarT}.
Since the sequences $\bar x^{T_k}$, $\bar u^{T_k}$, $\bar p_x^{T_k}/(\Vert\bar p_x^{T_k}\Vert+\Vert\bar p_y^{T_k}\Vert)$, $\bar p_y^{T_k}/(\Vert\bar p_x^{T_k}\Vert+\Vert\bar p_y^{T_k}\Vert)$ and $y_1^T/T$ are bounded (the latter, by  Assumption \ref{H2_y1T}), taking a subsequence if necessary, we have
$$
\bar x^{T_k}\rightarrow\bar x^\infty, \quad
\bar u^{T_k}\rightarrow\bar u^\infty, \quad
\frac{\bar p_x^{T_k}}{\Vert\bar p_x^{T_k}\Vert+\Vert\bar p_y^{T_k}\Vert} \rightarrow \psi_x^\infty, \quad
\frac{\bar p_y^{T_k}}{\Vert\bar p_x^{T_k}\Vert+\Vert\bar p_y^{T_k}\Vert} \rightarrow \psi_y^\infty, \quad
\frac{y_1^{T_k}}{T_k}\rightarrow \bar d^\infty
$$ 
as $k\rightarrow+\infty$, for some $\bar x^\infty\in\R^n$, $\bar u^\infty\in\R^m$, $(\psi_x^\infty,\psi_y^\infty)\in\R^n\times\R^n\setminus\{(0,0)\}$ (the nontriviality of the latter being crucial in the reasoning) and $\bar d^\infty\in\R^p$ satisfying $\Vert \bar d^\infty\Vert\leq C_1$. Moreover, by Assumption \ref{H2_unique2}, $\bar u^{T_k}\in\Omega_1$ where $\Omega_1\subset\mathring{\Omega}$ is closed, hence $\bar u^\infty\in\Omega_1$ and thus $\bar u^\infty\in\mathring{\Omega}$. 
Considering \eqref{lagbarT} with $T=T_k$ and dividing the first line by $\Vert\bar p_x^{T_k}\Vert+\Vert\bar p_y^{T_k}\Vert$, which converges to $+\infty$, and then taking the limit $k\rightarrow+\infty$ yields
\begin{equation}\label{lagpsiinfty}
\begin{split}
& df(\bar x^\infty,\bar u^\infty)^\top\psi_x^\infty + dg(\bar x^\infty,\bar u^\infty)^\top\psi_y^\infty   = 0 , \\
& f(\bar x^\infty,\bar u^\infty) =0, \qquad g(\bar x^\infty,\bar u^\infty) = \bar d^\infty .
\end{split}
\end{equation}

Second, let us prove that $(\bar x^\infty,\bar u^\infty)$ is an optimal solution of the optimization problem 
\begin{equation}\label{staticinfty}
\min \left\{ f^0(x,u) \ \mid\ (x,u)\in\R^n\times\Omega, \quad f(x,u)=0 , \quad  g(x,u) = \bar d^\infty \right\} .
\end{equation}
By definition, we have
$$
f^0(\bar x^{T_k},\bar u^{T_k}) \leq f^0(x,u)\qquad \forall (x,u)\ \mid\ f(x,u)=0\quad\textrm{and}\quad g(x,u)=\frac{y_1^{T_k}-y_0}{T_k}\qquad \forall k\in\N .
$$
The fact follows by passing to the limit and by using the fact that, by Assumption \ref{inftyqualified}, the set $\{ (x,u)\in\R^n\times\Omega\ \mid\ f(x,u)=0,\quad g(x,u)=\bar d^\infty\}$ is a submanifold of $\R^n\times\R^m$: indeed, the latter assumption entails the fact that, at every point of that set, the mapping $(df(x,u),dg(x,u))$ is surjective, which allows us to use the implicit function theorem (more precisely, the tubular neighborhood theorem). 

Let us now conclude by raising a contradiction. By Assumption \ref{inftyqualified}, the optimization problem \eqref{staticinfty} is qualified and has no abnormal minimizer. But we have proved that $(\bar x^\infty,\bar u^\infty)$ is a minimizer of \eqref{staticinfty}, which has a nontrivial abnormal lift by \eqref{lagpsiinfty} (recall that $(\psi_x^\infty,\psi_y^\infty)\neq (0,0)$). This is a contradiction.
\end{proof}

\begin{remark}
In the LQ case \eqref{optcont1_LQ}, the above proof is simpler and does not require to use the implicit function theorem. 
\end{remark}

\begin{remark}
As noticed in the above proof, by Assumption \ref{H2_y1T}, the family $\big(\frac{y_1^T}{T}\big)_{T\geq T_0}$ is bounded. Therefore, by Lemma \ref{lem_static_asympt}, the family $(\bar x^T,\frac{y_1^T}{T},\bar u^T,\bar p_x^T,\bar p_y^T)_{T\geq T_0}$ is bounded. It may however fail to have a limit as $T\rightarrow+\infty$. For instance, in the motivating example \eqref{simple_example}, if we choose $y_1^T=T\sin(T)$ then we have $\bar x^T=\bar p_y^T=\sin(T)$, which do not have any limit as $T\rightarrow+\infty$. 

Hereafter, throughout all the proof (which is quite long) of Theorem \ref{turnpike_thm}, we fix a limit point (closure point) of that family and a converging subsequence $(T_k)_{k\in\N^*}$. This gives a ``limit turnpike-static" optimal control problem.
\end{remark}

\paragraph{Limit turnpike-static optimal control problem.}
By Lemma \ref{lem_static_asympt} and Assumption \ref{H2_y1T}, the family $(\bar x^T,\frac{y_1^T}{T},\bar u^T,\bar p_x^T,\bar p_y^T)_{T\geq T_0}$ is bounded as $T\rightarrow+\infty$.
Let $(\bar x^\infty,\bar d^\infty,\bar u^\infty,\bar p_x^\infty,\bar p_y^\infty)$ and let $(T_k)_{k\in\N^*}$ be a sequence of positive real numbers converging to $+\infty$ such that
\begin{equation}\label{limit_tuple}
\bar x^{T_k} \rightarrow\bar x^\infty, \qquad \frac{y_1^{T_k}}{T_k}\rightarrow\bar d^\infty, \qquad \bar u^{T_k}\rightarrow\bar u^\infty, \qquad \bar p_x^{T_k} \rightarrow\bar p_x^\infty, \qquad \bar p_y^{T_k} \rightarrow\bar p_y^\infty
\end{equation}
as $k\rightarrow+\infty$. 
Note that $\Vert\bar d^\infty\Vert\leq C_1$ and $\bar u^\infty\in\Omega_1$ (recall that $\Omega_1$ is a closed subset of $\Omega$).
We infer from \eqref{lagbarT} that
\begin{equation}\label{lagbarinfty}
\begin{split}
& df(\bar x^\infty,\bar u^\infty)^\top\bar p_x^\infty + dg(\bar x^\infty,\bar u^\infty)^\top\bar p_y^\infty  - df^0(\bar x^\infty,\bar u^\infty)^\top = 0 , \\
& f(\bar x^\infty,\bar u^\infty) =0, \qquad g(\bar x^\infty,\bar u^\infty) = \bar d^\infty .
\end{split}
\end{equation}

\begin{lemma}\label{lemoptinfty}
The pair $(\bar x^\infty,\bar u^\infty)$ is an optimal solution of the so-called ``limit turnpike-static" problem
\begin{equation}\label{limitstaticpb}
\min \left\{ f^0(x,u) \ \mid\ (x,u)\in\R^n\times\Omega, \quad f(x,u)=0 , \quad  g(x,u) = \bar d^\infty \right\}
\end{equation}
and its unique Lagrange multiplier is normal and is $(\bar p_x^\infty,\bar p_y^\infty)$.
\end{lemma}

\begin{proof}
We set $N^\infty = \{(x,u)\in\R^n\times\Omega\ \mid\ f(x,u)=0 , \quad  g(x,u) = \bar d^\infty\}$ and $N^T = \{(x,u)\in\R^n\times\Omega\ \mid\ f(x,u)=0 , \quad  g(x,u) = \frac{y_1^T-y_0}{T}\}$. 
Let $(x,u)\in N^\infty$. For every $k$, by Assumption \ref{inftyqualified} and by the tubular neighborhood theorem, let $(x^k,u^k)\in N^{T_k}$ such that $(x^k,u^k)\rightarrow (x,u)$ as $k\rightarrow +\infty$. We have $f^0(\bar x^{T_k},\bar u^{T_k}) \leq f^0(x^k,u^k)$ because $(\bar x^{T_k},\bar u^{T_k})$ minimizes $f^0$ on $N^{T_k}$, and then, taking the limit, we obtain that $f^0(\bar x^\infty,\bar u^\infty)\leq f^0(x,u)$. Hence, $(\bar x^\infty,\bar u^\infty)$ is an optimal solution of \eqref{limitstaticpb}. 

By Assumption \ref{inftyqualified}, $(\bar x^\infty,\bar u^\infty)$ has a unique extremal lift, which is normal. 
By \eqref{lagbarinfty}, the Lagrange multiplier must be $(\bar p_x^\infty,\bar p_y^\infty)$.
\end{proof}

We denote by $A^\infty$, $\bar A^\infty$, $\bar B^\infty$, $\bar U^\infty$ and $\bar W^\infty$ the limits, respectively, of $A^{T_k}$, $\bar A^{T_k}$, $\bar B^{T_k}$, $\bar U^{T_k}$ and $\bar W^{T_k}$ as $k\rightarrow+\infty$.
By taking the limit $T_k\rightarrow+\infty$ in Assumptions \ref{H2_U}, \ref{H2_W} and \ref{H2_Kalman}, we have $\bar U^\infty\geq\delta I_m$, $\bar W^\infty\geq\delta I_n$, and the pair $(\bar A^\infty,\bar B^\infty)$ (equivalently, $(A^\infty,B^\infty)$) satisfies the Kalman condition.

\begin{lemma}\label{CV_rate_static}
The convergence estimates in \eqref{limit_tuple} are in $\frac{1}{T_k}$, i.e., there exists $C>0$ such that $\Vert\bar x^{T_k} -\bar x^\infty\Vert\leq\frac{C}{T_k}$, and the same for the others, for every $k\in\N^*$.
\end{lemma}

\begin{proof}
The proof is similar to the first part of the proof of Lemma \ref{lem_static_asympt}: it suffices to replace the optimality system \eqref{lagbarT} with the limit one \eqref{lagbarinfty}. Then all arguments work exactly in the same way (also, replace $\frac{y_1^T-y_0}{T}$ with the limit $\bar d^\infty$), since the limit pair $(\bar A^\infty,\bar B^\infty)$ satisfies the Kalman assumption, $\bar W^\infty$ is positive definite
\end{proof}

\subsection{Auxiliary optimal control problem}\label{sec_auxiliary}
Given any $t_0,t_1\in[0,+\infty)$ such that $t_0<t_1$, any $x_0,x_1\in\R^n$ and any $\eta\in\R^p$, we consider the auxiliary optimal control problem
\begin{equation*}
\mathcal{P}(t_0,t_1,x_0,x_1,\eta_0,\eta_1)\qquad 
\left\{\begin{split}
& \dot x(t) = f(x(t),u(t)), \qquad  x(t_0) = x_0, \quad x(t_1) = x_1  \\
& \dot y(t) = g(x(t),u(t)), \qquad\, y(t_0)=t_0\eta_0, \quad\ \, y(t_1)= t_1\eta_1 \\
& u(t)\in\Omega \\
& \min \int_{t_0}^{t_1} f^0(x(t),u(t))\, dt 
\end{split} \right. 
\end{equation*}
According to the Pontryagin maximum principle, if $(x(\cdot),y(\cdot),u(\cdot))$ is a locally optimal solution of $\mathcal{P}(t_0,t_1,x_0,x_1,\eta_0,\eta_1)$ having at least one normal extremal lift, then there exist $p_x(\cdot):[t_0,t_1]\rightarrow\R^n$ absolutely continuous and $p_y\in\R^p$ such that
\begin{equation}\label{extrem}
\begin{split}
\dot x(t) &= f(x(t),u(t)) \\
\dot y(t) &= g(x(t),u(t)) \\
\dot p_x(t) &= - \frac{\partial f}{\partial x}(x(t),u(t))^\top p_x(t) - \frac{\partial g}{\partial x}(x(t),u(t))^\top p_y + \frac{\partial f^0}{\partial x}(x(t),u(t))^\top 
\end{split}
\end{equation}
where $u(t)$ is solution of
\begin{equation}\label{extremu}
\frac{\partial f}{\partial u}(x(t),u(t))^\top p_x(t) + \frac{\partial g}{\partial u}(x(t),u(t))^\top p_y - \frac{\partial f^0}{\partial x}(x(t),u(t))^\top = 0
\end{equation}
almost everywhere on $[t_0,t_1]$.

\begin{lemma}\label{leminftydyn}
Taking $x_0=x_1=\bar x^\infty$ and $\eta_0=\eta_1=\bar d^\infty$, for any $t_0,t_1\in[0,+\infty)$ such that $t_0<t_1$, the trajectory $t\mapsto (x(t),y(t),u(t))=(\bar x^\infty,t\bar d^\infty,\bar u^\infty)$ is an optimal solution of the optimal control problem $\mathcal{P}(t_0,t_1,\bar x^\infty,\bar x^\infty,\bar d^\infty,\bar d^\infty)$. Moreover, this optimal solution has a unique extremal lift which is normal, the adjoint vectors corresponding to $x$ and $y$ being constant, equal to $\bar p_x^\infty$, $\bar p_y^\infty$ respectively.
\end{lemma}

\begin{proof}
By Assumption \ref{H2_dissip}, the partial $x$-dissipativity inequality \eqref{dissip} implies that, for any $x(\cdot)$ solution of $\dot x(t)=f(x(t),u(t))$ on $[t_0,t_1]$ such that $x(t_0)=x(t_1)$, since $S^{T_k}(x(t_0)) = S^{T_k}(x(t_1))$ we get that $\int_{t_0}^{t_1} f^0(\bar x^{T_k},\bar u^{T_k})\, dt \leq \int_{t_0}^{t_1} f^0(x(t),u(t))\, dt$, and letting $k$ tend to $+\infty$ we obtain that $\int_{t_0}^{t_1} f^0(\bar x^\infty,\bar u^\infty)\, dt \leq \int_{t_0}^{t_1} f^0(x(t),u(t))\, dt$. Therefore $(x(t),y(t),u(t))=(\bar x^\infty,t\bar d^\infty,\bar u^\infty)$ is an optimal solution of $\mathcal{P}(t_0,t_1,\bar x^\infty,\bar x^\infty,\bar d^\infty,\bar d^\infty)$. The rest of the statement follows from \eqref{lagbarinfty}. 
\end{proof}

\begin{remark}\label{rem_statictodynamic}
Lemma \ref{leminftydyn} shows that $(x(t),y(t),u(t))=(\bar x^\infty,t\bar d^\infty,\bar u^\infty)$ is not only solution of the limit turnpike-static problem, but also of the \emph{dynamical} optimal control problem $\mathcal{P}(t_0,t_1,\bar x^\infty,\bar x^\infty,\bar d^\infty,\bar d^\infty)$. This stronger fact is due to the dissipativity assumption. This is an important step because here, in some sense, we pass from static to dynamic.

Note that $t\mapsto(\bar x^\infty,t\bar d^\infty,\bar u^\infty)$ is a globally optimal solution of $\mathcal{P}(t_0,t_1,\bar x^\infty,\bar x^\infty,\bar d^\infty,\bar d^\infty)$. In the next lemma, we show that the property that the trajectory $t\mapsto(\bar x^\infty,t\bar d^\infty,\bar u^\infty)$ is a locally optimal solution having a normal extremal lift is robust near $\bar x^\infty$ and $\bar d^\infty$, for any $0\leq t_0<t_1$.
\end{remark}

The main result of that section is the following.

\begin{proposition}\label{prop_sensitivity_dyn}
There exist $\varepsilon_0>0$, $T_0>0$, $C>0$ and $\nu>0$ such that, for any $t_0,t_1\in[0,+\infty)$ such that $t_1-t_0\geq T_0$, for any $\varepsilon\in(0,\varepsilon_0)$, for any $x_0,x_1\in B(\bar x^\infty,\varepsilon)$ and any $\eta_0,\eta_1\in B(\bar d^\infty,\varepsilon)$: 
\begin{enumerate}[label=(\roman*)]
\item\label{casei} the optimal control problem $\mathcal{P}(t_0,t_1,x_0,x_1,\eta_0,\eta_1)$ has a locally optimal solution $(x(\cdot),y(\cdot),u(\cdot))$ (for the $L^\infty$ topology on $u$) that has a normal extremal lift $(x(\cdot),y(\cdot),p_x(\cdot),p_y,-1,u(\cdot))$, with $u$ being $C^1$; \\
moreover, this optimal solution is locally unique in the sense that there does not exist any other control $\tilde u\in L^\infty([t_0,t_1],\mathcal{U})$ whose corresponding trajectory would also be a locally optimal solution of $\mathcal{P}(t_0,t_1,x_0,x_1,\eta_0,\eta_1)$;
\item\label{caseii} $t\mapsto(x(t),\frac{y(t)}{\max(1,t)},p_x(t),p_y,u(t))$ converges to $(\bar x^\infty, \bar d^\infty, \bar p_x^\infty, \bar p_y^\infty,\bar u^\infty)$ in $C^0$ topology on $[t_0,t_1]$ as $\varepsilon\rightarrow 0$; \\
(we divide $y(t)$ by $\max(1,t)$ to avoid division by $0$)
\item\label{caseiii} 
we have the estimates
\begin{equation}\label{linear_turnpike_estimates_prop}
\begin{split}
& \Vert x(t)-\bar x^\infty\Vert + \Vert p_x(t)-\bar p_x^\infty\Vert + \Vert u(t)-\bar u^\infty\Vert \\
& \qquad \leq C\varepsilon \left( e^{-\nu (t-t_0)} + e^{-\nu(t_1-t)} \right) + \frac{C}{t_1-t_0} \left( \varepsilon + \Vert y(t_1)-y(t_0)-(t_1-t_0)\bar d^\infty\Vert \right)  , \\
& \Vert y(t)-t\bar d^\infty\Vert \leq C \left( \varepsilon + \Vert y(t_0)-t_0\bar d^\infty\Vert + \Vert y(t_1)-y(t_0)-(t_1-t_0)\bar d^\infty\Vert \right), \\
& \Vert p_y-\bar p_y^\infty\Vert \leq \frac{C}{t_1-t_0} \left( \varepsilon + \Vert y(t_1)-y(t_0)-(t_1-t_0)\bar d^\infty\Vert \right)  ,
\end{split}
\end{equation}
for every $t\in[t_0,t_1]$.
\end{enumerate}
\end{proposition}

The rest of the section is devoted to proving Proposition \ref{prop_sensitivity_dyn}. The proof is quite lengthy. We use sensitivity analysis, conjugate point theory, analysis of a dynamical system near a hyperbolic singular point, and a number of bootstrap arguments aiming at improving, step by step, the estimates. The successive bootstrap arguments ultimately lead to \eqref{linear_turnpike_estimates_prop}.
A major difficulty in Proposition \ref{prop_sensitivity_dyn} is to prove the existence of $\varepsilon_0$ not depending on $t_0,t_1$ for $t_1-t_0$ large enough.

More precisely, in Section \ref{sec-sentiv}, we establish Lemma \ref{lem_sensitivity}, which gives \ref{casei} (except the local uniqueness property) and \ref{caseii} of Proposition \ref{prop_sensitivity_dyn} but with $\varepsilon_0$ depending on $t_0$ and $t_1$. The proof of the existence of $\varepsilon_0$ not depending on $t_0,t_1$ will be done further, in Section \ref{sec-unifneighb}: it requires other results and several bootstrap arguments.

In-between, in Section \ref{sec_localturn}, we study the extremal system \eqref{extrem}-\eqref{extremu} locally around the trivial trajectory $t\mapsto(\bar x^\infty,t\bar d^\infty,\bar p_x^\infty,\bar p_y^\infty)$. We will call ``linearized system", the linearization of the extremal system \eqref{extrem}-\eqref{extremu} along this trivial trajectory. The main result in Section \ref{sec_localturn} is Lemma \ref{lem_local_turnpike}, which gives \ref{caseiii}, i.e., the linear turnpike estimates \eqref{linear_turnpike_estimates_prop}. 

Sections \ref{sec_proof_lem_expest} and \ref{sec_proof_lem_R} are devoted to proving two technical lemmas, used in the proof of Lemma \ref{lem_local_turnpike}.

Finally, in Section \ref{sec-unifneighb}, we conclude the proof of Proposition \ref{prop_sensitivity_dyn}, by bootstrapping Lemmas \ref{lem_sensitivity} and \ref{lem_local_turnpike} and finally establish the existence of $\varepsilon_0$ that is uniform with respect to $t_0$ and $t_1$ whenever $t_1-t_0$ is large enough, and obtaining the local uniqueness property.

\subsubsection{Sensitivity analysis}\label{sec-sentiv}
In this section, we establish the following lemma, which implies \ref{casei} (except the local uniqueness property) and \ref{caseii} of Proposition \ref{prop_sensitivity_dyn} but, for the moment, with $\varepsilon_0$ depending on $t_0$ and $t_1$. 

\begin{lemma}\label{lem_sensitivity}
There exist $T_0>0$ such that, for any $t_0,t_1\in[0,+\infty)$ satisfying $t_1-t_0\geq T_0$, there exist $\varepsilon_0>0$ (depending on $t_0$ and $t_1$) such that, for any $\varepsilon\in(0,\varepsilon_0)$, $x_0,x_1\in B(\bar x^\infty,\varepsilon)$ and $\eta_0,\eta_1\in B(\bar d^\infty,\varepsilon)$, the optimal control problem $\mathcal{P}(t_0,t_1,x_0,x_1,\eta_0,\eta_1)$ has a locally optimal solution $(x(\cdot),y(\cdot),u(\cdot))$ that has a normal extremal lift $(x(\cdot),y(\cdot),p_x(\cdot),p_y,-1,u(\cdot))$, with $u$ being $C^1$. \\
Moreover, $t\mapsto(x(t),\frac{y(t)}{\max(1,t)},p_x(t),p_y,u(t))$ converges to $(\bar x^\infty, \bar d^\infty, \bar p_x^\infty, \bar p_y^\infty,\bar u^\infty)$ in $C^0$ topology on $[t_0,t_1]$ as $\varepsilon\rightarrow 0$ converges to $(\bar x^\infty,\bar x^\infty,\bar d^\infty,\bar d^\infty)$.
\end{lemma}

\begin{proof}
The proof goes by classical sensitivity analysis arguments.
We denote by $E:[0,+\infty)^2\times\R^n\times\R^p\times L^\infty([t_0,t_1],\Omega)\rightarrow\R^{n+p}$ the end-point mapping, defined by $E(t_0,t_1,x_0,\eta_0,u)=(x(t_1),y(t_1))$ where $(x(\cdot),y(\cdot))$ is the solution on $[t_0,t_1]$ of $\dot x(t)=f(x(t),u(t))$, $\dot y(t)=g(x(t),u(t))$ corresponding to the control $u$ and such that $x(t_0)=x_0$ and $y(t_0)=t_0\bar d^\infty$, and we denote by $C(t_0,t_1,x_0,\eta_0,u)=\int_{t_0}^{t_1} f^0(x(t),u(t))\, dt$ the associated cost. The optimal control problem $\mathcal{P}(t_0,t_1,x_0,x_1,\eta_0,\eta_1)$ is then equivalent to minimize $C(t_0,t_1,x_0,\eta_0,u)$ over all possible controls $u\in L^\infty([t_0,t_1],\Omega)$ such that $E(t_0,t_1,x_0,\eta_0,u)=(x_1,t_1\eta_1)$.

For $0<t_0<t_1$ fixed, if $(x(\cdot),y(\cdot),u(\cdot))$ is an optimal solution of $\mathcal{P}(t_0,t_1,x_0,x_1,\eta_0,\eta_1)$ on $[t_0,t_1]$, such that $u$ takes its values in $\Omega_1\subsetneq\Omega$ and is not abnormal, then according to the Lagrange multiplier rule there must exist $\psi\in\R^{n+p}$ such that 
\begin{multline}\label{lagt1}
\frac{\partial C}{\partial u}(t_0,t_1,x_0,\eta_0,u) + \big\langle\psi,\frac{\partial E}{\partial u}(t_0,t_1,x_0,\eta_0,u)\big\rangle = 0 , \\
\textrm{i.e.},\quad \frac{\partial L}{\partial u}(t_0,t_1,x_0,\eta_0,x_1,\eta_1,u,\psi)=0 ,
\end{multline}
(this equality is in the dual of $L^1([t_0,t_1],\R^m)$ that is identified with $L^\infty([t_0,t_1],\R^m)$)
where
$$
L(t_0,t_1,x_0,\eta_0,x_1,\eta_1,u,\psi)=C(t_0,t_1,x_0,\eta_0,u)+\langle\psi,E(t_0,t_1,x_0,\eta_0,u)-(x_1,t_1\eta_1)\rangle
$$
is the Lagrangian of the optimization problem. Therefore
$$
G(t_0,t_1,x_0,\eta_0,x_1,\eta_1,u,\psi) = \begin{pmatrix}
\frac{\partial L}{\partial u}(t_0,t_1,x_0,\eta_0,x_1,\eta_1,u,\psi) \\[1mm] E(t_0,t_1,x_0,\eta_0,u)-(x_1,t_1\eta_1)
\end{pmatrix} = 0 
$$
where $G: [0,+\infty)^2\times(\R^{n+p})^2\times L^\infty([t_0,t_1],\R^m)\times\R^{n+p} \rightarrow L^\infty([t_0,t_1],\R^m)\times\R^{n+p}$ is a $C^1$ mapping. 
The first-order optimality condition \eqref{lagt1} is the preliminary equation ultimately leading to the ``weak" Pontryagin maximum principle (see \cite{trelat_JDCS2000,trelat_book,Trelat_JOTA2012}); this means that, equivalently, there exist $p_x(\cdot):[0,t_1]\rightarrow\R^n$ and $p_y\in\R^p$ such that the extremal equations \eqref{extrem} and \eqref{extremu} are satisfied almost everywhere on $[t_0,t_1]$.
Moreover, we have $-\psi=(p_x(t_1),p_y)$, i.e., the Lagrange multiplier $-\psi$ coincides with the adjoint vector at the final time. 
Note, since $u$ takes its values in $\Omega_1$, we indeed have $\frac{\partial H}{\partial u}=0$ along the extremal (which is \eqref{extremu}), and when $(x(t),p_x(t),p_y,u(t))$ is close enough to $(\bar x^\infty,\bar p^\infty,\bar p^\infty,\bar u^\infty)$ (what we are going to ensure next), thanks to Assumption \ref{H2_U} one can solve this equation by the implicit function theorem and get $u$ as a $C^1$ function of $(x(t),p_x(t),p_y)$. 

Note that we have not proved yet that $\mathcal{P}(t_0,t_1,x_0,x_1,\eta_0,\eta_1)$ has an optimal solution. The equations \eqref{extrem}-\eqref{extremu} constitute the first-order optimality system for a locally optimal solution, if it exists and if it has a normal extremal lift. Hereafter, we are going to prove that, locally near the trivial solution $t\mapsto(\bar x^\infty,t\bar d^\infty,\bar p_x^\infty,\bar p_y^\infty,\bar u^\infty)$ of \eqref{extrem}-\eqref{extremu}, for all $0\leq t_0<t_1$ there exists a solution of the extremal system \eqref{extrem}-\eqref{extremu} such that $x(t_0)=x_0$, $x(t_1)=x_1$, $y(t_0)=t_0\eta_0$ and $y(t_1)=t_1\eta_1$, with $u$ of class $C^1$ as above, and that such a solution is indeed a locally optimal solution of $\mathcal{P}(t_0,t_1,x_0,x_1,\eta_0,\eta_1)$. This is done, hereafter, by sensitivity analysis and by conjugate point theory. 

By Lemma \ref{leminftydyn}, given any $0\leq t_0<t_1$ we have 
$$
G(t_0,t_1,\bar x^\infty,\bar d^\infty,\bar x^\infty,\bar d^\infty,\bar u^\infty,(\bar p_x^\infty,\bar p_y^\infty))=0 .
$$
Now, for any $0\leq t_0<t_1$ fixed and any $(x_0,x_1,\eta_0,\eta_1)$ sufficiently close to $(\bar x^\infty,\bar x^\infty,\bar d^\infty,\bar d^\infty)$, we can solve with respect to $(u,\psi)$ the equation $G(t_0,t_1,x_0,\eta_0,x_1,\eta_1,u,\psi) = 0$: this is possible thanks to the implicit function theorem, because the sensitivity matrix (of operators)
$$
\frac{\partial G}{\partial(u,\psi)} = \begin{pmatrix}
\frac{\partial^2 L}{\partial u^2} & \frac{\partial E}{\partial u}^* \\[1mm]
\frac{\partial E}{\partial u} & 0
\end{pmatrix}
$$
is boundedly invertible at the point $(t_0,t_1,\bar x^\infty,\bar d^\infty,\bar x^\infty,\bar d^\infty,\bar u^\infty,(\bar p_x^\infty,\bar p_y^\infty))$ (the argument here is quite similar to the one used in the proof of Lemma \ref{lem_static_asympt}, which is not a surprise in view of Remark \ref{rem_statictodynamic}: we pass from static to dynamic). This is so, because of the two following facts:
\begin{itemize}
\item The Fr\'echet differential 
$$
\frac{\partial E}{\partial u}(t_0,t_1,\bar x^\infty,\bar d^\infty,\bar u^\infty) : L^\infty([t_0,t_1],\R^m)\rightarrow\R^{n+p}
$$
is surjective: indeed, equivalently (see, e.g., \cite{BonnardChyba,trelat_book}), the linearized control system at $(\bar x^\infty,\bar u^\infty)$ is controllable, because the pair $(A^\infty,\bar B^\infty)$ (defined at the end of Section \ref{sec_prelim_static}) satisfies the Kalman condition.
\item The quadratic form
$$
Q_{t_0,t_1} = \frac{\partial^2 L}{\partial u^2}(t_0,t_1,\bar x^\infty,\bar d^\infty,\bar x^\infty,\bar d^\infty,\bar u^\infty,(\bar p_x^\infty,\bar p_y^\infty))_{\vert \ker \frac{\partial E}{\partial u}(t_0,t_1,\bar x^\infty,\bar d^\infty,\bar u^\infty)}
$$
is positive definite for any $0\leq t_0<t_1$ as a consequence of conjugate point theory (see \cite{AgrachevSachkov,BCT_COCV2007,BonnardChyba,Trelat_JOTA2012}): indeed, the strong Legendre condition $\bar U^\infty\geq \delta I_m$  (see the end of Section \ref{sec_prelim_static}) implies that $Q_{t_0,t_1}>0$ for any $0\leq t_0<t_1$. Note that, usually, the strong Legendre assumption would imply $Q_{t_0,t_1}>0$ only for $t_1-t_0>0$ small enough (see \cite[Proposition 20.2]{AgrachevSachkov}), but here positive-definiteness is valid for any $t_1-t_0>0$, or equivalently, the first conjugate time associated with the trivial trajectory $(x(t),y(t),u(t))=(\bar x^\infty, t\bar d^\infty, \bar u^\infty)$ is equal to $+\infty$. To be more precise, using \cite{AgrachevSachkov,BCT_COCV2007,BonnardChyba}, the matrix of the variational system along this reference trajectory is the constant matrix\footnote{Actually, this matrix is exactly the matrix of the linearized system \eqref{systdeltaT} studied further.}
$$
\begin{pmatrix}
A^\infty & \bar B^\infty (\bar U^\infty)^{-1} (\bar B^\infty)^\top \\
\widetilde W & - (A^\infty)^\top
\end{pmatrix}
\qquad \textrm{with}\qquad \widetilde W = \begin{pmatrix} \bar W^\infty & 0 \\ 0 & 0 \end{pmatrix} .
$$
The variational system coincides with the optimality system of the LQ problem consisting of minimizing the cost functional 
$$
\int_{t_0}^{t_1} \left( \delta x(t))^\top \bar W^\infty \delta x(t) + (\delta u(t))^\top \bar U^\infty \delta u(t) \right) dt
$$
for the control system $\delta\dot x(t) = A^\infty\delta x(t)+\bar B^\infty\delta u(t)$, $\delta x(t_0)=0$, $\delta x(t_1)=0$. According to the classical conjugate point theory, for $t_0$ fixed, the first conjugate time $t_c$ is the infimum of times $t_1>t_0$ such that the minimal cost is $0$ if $t_0<t_1<t_c$ and is $-\infty$ if $t_0<t_c<t_1$. Here, we have $\bar U^\infty>0$ and $\bar W^\infty>0$  (see the end of Section \ref{sec_prelim_static}) and thus the minimal cost is $0$ for every $t_1>t_0$, which implies that $t_c=+\infty$.
\end{itemize}
Hence, by the implicit function theorem, for any $0\leq t_0<t_1$, for any $(x_0,x_1,\eta_0,\eta_1)$ close enough to $(\bar x^\infty,\bar x^\infty,\bar d^\infty,\bar d^\infty)$, there exist $(u,\psi)$ close to $(\bar u^\infty,(\bar p_x^\infty,\bar p_y^\infty))$ such that $G(t_0,t_1,x_0,\eta_0,x_1,\eta_1,u,\psi)=0$, with, moreover, $u\in L^1([t_0,t_1],\Omega_1)$. This means that the trajectory $(x(\cdot),y(\cdot),u(\cdot))$ has a normal extremal lift $(x(\cdot),y(\cdot),p_x(\cdot),p_y,-1,u(\cdot))$ satisfying the extremal equations \eqref{extrem}-\eqref{extremu}, and is such that $x(0)=x_0$, $y(0)=t_0\eta_0$, $x(t_1)=x_1$ and $y(t_1)=t_1\eta_1$. In turn, since $Q_{t_0,t_1}$ is positive definite, it follows from the conjugate point theory (see, e.g., \cite[Theorem 21.3]{AgrachevSachkov}) that $(x(\cdot),y(\cdot),u(\cdot))$ is locally optimal on $[0,t_1]$ (for the $L^\infty$ topology on $u$). Moreover, $t\mapsto(x(t),\frac{y(t)}{\max(1,t)},p_x(t),p_y,u(t))$ converges to $(\bar x^\infty, \bar d^\infty, \bar p_x^\infty, \bar p_y^\infty,\bar u^\infty)$ in $C^0$ topology when $(x_0,x_1,\eta_0,\eta_1)\rightarrow(\bar x^\infty,\bar x^\infty,\bar d^\infty,\bar d^\infty)$.
\end{proof}

\subsubsection{Local linear turnpike estimates}\label{sec_localturn}
In this section, we study the extremal system \eqref{extrem}-\eqref{extremu} locally near the trivial solution $t\mapsto (\bar x^\infty, t\bar d^\infty, \bar p_x^\infty, \bar p_y^\infty)$ of that system.
Throughout the section, we set
\begin{multline}\label{varperturb}
x(t) = \bar x^\infty+\delta x(t),\quad y(t)=t\bar d^\infty+\delta y(t),\\
p_x(t)=\bar p_x^\infty+\delta p_x(t),\quad p_y=\bar p_y^\infty+\delta p_y,
\quad u(t)=\bar u^\infty+\delta u(t),
\end{multline}
and we anticipate that $(\delta x, \delta y, \delta p_x, \delta p_y, \delta u)$ are small, uniformly on $[t_0,t_1]$; in particular, $\delta u$ is  small enough so that, thanks to Assumption \ref{H2_U} (passed to the limit), one can solve \eqref{extremu} by the implicit function theorem and get $u$ as a $C^1$ function of $(x(t),p_x(t),p_y)$. 

In this section, we establish linear turnpike estimates similar to those stated in the theorem for the system \eqref{extrem}-\eqref{extremu}, close enough to $(\bar x^\infty, t\bar d^\infty, \bar p_x^\infty, \bar p_y^\infty)$, by analyzing the linearized system.

Plugging the variables \eqref{varperturb} in \eqref{extremu} and then in \eqref{extrem}, and performing Taylor expansions at the first order, we obtain first
$$
\delta u(t) = (\bar U^\infty)^{-1} \left( \bar H_{ux}^\infty \delta x(t) + (\bar B_1^\infty)^\top \delta p_x(t) + (\bar B_2^\infty)^\top \delta p_y \right) + \mathrm{O}( \Vert\delta x(t)\Vert^2, \Vert\delta p_x(t)\Vert^2, \Vert\delta p_y\Vert^2 ) 
$$
(note that, for instance, $2\Vert\delta x(t)\Vert \Vert\delta p_x(t)\Vert \leq \Vert\delta x(t)\Vert^2 + \Vert\delta p_x(t)\Vert^2$ by the Young inequality), 
and then, recalling that $\delta p_y^T\in\R^p$ is constant,
\begin{equation}\label{systdeltaT}
\begin{split}
\delta\dot x(t) &= A_1^\infty \delta x(t) + \bar B_1^\infty (\bar U^\infty)^{-1} (\bar B_1^\infty)^\top \delta p_x(t) + \bar B_1^\infty (\bar U^\infty)^{-1} (\bar B_2^\infty)^\top \delta p_y + \mathrm{O}(\star) \\
\delta\dot y(t) &= A_2^\infty \delta x(t) + \bar B_2^\infty (\bar U^\infty)^{-1} (\bar B_1^\infty)^\top \delta p_x(t) + \bar B_2^\infty (\bar U^\infty)^{-1} (\bar B_2^\infty)^\top \delta p_y + \mathrm{O}(\star) \\
\delta\dot p_x(t) &= \bar W^\infty \delta x(t) - (A_1^\infty)^\top \delta p_x(t) - (A_2^\infty)^\top \delta p_y + \mathrm{O}(\star) \\
\delta\dot p_y(t) &= 0
\end{split}
\end{equation}
where $\mathrm{O}(\star) = \mathrm{O}( \Vert\delta x(t)\Vert^2, \Vert\delta p_x(t)\Vert^2, \Vert\delta p_y\Vert^2 )$,
which can also be written, in a matrix form, as
\begin{equation}\label{systdeltaT_matrix}
\dot Z(t) =
\begin{pmatrix}
A^\infty & \bar B^\infty (\bar U^\infty)^{-1} (\bar B^\infty)^\top \\
\widetilde W & - (A^\infty)^\top
\end{pmatrix}
Z(t) + \mathrm{O}( \Vert\delta x(t)\Vert^2, \Vert\delta p_x(t)\Vert^2, \Vert\delta p_y\Vert^2 )
\end{equation}
with
$$
Z(t) = \begin{pmatrix} \delta x(t) \\ \delta y(t) \\ \delta p_x(t) \\ \delta p_y \end{pmatrix}, \qquad
A^\infty = \begin{pmatrix} A_1^\infty & 0 \\ A_2^\infty & 0 \end{pmatrix}, \qquad
\widetilde W = \begin{pmatrix} \bar W^\infty & 0 \\ 0 & 0 \end{pmatrix} ,
$$
where $A^\infty$, $\bar B^\infty$ and $\bar W^\infty$  have been defined at the end of Section \ref{sec_prelim_static}.
Here and in what follows, the ``big-O" notation $\mathrm{O}( \Vert\delta x(t)\Vert^2, \Vert\delta p_x(t)\Vert^2, \Vert\delta p_y\Vert^2 )$ stands for a term satisfying: there exist $C>0$ and $\alpha>0$ such that, given any $t\in\R$ satisfying $\max( \Vert\delta x(t)\Vert, \Vert\delta p_x(t)\Vert, \Vert\delta p_y\Vert) \leq \alpha$, we have
\begin{equation}\label{def_O2}
\Vert \mathrm{O}( \Vert\delta x(t)\Vert^2, \Vert\delta p_x(t)\Vert^2, \Vert\delta p_y\Vert^2 ) \Vert 
\leq C \max \left( \Vert\delta x(t)\Vert^2 , \Vert\delta p_x(t)\Vert^2 , \Vert\delta p_y\Vert^2 \right)  .
\end{equation}
The ``little-o" notation $\mathrm{o}( \delta x(t), \delta p_x(t), \delta p_y )$ stands for a term satisfying: for every $\varepsilon>0$ there exists $\eta>0$ such that, given any $t\in\R$ satisfying $\max( \Vert\delta x(t)\Vert, \Vert\delta p_x(t)\Vert, \Vert\delta p_y\Vert) \leq \eta$, we have
\begin{equation*}
\Vert \mathrm{o}( \delta x(t), \delta p_x(t), \delta p_y ) \Vert 
\leq \varepsilon \max \left( \Vert\delta x(t)\Vert , \Vert\delta p_x(t)\Vert , \Vert\delta p_y\Vert \right)  .
\end{equation*}
Of course, a term $\mathrm{O}( \Vert\delta x(t)\Vert^2, \Vert\delta p_x(t)\Vert^2, \Vert\delta p_y\Vert^2 )$ is \emph{a fortiori} a $\mathrm{o}( \delta x(t), \delta p_x(t), \delta p_y )$.

It is important to note that, in the system \eqref{systdeltaT} (equivalently, \eqref{systdeltaT_matrix}), the remainder terms do not involve $\delta y(t)$.

The ``linearized system", i.e., the linearization along $t\mapsto (\bar x^\infty, t\bar d^\infty, \bar p_x^\infty, \bar p_y^\infty)$ of the system \eqref{extrem}-\eqref{extremu} is the system \eqref{systdeltaT_matrix} without the remainder terms. But, in what follows, we do not ignore those remainder terms: they have to be handled carefully.

To this aim, we write the system \eqref{systdeltaT} in the following form, more amenable to its analysis.
Setting
\begin{equation*}
z(t) = \begin{pmatrix} \delta x(t)\\ \delta p_x(t) \end{pmatrix}
\end{equation*}
and defining the matrices
\begin{equation}\label{defM}
M = \begin{pmatrix}
A_1^\infty  & \bar B_1^\infty (\bar U^\infty)^{-1} (\bar B_1^\infty)^\top \\
\bar W^\infty & - (A_1^\infty)^\top 
\end{pmatrix} ,
\qquad
V = \begin{pmatrix}
-\bar B_1^\infty (\bar U^\infty)^{-1} (\bar B_2^\infty)^\top  \\
(A_2^\infty)^\top
\end{pmatrix}  ,
\end{equation}
we write the system \eqref{systdeltaT} as
\begin{eqnarray}
\dot z(t) &=& M z(t) - V \delta p_y + \mathrm{O}(\Vert z(t)\Vert^2,\Vert\delta p_y\Vert^2)  \label{systdelta1} \\
\delta\dot y(t) &=& L z(t) + \bar B_2^\infty (\bar U^\infty)^{-1} (\bar B_2^\infty)^\top \delta p_y + \mathrm{O}(\Vert z(t)\Vert^2,\Vert\delta p_y\Vert^2) \label{systdelta2} 
\end{eqnarray}
where
$$
L = \begin{pmatrix}
A_2^\infty & \bar B_2^\infty (\bar U^\infty)^{-1} (\bar B_1^\infty)^\top
\end{pmatrix} .
$$
Note that, when $\bar H_{ux}=0$, the system \eqref{systdelta1}-\eqref{systdelta2}, without the remainder terms, corresponds exactly to the extremal system associated with the linear-quadratic optimal control problem \eqref{optcont1_LQ}.

We have the following result.

\begin{lemma}\label{lem_local_turnpike}
There exist $\varepsilon_1>0$, 
$C>0$ and $\nu>0$ such that, for all $0\leq t_0< t_1$, 
for every solution $(z(\cdot),\delta y(\cdot),\delta p_y)$ of \eqref{systdelta1}-\eqref{systdelta2} satisfying
\begin{equation}\label{estimdeltar}
\gamma_{t_0,t_1} = \max( \Vert z(t_0)\Vert, \Vert z(t_1)\Vert, \Vert\delta p_y\Vert ) \leq \varepsilon_1, 
\end{equation}
we have
\begin{equation}\label{linear_turnpike_estimates_lemma}
\begin{split}
& 
\Vert z(t)\Vert \leq C \gamma_{t_0,t_1} \left( e^{-\nu (t-t_0)} + e^{-\nu(t_1-t)} \right) + \frac{C}{t_1-t_0} \left( \gamma_{t_0,t_1} + \Vert\delta y(t_1)-\delta y(t_0)\Vert \right) ,  \\
& \Vert \delta y(t)\Vert \leq C\left( \gamma_{t_0,t_1} + \Vert\delta y(t_0)\Vert + \Vert\delta y(t_1)-\delta y(t_0)\Vert \right), \\
& \Vert\delta p_y\Vert \leq \frac{C}{t_1-t_0} \left( \gamma_{t_0,t_1} + \Vert\delta y(t_1)-\delta y(t_0)\Vert \right) ,
\end{split}
\end{equation}
for every $t\in[t_0,t_1]$. 
\end{lemma}

Lemma \ref{lem_local_turnpike} can be seen as a \emph{bootstrap} result: as soon as a solution $(\delta x(\cdot),\delta y(\cdot),\delta p_x(\cdot),\delta p_y)$ of \eqref{systdeltaT} is bounded at its extremities in $\delta x$ and $\delta p_x$ with a sufficiently small bound, then, automatically, this solution enjoys the estimates \eqref{linear_turnpike_estimates_lemma} on $[t_0,t_1]$, which are similar to the linear turnpike estimates stated in Theorem \ref{turnpike_thm}. Note that the constants $\varepsilon_1$, 
$C$ and $\nu$ do not depend on $t_0,t_1$.

\begin{proof}
We have the following spectral property for the matrix $M$:

\begin{itemize}
\item[] \textit{
The matrix $M$ is hyperbolic, i.e., all (complex) eigenvalues of $M$ have a nonzero real part. 
}
\end{itemize}

Let us establish this property.
The fact that a Hamiltonian matrix is hyperbolic under a Kalman assumption is classical (see, e.g., \cite{AK,Kucera}). The proof that we give hereafter is the same as in \cite{TZ} (which, itself, was borrowed from \cite{AndersonKokotovic,WildeKokotovic}). We recall it not only for completeness but also because the precise changes of variables that are performed below are going to be used, further, in an instrumental way.

Let $E_-$ (resp., $E_+$) be the minimal symmetric negative definite (resp., maximal symmetric positive definite) matrix solution of the algebraic Riccati equation 
$$
X A_1^\infty + (A_1^\infty)^\top X + X \bar B_1^\infty (\bar U^\infty)^{-1} (\bar B_1^\infty)^\top X - \bar W^\infty = 0 .
$$
Setting
\begin{equation}\label{def_P}
P = \begin{pmatrix}
I_n & I_n \\ E_- & E_+
\end{pmatrix} ,
\end{equation}
the matrix $P$ is invertible and we have (note that $E_+-E_-$ is invertible)
\begin{equation}\label{invP}
P^{-1} = \begin{pmatrix}
E_-^{-1} E_+ (E_+-E_-)^{-1} E_- & -(E_+-E_-)^{-1} \\
-(E_+-E_-)^{-1} E_- & (E_+-E_-)^{-1}
\end{pmatrix} 
\end{equation}
and
$$
P^{-1} M P = \begin{pmatrix}
M_- & 0 \\
0 & M_+
\end{pmatrix} 
$$
where
$$
M_- = A_1^\infty + \bar B^\infty (\bar U^\infty)^{-1} (\bar B^\infty)^\top E_-, \qquad
M_+ = A_1^\infty + \bar B^\infty (\bar U^\infty)^{-1} (\bar B^\infty)^\top E_+ .
$$
Moreover, subtracting the Riccati equations satisfied by $E_+$ and $E_-$, we have
$$
(E_+-E_-) M_+ + M_-^\top (E_+-E_-) = 0 ,
$$
hence the eigenvalues of $M_+$ are the negative of those of $M_-$, which have negative real parts by the well known Riccati theory (see \cite{AK}). Here, we have used the facts that the pair $(A_1^\infty,\bar B_1^\infty)$ satisfies the Kalman assumption (this is inferred from Assumption \ref{H2_Kalman} and by taking the limit) and that the matrices $\bar W^\infty$ and $\bar U^\infty$ are symmetric positive definite (by taking the limit in Assumptions \ref{H2_U} and \ref{H2_W}).
This proves the claimed spectral property of $M$.

\medskip

Since $M$ is hyperbolic, first of all, it is invertible and then \eqref{systdelta1} can be written as
$$
\frac{d}{dt} \left( z(t) - M^{-1}V\delta p_y \right)  =  M \left( z(t) - M^{-1}V\delta p_y \right)  + \mathrm{O}(\Vert z(t)\Vert^2,\Vert\delta p_y\Vert^2) .
$$
Second, by the proof of Lemma \ref{lem_local_turnpike}, setting
\begin{equation}\label{cdv-+}
z(t) - M^{-1}V\delta p_y = \begin{pmatrix} \delta x(t)\\ \delta p_x(t)\end{pmatrix}- M^{-1}V\delta p_y 
= P \begin{pmatrix} v_-(t)\\ v_+(t)\end{pmatrix} 
= \begin{pmatrix} v_-(t)+v_+(t) \\ E_-v_-(t)+E_+v_+(t) \end{pmatrix} ,
\end{equation}
where $P$ is defined by \eqref{def_P}, and equivalently, thanks to \eqref{invP},
\begin{equation}\label{cdv-+_inv}
\begin{pmatrix} v_-(t)\\ v_+(t)\end{pmatrix} 
= \begin{pmatrix}  E_-^{-1} E_+ (E_+-E_-)^{-1} E_- \, \delta x(t) - (E_+-E_-)^{-1} \, \delta p_x(t) \\ -(E_+-E_-)^{-1} E_- \, \delta x(t) + (E_+-E_-)^{-1} \, \delta p_x(t) \end{pmatrix} - P^{-1} M^{-1}V\delta p_y ,
\end{equation}
(we write these formulas explicitly because we will use them in an instrumental way further)
we have
\begin{equation}\label{EDOv-v+}
\begin{split}
\dot v_-(t) &= M_- v_-(t) + \mathrm{O}(\Vert v_-(t)\Vert^2,\Vert v_+(t)\Vert^2,\Vert\delta p_y\Vert^2) , \\
\dot v_+(t) &= M_+ v_+(t) + \mathrm{O}(\Vert v_-(t)\Vert^2,\Vert v_+(t)\Vert^2,\Vert\delta p_y\Vert^2) .
\end{split}
\end{equation}
We have the following result.

\begin{lemma}\label{lem_expest}
There exist $\varepsilon_1$, 
$C>0$ and $\nu>0$ such that, for all $0\leq t_0\leq t_1$, 
for every solution $(v_-(\cdot),v_+(\cdot),\delta p_y)$ of \eqref{EDOv-v+} on $[t_0,t_1]$ satisfying
\begin{equation}\label{lem_expest_assumption}
\Vert v_-(t_0)\Vert\leq\varepsilon_1, \quad
\Vert v_-(t_1)\Vert\leq\varepsilon_1, \quad
\Vert v_+(t_0)\Vert\leq\varepsilon_1, \quad
\Vert v_+(t_1)\Vert\leq\varepsilon_1 , \quad
\Vert \delta p_y\Vert\leq\varepsilon_1 , 
\end{equation}
we have
\begin{eqnarray}
\Vert v_-(t)\Vert &\leq& C\Vert v_-(t_0)\Vert\, e^{-\nu (t-t_0)} + \mathrm{o}(\Vert v_+(t_1)\Vert\, e^{-\nu (t_1-t)}) + \mathrm{o}(\delta p_y) \label{estimcroisees_1} \\
\Vert v_+(t)\Vert &\leq& C\Vert v_+(t_1)\Vert\, e^{-\nu (t_1-t)} + \mathrm{o}(\Vert v_-(t_0)\Vert\, e^{-\nu (t-t_0)}) + \mathrm{o}(\delta p_y) \label{estimcroisees_2} 
\end{eqnarray}
for every $t\in[t_0,t_1]$. 
\end{lemma}

Lemma \ref{lem_expest} is proved in Section \ref{sec_proof_lem_expest} hereafter. Let us admit it temporarily and finish the proof of Lemma \ref{lem_local_turnpike}.

Using \eqref{cdv-+}, multiplying if necessary $\varepsilon_1$ and $C$ by a positive constant only depending on $P$, $M$, $E_-$ and $E_+$, we infer from the assumption \eqref{estimdeltar} done in Lemma \ref{lem_local_turnpike} that \eqref{lem_expest_assumption} is satisfied
and thus also \eqref{estimcroisees_1} and \eqref{estimcroisees_2} by Lemma \ref{lem_expest}, and therefore
\begin{equation}\label{16:51}
\Vert z(t) - M^{-1}V\delta p_y\Vert \leq C\gamma_{t_0,t_1} \left( e^{-\nu (t-t_0)} + e^{-\nu(t_1-t)} \right) + \mathrm{o}(\delta p_y) \qquad \forall t\in[t_0,t_1]  .
\end{equation}
It remains to estimate the norm of $\delta p_y$.
Writing \eqref{systdelta2} in the form
\begin{equation}\label{eqdoty}
\delta \dot y(t) = L ( z(t) - M^{-1} V \delta p_y ) + R \delta p_y  + \mathrm{o}(z(t),\delta p_y)
\end{equation}
where $R$ is the square matrix of size $p$ defined by
\begin{equation}\label{defRT}
R = L M^{-1}V + \bar B_2^\infty(U^\infty)^{-1}(\bar B_2^\infty)^\top  ,
\end{equation}
integrating in time, using that $\int_{t_0}^{t_1} e^{-\nu t}\, dt\leq\frac{1}{\nu}$, we infer (taking a larger constant $C$ if necessary), using \eqref{16:51}, that 
\begin{equation}\label{deltay2}
\Vert \delta y(t) - \delta y(t_0) - (t-t_0) R\,\delta p_y\Vert\leq C\gamma_{t_0,t_1} + \mathrm{o}((t-t_0)\delta p_y) \qquad \forall t\in[t_0,t_1] 
\end{equation}
and thus
\begin{equation}\label{eq_in_py2}
\Vert R\, \delta p_y \Vert \leq \frac{1}{t_1-t_0} \left( C\gamma_{t_0,t_1} + \Vert\delta y(t_1)-\delta y(t_0)\Vert \right) + \mathrm{o}(\delta p_y) .
\end{equation}
By Lemma \ref{lem_R}, which is established in Section \ref{sec_proof_lem_R}, the matrix $R$ is symmetric positive definite, and thus is invertible. 
Taking a larger constant $C$ if necessary, we infer from \eqref{eq_in_py2} and from Lemma \ref{lem_R} the estimate for $\delta p_y$ given in \eqref{linear_turnpike_estimates_lemma}.
We finally infer from \eqref{16:51} and \eqref{deltay2} the estimates for $z(t)$ and $\delta y(t)$ in \eqref{linear_turnpike_estimates_lemma}.
\end{proof}

\begin{remark}
In \eqref{syst1}-\eqref{syst2}-\eqref{syst12_Omega}-\eqref{mincost1}, the dynamics in $y$ is written as $\dot y(t)=h(x(t),y(t),u(t))$ with $h(x,y,u)=g(x,u)$: the mapping $h$ does not depend on $y$. Our strategy of proof would not apply to a general mapping $h$, depending on $y$. Indeed, otherwise, when linearizing $\dot y(t)=h(x(t),y(t),u(t))$ along the path $(\bar x^\infty,t\bar d^\infty,\bar u^\infty)$, we would get $\delta\dot y(t)=A_2^\infty \delta x(t)+\bar B_2^\infty\delta u(t)+Q(t)\delta y(t)$ where $Q(t)=\frac{\partial h}{\partial y}(\bar x^\infty,t\bar d^\infty,\bar u^\infty)$ depends on $t$ and thus the matrix $M$ defined by \eqref{defM} would as well depend on $t$ (anyway, only through the part in $Q$). The argument of hyperbolicity, then, cannot be applied in general, in particular we think that it may fail whenever the eigenvectors of $Q(t)$ are oscillating too fast. Nevertheless, we think that the technique of proof may be extended to the case where $t\mapsto Q(t)$ is slowly-varying, as in the context of control or stabilization by quasi-static deformation (see \cite{CoronTrelat_SICON2004} and references therein). We leave this issue open.
\end{remark}


\subsubsection{Proof of Lemma \ref{lem_expest}}\label{sec_proof_lem_expest}
Similar estimates have been given in \cite[Section 3.2, (34)]{TZ}, but without the remainder term in $\delta p_y$. The proof uses the method of successive approximations, following the classical Banach fixed point argument, like in the proof of the stable manifold theorem done in \cite[Chapter 2.7, proof of the theorem page 109]{Perko}.

We infer from \eqref{EDOv-v+} that 
\begin{equation*}
\begin{split}
v_-(t) &= e^{(t-t_0)M_-} v_-(t_0) + \int_{t_0}^t e^{(t-s)M_-} r_-(v_-(s),v_+(s),\delta p_y) \, ds \\
v_+(t) &= e^{(t-t_1)M_+} v_+(t_1) - \int_t^{t_1} e^{(t-s)M_+} r_+(v_-(s),v_+(s),\delta p_y) \, ds 
\end{split}
\end{equation*}
where $r_\pm(v_-,v_+,\delta p_y) = \mathrm{O}(\Vert v_-\Vert^2,\Vert v_+\Vert^2,\Vert\delta p_y\Vert^2)$. 
Since $M_-$ and $-M_+$ are Hurwitz, there exist $C,\nu>0$ such that $\Vert e^{\tau M_-}\Vert \leq C e^{-2\nu \tau}$ and $\Vert e^{-\tau M_+}\Vert \leq C e^{-2\nu \tau}$ for every $\tau\geq 0$.
Taking a larger constant $C$ if necessary, using the definition \eqref{def_O2}, there exists $\alpha_0>0$ such that, for every $\alpha\in(0,\alpha_0)$, under the \emph{a priori} assumption that, for every $t\in[t_0,t_1]$, $\max(\Vert v_-(t)\Vert, \Vert v_+(t)\Vert, \Vert\delta p_y\Vert) \leq \alpha$, we have
\begin{equation}\label{doubleyz1430}
\begin{split}
\Vert v_-(t)\Vert &\leq C e^{-2\nu(t-t_0)} \Vert v_-(t_0)\Vert + \alpha\frac{C}{2\nu}\Vert\delta p_y\Vert + \alpha C \int_{t_0}^t e^{-2\nu(t-s)} ( \Vert v_-(s)\Vert + \Vert v_+(s)\Vert ) \, ds \\
\Vert v_+(t)\Vert &\leq C e^{-2\nu(t_1-t)} \Vert v_+(t_1)\Vert + \alpha\frac{C}{2\nu}\Vert\delta p_y\Vert + \alpha C \int_t^{t_1} e^{-2\nu(s-t)} ( \Vert v_-(s)\Vert + \Vert v_+(s)\Vert ) \, ds 
\end{split}
\end{equation}
(we have used that $\int_0^{+\infty} e^{-2\nu\tau}\, d\tau=\frac{1}{2\nu}$).
The reader who is familiar with the proof of the stable manifold theorem can then conclude straightforwardly, by plugging the rough estimates $\Vert v_-(t)\Vert \lesssim C e^{-\nu(t-t_0)} \Vert v_-(t_0)\Vert$ and $\Vert v_+(t)\Vert \lesssim C e^{-\nu(t_1-t)}$ into the integrals and, since the argument goes by fixed point, we know in advance that it is valid to derive in this way the estimates \eqref{estimcroisees_1} and \eqref{estimcroisees_2}. 

\medskip

Hereafter, as required by one of the referees, we give the complete argument. 
Using that $e^{-2\nu s}\leq e^{-\nu s}$ for every $s\geq 0$ (like in the proof of the stable manifold theorem done in \cite[Chapter 2.7, page 109]{Perko}, it is important to let a ``margin" in order to avoid a resonance phenomenon), 
we infer from \eqref{doubleyz1430} that
\begin{eqnarray}
y &\leq& a + \alpha Ky + \alpha Kz \label{inegy} \\
z &\leq& b + \alpha Ly + \alpha Lz \label{inegz}
\end{eqnarray}
where 
$$
y(t) = \Vert v_-(t)\Vert, \qquad z(t) = \Vert v_+(t)\Vert,
$$
$$
a(t)=Ce^{-\nu(t-t_0)}y(t_0)+\alpha\frac{C}{2\nu}\Vert\delta p_y\Vert,
\qquad b(t)=Ce^{-\nu(t_1-t)}z(t_1)+\alpha\frac{C}{2\nu}\Vert\delta p_y\Vert,
$$
and where the linear operators $K,L:C^0([t_0,t_1])\rightarrow C^0([t_0,t_1])$ are defined by
$$
(Kf)(t) = C\int_{t_0}^t e^{-2\nu(t-s)} f(s)\, ds, \qquad(Lf)(t) = C\int_t^{t_1} e^{-2\nu(s-t)} f(s)\, ds \qquad \forall f\in C^0([t_0,t_1]).
$$
The operators $K$ and $L$ are bounded on $C^0([t_0,t_1])$ (we have $\Vert K\Vert\leq \frac{C}{2\nu}$ and the same for $L$) and moreover are positively monotone in the following sense: given any $f_1,f_2\in C^0([t_0,t_1])$,
$$
0\leq f_1\leq f_2 \quad\Rightarrow\quad 0 \leq Kf_1\leq Kf_2\quad\textrm{and}\quad 0 \leq Lf_1\leq Lf_2 .
$$
Taking $\alpha_0$ smaller if necessary so that $0<\alpha_0<\frac{2\nu}{C}$, the operators $\mathrm{id}-\alpha K$ and $\mathrm{id}-\alpha L$ are boundedly invertible, and the operators $(\mathrm{id}-\alpha K)^{-1}$ and $(\mathrm{id}-\alpha L)^{-1}$ are positively monotone. Indeed,
$$
(\mathrm{id}-\alpha K)^{-1} = \sum_{j=0}^{+\infty} \alpha^j K^j 
$$
is a converging sum of positively monotone operators, hence is positively monotone (and the same is true for $(\mathrm{id}-\alpha L)^{-1}$). 

By \eqref{inegz}, we have $(\mathrm{id}-\alpha L)z \leq b+\alpha Ly$ and thus, by positive monotonicity, $z \leq (\mathrm{id}-\alpha L)^{-1}(b+\alpha Ly)$. Plugging this inequality into \eqref{inegy}, using that $\mathrm{id} + (\mathrm{id}-\alpha L)^{-1} \alpha L = (\mathrm{id}-\alpha L)^{-1}$, we get
$$
y \leq a + \alpha K (\mathrm{id}-\alpha L)^{-1}b + \alpha K (\mathrm{id}-\alpha L)^{-1} y
$$
and thus $(\mathrm{id} - U_\alpha) y  \leq a + U_\alpha b$ where $U_\alpha=\alpha K (\mathrm{id}-\alpha L)^{-1}$. The linear operator $U_\alpha$ is bounded on $C^0([t_0,t_1])$, and $\Vert U_\alpha\Vert \leq \frac{\alpha C}{2\nu-\alpha C}$. Taking $\alpha_0$ smaller if necessary so that $0<\alpha_0<\frac{\nu}{C}$, $\mathrm{id} - U_\alpha$ is boundedly invertible and
$$
(\mathrm{id} - U_\alpha)^{-1} 
= \sum_{j=0}^{+\infty} \alpha^j K^j \left( \sum_{k=0}^{+\infty} \alpha^k L^k \right)^j 
$$
is a converging sum, with positive coefficients, of positively monotone operators, hence is positively monotone. Therefore 
$$
y \leq (\mathrm{id} - U_\alpha)^{-1}a + (\mathrm{id} - U_\alpha)^{-1}U_\alpha b 
= a +  (\mathrm{id} - U_\alpha)^{-1}U_\alpha (a+b)
$$
where we have used that $(\mathrm{id} - U_\alpha)^{-1} = \mathrm{id} + (\mathrm{id} - U_\alpha)^{-1}U_\alpha$. Noting that $(\mathrm{id}-\alpha K(\mathrm{id}-\alpha L)^{-1})^{-1}$ and $\alpha K$ commute (this is straightforwardly checked by a series expansion), we even have
$$
(\mathrm{id} - U_\alpha)^{-1}U_\alpha = \alpha K (\mathrm{id}-\alpha(K+L))^{-1} .
$$
The same can be done for $z$. We finally have
\begin{equation*}
\begin{split}
y &\leq a +  \alpha K (\mathrm{id}-\alpha(K+L))^{-1} (a+b) \\
z &\leq b +  \alpha L (\mathrm{id}-\alpha(K+L))^{-1} (a+b) 
\end{split}
\end{equation*}
Since $(\mathrm{id}-\alpha(K+L))^{-1} = \sum_{j=0}^{+\infty} \alpha^j(K+L)^j$, let us estimate the power iterates of $\alpha(K+L)$ applied to $e^{-\nu(t-t_0)}$ and $e^{-\nu(t_1-t)}$. We compute
$$
K e^{-\nu(t-t_0)} \leq \frac{C}{\nu} e^{-\nu(t-t_0)},
\qquad
L e^{-\nu(t-t_0)} \leq \frac{C}{3\nu} e^{-\nu(t-t_0)},
$$
$$
K e^{-\nu(t_1-t)} \leq \frac{C}{3\nu} e^{-\nu(t_1-t)},
\qquad
L e^{-\nu(t_1-t)} \leq \frac{C}{\nu} e^{-\nu(t_1-t)},
$$
hence $\alpha^j(K+L)^je^{-\nu(t-t_0)}\leq (\frac{2C\alpha}{\nu})^j e^{-\nu(t-t_0)}$, and the same for $e^{-\nu(t_1-t)}$, and thus, taking the (convergent) sum and changing slightly constants if necessary, we finally infer (taking $\alpha_0$ small enough) that, given any $\alpha\in(0,\alpha_0)$,
\begin{equation}\label{inegyz1518}
\begin{split}
y(t) &\leq C e^{-\nu(t-t_0)} y(t_0) + 2\alpha \frac{C}{\nu} ( \Vert\delta p_y\Vert + e^{-\nu(t-t_0)}y(t_0) + e^{-\nu(t_1-t)} y(t_1) ) \\
z(t) &\leq C e^{-\nu(t_1-t)} z(t_1) + 2\alpha \frac{C}{\nu} ( \Vert\delta p_y\Vert + e^{-\nu(t-t_0)}y(t_0) + e^{-\nu(t_1-t)} y(t_1) )
\end{split}
\end{equation}
for every $t\in[t_0,t_1]$, under the \emph{a priori} assumption that $\max(y(t), z(t), \Vert\delta p_y\Vert) \leq \alpha$ for every $t\in[t_0,t_1]$.
Finally, we infer from \eqref{inegyz1518} that, if $\max(y(t_0), y(t_1),z(t_0),z(t_1), \Vert\delta p_y\Vert) \leq \alpha / (C+6\frac{C}{\nu}\alpha_0) $, then actually $\max(y(t), z(t), \Vert\delta p_y\Vert) \leq \alpha$ for every $t\in[t_0,t_1]$.
Since this has been done for every $\alpha\in(0,\alpha_0)$, the lemma follows. 

\subsubsection{Proof of Lemma \ref{lem_R}}\label{sec_proof_lem_R}
\begin{lemma}\label{lem_R}
The matrix $R$ defined by \eqref{defRT} is symmetric positive definite, and thus is invertible. 
\end{lemma}

This section is devoted to proving Lemma \ref{lem_R}. 
Throughout this section, we do not write the superscript $\infty$ in all matrices, to facilitate the reading. 
Let us first express the matrix $R$, defined by \eqref{defRT}, in a more explicit way.
Recall that
$$
R = \begin{pmatrix} A_2 \bar B_2\bar U^{-1}\bar B_1^\top\end{pmatrix} M^{-1}\begin{pmatrix}-\bar B_1\bar U^{-1}\bar B_2^\top\\ A_2^\top\end{pmatrix}+\bar B_2\bar U^{-1}\bar B_2^\top\qquad\textrm{with}\qquad
M =  \begin{pmatrix} A_1 & \bar B_1\bar U^{-1}\bar B_1^\top \\ \bar W & -A_1^\top \end{pmatrix} .
$$
We set
$$
 \begin{pmatrix} C_1\\ C_2\end{pmatrix} = M^{-1}\begin{pmatrix}-\bar B_1\bar U^{-1}\bar B_2^\top\\ A_2^\top\end{pmatrix}
$$
so that
\begin{equation}\label{10:50}
\begin{split}
A_1C_1+\bar B_1\bar U^{-1}\bar B_1^\top C_2&=-\bar B_1\bar U^{-1}\bar B_2^\top \\
\bar WC_1-A_1^\top C_2&=A_2^\top
\end{split}
\end{equation}
Since $\bar W$ is invertible, we infer from the second equation of \eqref{10:50} that $C_1=\bar W^{-1}A_1^\top C_2+\bar W^{-1}A_2^\top$, and plugging this expression into the first equation of \eqref{10:50} gives
\begin{equation}\label{10:53}
\left( A_1\bar W^{-1}A_1^\top+\bar B_1\bar U^{-1}\bar B_1^\top\right) C_2 = -A_1\bar W^{-1}A_2^\top - \bar B_1\bar U^{-1}\bar B_2^\top
\end{equation}

\begin{lemma}\label{lem_matinv}
The matrix $A_1\bar W^{-1}A_1^\top+\bar B_1\bar U^{-1}\bar B_1^\top$ is invertible.
\end{lemma}

\begin{proof}
Let $\xi\in\R^n$ be such that $\left( A_1\bar W^{-1}A_1^\top+\bar B_1\bar U^{-1}\bar B_1^\top\right)\xi=0$. Multiplying to the left by $\xi^\top$, we obtain $\Vert \bar W^{-1/2}A_1^\top\xi\Vert^2+\Vert \bar U^{-1/2}\bar B_1^\top\xi\Vert^2=0$ and hence $\xi\in\ker(A_1^\top)\cap\ker(\bar B_1^\top)$. As in the proof of Lemma \ref{lem_local_turnpike}, since the pair $(A_1,\bar B_1)$ satisfies the Kalman condition, the latter intersection is $\{0\}$. The claim follows.
\end{proof}

We infer from \eqref{10:53} and from Lemma \ref{lem_matinv} that
$$
C_2 = - \left( A_1\bar W^{-1}A_1^\top+\bar B_1\bar U^{-1}\bar B_1^\top\right)^{-1} \left( A_1\bar W^{-1}A_2^\top + \bar B_1\bar U^{-1}\bar B_2^\top \right)
$$
and thus
$$
C_1 = - \bar W^{-1}A_1^\top \left( A_1\bar W^{-1}A_1^\top+\bar B_1\bar U^{-1}\bar B_1^\top\right)^{-1} \left( A_1\bar W^{-1}A_2^\top + \bar B_1\bar U^{-1}\bar B_2^\top \right) + \bar W^{-1} A_2^\top .
$$
We conclude that
\begin{multline*}
R = - \left( A_1\bar W^{-1}A_2^\top + \bar B_1\bar U^{-1}\bar B_2^\top \right)^\top \left( A_1\bar W^{-1}A_1^\top+\bar B_1\bar U^{-1}\bar B_1^\top\right)^{-1} \left( A_1\bar W^{-1}A_2^\top + \bar B_1\bar U^{-1}\bar B_2^\top \right) \\
+ A_2\bar W^{-1}A_2^\top + \bar B_2\bar U^{-1}\bar B_2^\top 
\end{multline*}
(note that $R$ is symmetric) that we can write as the sum of two symmetric matrices
\begin{equation*}
\begin{split}
R &= A_2 \bar W^{-1/2} \left( - \bar W^{-1/2} A_1^\top \left( A_1\bar W^{-1}A_1^\top+\bar B_1\bar U^{-1}\bar B_1^\top\right)^{-1} A_1 \bar W^{-1/2} + I_n \right) \bar W^{-1/2} A_2^\top \\
& \quad + \bar B_2 \bar U^{-1/2} \left( - \bar U^{-1/2} \bar B_1^\top \left( A_1\bar W^{-1}A_1^\top+\bar B_1\bar U^{-1}\bar B_1^\top\right)^{-1} \bar B_1 \bar U^{-1/2} + I_m \right) \bar U^{-1/2} \bar B_2^\top \\
&= \tilde A_2 \left( - \tilde A_1^\top \left( \tilde A_1\tilde A_1^\top+ \tilde B_1\tilde B_1^\top\right)^{-1} \tilde A_1 + I_n \right) \tilde A_2^\top 
+ \tilde B_2 \left( - \tilde B_1^\top \left( \tilde A_1\tilde A_1^\top+\tilde B_1\tilde B_1^\top\right)^{-1} \tilde B_1 + I_m \right) \tilde B_2^\top 
\end{split}
\end{equation*}
where we have set $\tilde A_i = A_i\bar W^{-1/2}$ and $\tilde B_i=\bar B_i\bar U^{-1/2}$ for $i=1,2$, 
and we are going to prove hereafter that the above two matrices are symmetric positive definite, so that $R$ is itself symmetric positive definite.

We have the following general result.

\begin{lemma}\label{lem_general}
Let $A$ be an arbitrary real-valued square matrix of size $n$, and let $B$ be an arbitrary real-valued matrix of size $n\times m$, where $n$ and $m$ are arbitrary nonzero integers. If $\ker(A^\top)\cap\ker(B^\top)=\{0\}$ then
\begin{equation}\label{ineqA1}
B^\top \left( AA^\top + BB^\top\right)^{-1} B \preceq I_m
\end{equation}
meaning that the symmetric matrix $I_m - B^\top \left( AA^\top + BB^\top\right)^{-1} B$ is positive semi-definite.
\end{lemma}

\begin{proof}
Since $\ker(A^\top)\cap\ker(B^\top)=\{0\}$, the matrix $AA^\top + BB^\top$ is invertible (same argument as in Lemma \ref{lem_matinv}).

When $B$ is invertible, the result is obvious: starting from
$$
AA^\top + BB^\top \succeq BB^\top ,
$$
taking the inverse and multiplying to the left by $B^\top$ and to the right by $B$, we obtain \eqref{ineqA1}.

When $B$ is not invertible, we follow anyway the above reasoning, adding $\varepsilon I_m$ for $\varepsilon>0$ to recover an appropriate invertibility property: starting from
$$
AA^\top + BB^\top + \varepsilon I_m \succeq BB^\top + \varepsilon I_m ,
$$
taking the inverse and multiplying to the left by $B^\top$ and to the right by $B$, we obtain
$$
B^\top \left( AA^\top + BB^\top + \varepsilon I_m\right)^{-1} B \preceq B^\top \left( BB^\top + \varepsilon I_m\right)^{-1} B .
$$
Now, we use the general fact that 
$$
B^\top \left( BB^\top + \varepsilon I_m\right)^{-1} B \underset{\varepsilon\rightarrow 0}{\longrightarrow} B^\#B
$$
which is a consequence of the Tikhonov regularization in the Moore-Penrose pseudo-inverse theory (see \cite{Tikhonov}). Here, $B^\#$ is the Moore-Penrose pseudo-inverse of $B$.
We also note the general fact that $B^\#B\preceq I_m$, because $I_m-B^\#B$ is an orthogonal projection.

Then, taking the limit $\varepsilon\rightarrow 0$ gives \eqref{ineqA1}.
\end{proof}

Noting that $\ker(\tilde A_1^\top)\cap\ker(\tilde B_1^\top)=\{0\}$ (this comes again from the Kalman condition on the pair $(A_1,\bar B_1)$), we infer from Lemma \ref{lem_general} that
$$
\tilde A_1^\top \left( \tilde A_1\tilde A_1^\top+ \tilde B_1\tilde B_1^\top\right)^{-1} \tilde A_1 \preceq I_n
\qquad\textrm{and}\qquad
\tilde B_1^\top \left( \tilde A_1\tilde A_1^\top+\tilde B_1\tilde B_1^\top\right)^{-1} \tilde B_1 \preceq I_m .
$$
Therefore $R\succeq 0$, i.e., $R$ is positive semi-definite.
This concludes the proof of Lemma \ref{lem_R}.

\subsubsection{Uniform neighborhoods: end of the proof of Proposition \ref{prop_sensitivity_dyn}}\label{sec-unifneighb}
To finish the proof of Proposition \ref{prop_sensitivity_dyn}, it remains to prove that one can choose $\varepsilon_0>0$ that does not depend on $t_0,t_1$ whenever $t_1-t_0$ is large enough. This fact will follow by inspecting more finely (and bootstrapping again) the study performed in Section \ref{sec_localturn}. 

Before starting the proof, we note that it suffices to establish Proposition \ref{prop_sensitivity_dyn} for $t_0=0$ and $\eta_0=0$. Indeed, if Proposition \ref{prop_sensitivity_dyn} is true for $t_0=0$ and $\eta_0=0$, then we infer the general case by shifting time and by considering $y(t)-y(t_0\eta_0)$ instead of $y(t)$. Hence, in what follows, we assume without loss of generality that $t_0=0$ and that $y(0)=0$, i.e., $\eta_0=0$.

Given $t_1>0$ arbitrarily fixed, we have proved in Lemma \ref{lem_sensitivity} that there exists $\varepsilon_0>0$ depending on $t_1$ such that, for any $\varepsilon\in(0,\varepsilon_0)$, for any $x_0,x_1\in B(\bar x^\infty,\varepsilon)$ and any $\eta_1\in B(\bar d^\infty,\varepsilon)$, the problem $\mathcal{P}(0,t_1,x_0,x_1,0,\eta_1)$ has a locally optimal solution $(x(\cdot),y(\cdot),p_x(\cdot),p_y(\cdot),u(\cdot))$ on $[0,t_1]$ satisfying \eqref{extrem}-\eqref{extremu}. 
We set $\delta x(t) = x(t)-\bar x^\infty$, $\delta y(t) = y(t)-t\bar d^\infty$, $\delta p_x(t) = p_x(t)-\bar p_x^\infty$, $\delta p_y = p_y-\bar p_y^\infty$ as in \eqref{varperturb}.

Having in mind the \emph{shooting method}, given any $t_1>0$, any $\delta x(0)$ such that $\Vert \delta x(0)\Vert \leq\varepsilon_0$ (recall that $\delta y(0)=0$), we consider the mapping $\Phi_{t_1,\delta x(0)}: \R^n\times\R^p\rightarrow \R^n\times\R^p$ defined (in a neighborhood of $(0,0)$) by
$$
\Phi_{t_1,\delta x(0)}(\delta p_x(0), \delta p_y) = \left( \delta x(t_1), \frac{\delta y(t_1)}{t_1} \right) 
$$
where $(\delta x(\cdot), \delta y(\cdot), \delta p_x(\cdot), \delta p_y)$ is solution of \eqref{systdeltaT}.
The mapping $\Phi_{t_1,\delta x(0)}$ is interpreted as the \emph{exponential mapping} (see \cite{BCT_COCV2007}) associated with the Hamiltonian system \eqref{extrem}-\eqref{extremu}, and it follows from the conjugate point theory that $t_1$ is not a conjugate time for the initial point $(\delta x(0),\delta y(0)=0)$ if and only if the mapping $\Phi_{t_1,\delta x(0)}$ is a local diffeomorphism (see \cite[Proposition 2.9]{BCT_COCV2007} or see \cite[Theorem 21.1]{AgrachevSachkov} in a more general context).

Here, the \emph{shooting problem} that we are going to solve is the following: we will prove that there exists $\varepsilon_0>0$ (to be chosen small enough) and $T_0>0$ (to be chosen large enough) such that, given any $t_1>T_0$, given any $x_0,x_1\in B(\bar x^\infty,\varepsilon_0)$ and given any $y_1^{t_1}\in\R^p$ such that $\frac{y_1^{t_1}}{t_1}\in B(\bar d^\infty,\varepsilon_0)$, there exist $\delta p_x(0)\in\R^n$ and $\delta p_y\in\R^p$ such that 
\begin{equation}\label{shooting_problem}
\Phi_{t_1,\delta x(0)}(\delta p_x(0), \delta p_y) 
= \left( x_1-\bar x^\infty, \frac{y_1^{t_1}}{t_1}-\bar d^\infty \right)
\end{equation}
(this is the shooting problem), i.e., such that the solution $(\delta x(\cdot), \delta y(\cdot), \delta p_x(\cdot), \delta p_y)$ of \eqref{systdeltaT}, of initial condition $(\delta x(0) = x_0-\bar x^\infty, \delta y(0)=0, \delta p_x(0), \delta p_y)$, satisfies $\delta x(t_1) = x_1-\bar x^\infty$ and $\delta y(t_1) = y_1^{t_1}-t_1\bar d^\infty$, and moreover we want to prove that $\Phi_{t_1,\delta x(0)}$ is a local diffeomorphism (and this, we insist, whatever $t_1>0$ large enough may be).

Let us consider again and follows the steps of the proof of Lemma \ref{lem_local_turnpike}. Assuming that $t_1\geq T_0$, setting $M^{-1}V=\begin{pmatrix}H_1 & H_2\end{pmatrix}^\top$ and $P^{-1}M^{-1}V=\begin{pmatrix}H_3 & H_4\end{pmatrix}^\top$, it follows from \eqref{cdv-+} that
\begin{equation}\label{deltaxt1}
\delta x(t_1) = v_-(t_1) + v_+(t_1) + H_1\,\delta p_y, 
\end{equation}
and from \eqref{cdv-+_inv} that
\begin{eqnarray}
v_-(0) &=& E_-^{-1} E_+ (E_+-E_-)^{-1} E_- \, \delta x(0) - (E_+-E_-)^{-1} \, \delta p_x(0) + H_3\,\delta p_y \label{v-0} \\ 
v_+(0) &=& -(E_+-E_-)^{-1} E_- \, \delta x(0) + (E_+-E_-)^{-1} \, \delta p_x(0) + H_4\,\delta p_y \label{v+0} 
\end{eqnarray}
Now, using \eqref{estimcroisees_1} at $t=t_1$ and \eqref{v-0}, we infer 
that
$$
v_-(t_1) =  \mathrm{O}(\delta x(0), \delta p_x(0), \delta p_y) \, e^{-\nu t_1} + \mathrm{o}(v_+(t_1)) + \mathrm{o}(\delta p_y) ,
$$
and combining with \eqref{deltaxt1} we obtain
\begin{multline}\label{v+t1}
v_+(t_1) = \delta x(t_1) - v_-(t_1) - H_1\,\delta p_y \\
= \delta x(t_1) - H_1\,\delta p_y + \mathrm{O}(\delta x(0), \delta p_x(0), \delta p_y) \, e^{-\nu t_1} + \mathrm{o}( \delta x(t_1), \delta p_y )  .
\end{multline}
Since $\delta x(0)=\mathrm{O}(1)$ and since we want to solve $\delta x(t_1)=x_1-\bar x^\infty=\mathrm{O}(1)$, using \eqref{estimcroisees_2} at $t=0$, \eqref{v-0} and \eqref{v+t1}, we must have
$$
v_+(0) = \mathrm{O}(\delta x(t_1),\delta p_y) \, e^{-\nu t_1} + \mathrm{o}(\delta x(0), \delta p_x(0), \delta p_y) ,
$$
and thus using \eqref{v+0}, finally, we find that
\begin{equation}\label{deltapx0}
\delta p_x(0) = E_- \delta x(0) - (E_+-E_-) H_4\,\delta p_y + \mathrm{O}(\delta x(t_1),\delta p_y) \, e^{-\nu t_1} + \mathrm{o}(\delta x(0),\delta p_y) .
\end{equation}

Finally, let us consider the last equation of the shooting problem \eqref{shooting_problem} to be solved, $\delta y(t_1)=y_1^{t_1}-t_1\bar d^\infty$.
Using \eqref{cdv-+}, we rewrite \eqref{eqdoty} as
\begin{equation}\label{newdeltay}
\delta\dot y(t) = L \begin{pmatrix} v_-(t)+v_+(t) \\ E_-v_-(t)+E_+v_+(t) \end{pmatrix} + R \delta p_y  + \mathrm{o}(z(t),\delta p_y) .
\end{equation}
It follows from \eqref{estimcroisees_1}, \eqref{estimcroisees_2}, \eqref{v-0}, \eqref{v+t1} and \eqref{deltapx0} that
\begin{eqnarray}
v_-(t) &=& \mathrm{O}(\delta x(0),\delta p_y,e^{-\nu t_1}) \, e^{-\nu t} + \mathrm{o}(\delta x(t_1),\delta p_y,e^{-\nu t_1}) \, e^{-\nu(t_1-t)} + \mathrm{o}(\delta p_y) ,    \label{newv-} \\
v_+(t) &=&  \mathrm{o}(\delta x(0),\delta p_y,e^{-\nu t_1}) \, e^{-\nu t} + \mathrm{O}(\delta x(t_1),\delta p_y,e^{-\nu t_1}) \, e^{-\nu(t_1-t)} + \mathrm{o}(\delta p_y)  .   \label{newv+}
\end{eqnarray}
Now, we infer from \eqref{newdeltay}, \eqref{newv-} and \eqref{newv+} that
$$
\delta\dot y(t) = R \delta p_y + \mathrm{O}(\delta x(0),\delta x(t_1),\delta p_y,e^{-\nu t_1}) ( e^{-\nu t} + e^{-\nu(t_1-t)} ) + \mathrm{o}(\delta p_y)
$$
and, integrating in time, using that $\delta y(0)=0$ and $\int_0^{t_1} e^{-\nu t}\, dt\leq\frac{1}{\nu}$, noting that $\mathrm{O}(\delta p_y) = \mathrm{o}(t_1\delta p_y)$ for $t_1$ large enough, we finally obtain
$$
\delta y(t_1) = t_1 R \delta p_y + \mathrm{o}(t_1\delta p_y) + \mathrm{O}(\delta x(0),\delta x(t_1),e^{-\nu t_1}) .
$$
Since we want to solve $\delta y(t_1)=y_1^{t_1}-t_1\bar d^\infty$, and since $R$ is invertible by Lemma \ref{lem_R}, we must have
\begin{equation}\label{shoot_deltapy}
\delta p_y = R^{-1} \left( \frac{y^{t_1}}{t_1}-\bar d^\infty \right) + \mathrm{o} \left( \frac{y^{t_1}}{t_1}-\bar d^\infty \right) + \frac{1}{t_1} \mathrm{O}( \delta x(0), \delta x(t_1), e^{-\nu t_1} )
\end{equation}
and, plugging \eqref{shoot_deltapy} into \eqref{deltapx0} we obtain
\begin{equation}\label{deltapx0_new}
\delta p_x(0) = E_- \delta x(0) - (E_+-E_-) H_4 R^{-1} \left( \frac{y^{t_1}}{t_1}-\bar d^\infty \right) + \mathrm{o} \left( \delta x(0), \delta x(t_1), \frac{y^{t_1}}{t_1}-\bar d^\infty, e^{-\nu t_1} \right)  .
\end{equation}
All in all, the shooting problem \eqref{shooting_problem} is equivalent to the equations \eqref{shoot_deltapy} and  \eqref{deltapx0_new}, which can be solved for an $\varepsilon_0$ small enough that is independent of $t_1$, and whose solving shows that the mapping $\Phi_{t_1,\delta,x(0)}$, which maps $(\delta p_x(0),\delta p_y)$ to $(\delta x(t_1),\frac{\delta y(t_1)}{t_1})$, is a local diffeomorphism, whatever $t_1>0$ large enough may be.
In turn, this gives the local uniqueness property stated in \ref{casei}.
This concludes the proof of the proposition.

\begin{remark}
When $\frac{y^{t_1}}{t_1}=\bar d^\infty+\mathrm{o}(1)$ as $t_1\rightarrow+\infty$, passing to the limit $t_1\rightarrow+\infty$ in \eqref{deltapx0_new} gives $\delta p_x(0) = E_- \delta x(0)$, which is the classical relation in algebraic Riccati theory (see, e.g., \cite[Proof of Theorem 4.4.5]{trelat_book}). 
\end{remark}

\subsection{Exploiting dissipativity} \label{sec_exploit_dissip}
Recall that we have considered a sequence $(T_k)_{k\in\N^*}$ of positive real numbers converging to $+\infty$ such that \eqref{limit_tuple} holds, i.e., $(\bar x^\infty,\bar d^\infty,\bar u^\infty,\bar p_x^\infty,\bar p_y^\infty)$ is a closure point of the family $(\bar x^T,\frac{y_1^T}{T},\bar u^T,\bar p_x^T,\bar p_y^T)_{T\geq T_0}$.
Considering the optimal solution $(x^{T_k}(\cdot),y^{T_k}(\cdot),u^{T_k}(\cdot))$ of \eqref{syst1}-\eqref{syst2}-\eqref{terminalconditions_12}-\eqref{syst12_Omega}-\eqref{mincost1}, let us now exploit the strict dissipativity property to derive the following result. 

\begin{lemma}\label{lem_strict1}
Let $\varphi:\N^*\rightarrow\N^*$ be a function such that $\varphi(k)\leq T_k$ for every $k\in\N^*$ and $\varphi(k)\rightarrow+\infty$ as $k\rightarrow+\infty$. Considering a subsequence of $(T_k)_{k\in\N^*}$ if necessary, there exists a measurable subset $C_\varphi\subset[0,1]$ of full Lebesgue measure such that, for every $s\in C_\varphi$,
\begin{eqnarray}
&&
x^{T_k}(\varphi(k)s)\rightarrow\bar x^\infty, \qquad
u^{T_k}(\varphi(k)s)\rightarrow\bar u^\infty, \qquad
\frac{y^{T_k}(\varphi(k)s)}{\varphi(k)}\rightarrow s\bar d^\infty , \label{cvs1}\\
&& x^{T_k}(T_k-\varphi(k)s)\rightarrow\bar x^\infty, \qquad
u^{T_k}(T_k-\varphi(k)s)\rightarrow\bar u^\infty, \qquad
\frac{y^{T_k}(T_k-\varphi(k)s)}{T_k-\varphi(k)s}\rightarrow s\bar d^\infty , \label{cvs2}
\end{eqnarray}
as $k\rightarrow +\infty$.
\end{lemma}

Of course, this lemma can be applied with $\varphi(k)=T_k$, but as we are going to see it will be of interest to apply it with functions $\varphi$ such that $\varphi(k)=\mathrm{o}(T_k)$. A difficulty will arise from the fact that the full Lebesgue measure set $C_\varphi$ depends on $\varphi$.

\begin{proof}
The partial $x$-strict dissipativity inequality \eqref{dissip} given by Assumption \ref{H2_dissip}, applied to $(x^{T_k}(\cdot),u^{T_k}(\cdot))$ on the time interval $[0,\varphi(k)]$, implies that
\begin{multline}\label{x-strict_dissip}
f^0(\bar x^{T_k},\bar u^{T_k}) \leq \frac{1}{\varphi(k)}\int_0^{\varphi(k)} f^0(x^{T_k}(t),u^{T_k}(t))\, dt + \frac{S^{T_k}(x^{T_k}(0))-S^{T_k}(x^{T_k}(\varphi(k)))}{\varphi(k)} \\
- \frac{1}{\varphi(k)} \int_0^{\varphi(k)} \alpha\left( \Vert x^{T_k}(t)-\bar x^{T_k}\Vert, \Vert u^{T_k}(t)-\bar u^{T_k}\Vert \right) dt .
\end{multline}
Let us establish \eqref{cvs1}. To establish \eqref{cvs2}, the method is the same by replacing the time interval $[0,\varphi(k)]$ with $[T_k-\varphi(k),T_k]$ in \eqref{x-strict_dissip}.

Let us prove that $\frac{1}{\varphi(k)} \int_0^{\varphi(k)} \alpha\left( \Vert x^{T_k}(t)-\bar x^{T_k}\Vert, \Vert u^{T_k}(t)-\bar u^{T_k}\Vert \right) dt \rightarrow 0$ as $k\rightarrow +\infty$. 
By contradiction, assume that there exists $\eta>0$ and a (sub)sequence of integers $k\rightarrow+\infty$ such that 
\begin{equation}\label{costeta}
\frac{1}{\varphi(k)} \int_0^{\varphi(k)} \alpha\left( \Vert x^{T_k}(t)-\bar x^{T_k}\Vert, \Vert u^{T_k}(t)-\bar u^{T_k}\Vert \right) dt \geq\eta
\end{equation}
for every $k$ large enough.
The cost on the time interval $[0,\varphi(k)]$ of the trajectory $(\tilde x^{T_k}(\cdot),\tilde y^{T_k}(\cdot))$ solution of \eqref{syst1}-\eqref{syst2} given by Assumption \ref{H2_cont} is (by splitting the time interval $[0,\varphi(k)]$ and thus the integral with respect to the bounded times $t_1^{T_k}$, $t_2^{T_k}$ given by that assumption)
$$
\frac{1}{\varphi(k)}\int_0^{\varphi(k)} f^0(\tilde x^{T_k}(t),\tilde u^{T_k}(t))\, dt 
= f^0(\bar x^{T_k},\bar u^{T_k}) + \mathrm{o}(1)
$$
as $k\rightarrow+\infty$. 
Hence, for $k$ large enough, using \eqref{x-strict_dissip} and \eqref{costeta}, we have obtained a trajectory solution of \eqref{syst1}-\eqref{syst2} whose cost is (strictly) less than the cost of the optimal trajectory $(x^{T_k}(\cdot),y^{T_k}(\cdot),u^{T_k}(\cdot))$: this is a contradiction.

Using the change of variable $t=\varphi(k)s$, we have therefore proved that 
$$
\lim_{k\rightarrow+\infty}
\int_0^1 \alpha\left( \Vert x^{T_k}(\varphi(k)s)-\bar x^{T_k}\Vert, \Vert u^{T_k}(\varphi(k)s)-\bar u^{T_k}\Vert \right) ds = 0 .
$$ 
By the converse of the Lebesgue dominated theorem, taking a subsequence if necessary, for almost every $s\in[0,1]$, $x^{T_k}(\varphi(k)s)-\bar x^{T_k}\rightarrow 0$ and $u^{T_k}(\varphi(k)s)-\bar u^{T_k}\rightarrow 0$ as $k\rightarrow +\infty$, and thus $x^{T_k}(\varphi(k)s)\rightarrow\bar x^\infty$ and $u^{T_k}(\varphi(k)s)\rightarrow\bar u^\infty$ as $k\rightarrow +\infty$.
It remains to prove that $\frac{y^{T_k}(\varphi(k)s)}{\varphi(k)}\rightarrow s\bar d^\infty$.
We set, for almost every $s\in[0,1]$,
$$
\hat x_k(s) = x^{T_k}(\varphi(k)s), \qquad \hat y_k(s) =  y^{T_k}(\varphi(k)s), \qquad\hat u_k(s) = u^{T_k}(\varphi(k)s)  .
$$
We already know that $(\hat x_k(\cdot),\hat u_k(\cdot))$ converges almost everywhere to $(\bar x^\infty,\bar u^\infty)$.
Moreover, we have
$$
\frac{1}{\varphi(k)} \frac{d}{ds} \hat y_k(s) = g(\hat x_k(s),\hat u_k(s)),\qquad \frac{1}{\varphi(k)}\hat y_k(0)=\frac{y_0}{\varphi(k)},
$$
and thus $\frac{1}{\varphi(k)}\hat y_k(s)\rightarrow\bar y(s)$ almost everywhere on $[0,1]$ with $\bar y(s)=s\, g(\bar x^\infty,\bar u^\infty)$ and $g(\bar x^\infty,\bar u^\infty)=\bar d^\infty$.
The lemma is proved.
\end{proof}

\paragraph{First improvement by bootstrap.}
Lemma \ref{lem_strict1} alone does not imply that, for instance, $t\mapsto x^{T_k}(t)$ converges to $\bar x^\infty$ uniformly on $[\varepsilon,T-\varepsilon]$ for any $\varepsilon>0$ fixed (one could a priori have ``sliding bumps"). This much stronger convergence property is however true and is obtained by bootstrapping again with the estimates \eqref{linear_turnpike_estimates_prop} given by Proposition \ref{prop_sensitivity_dyn}.
The main restriction in Lemma \ref{lem_strict1} comes from the full Lebesgue measure subset $C_\varphi$ depending on $\varphi$; we start with the following lemma, showing that we can take $s=1$ in the convergence of $x^{T_k}(\cdot)$, independently of $\varphi$.

\begin{lemma}\label{lemsansC}
Let $\varphi:\N^*\rightarrow\N^*$ be a function such that $\varphi(k)\leq T_k/2$ for every $k\in\N^*$ and $\varphi(k)\rightarrow+\infty$ as $k\rightarrow+\infty$. Considering a subsequence of $(T_k)_{k\in\N^*}$ if necessary, we have
\begin{equation*}
\begin{split}
& x^{T_k}(\varphi(k))\rightarrow\bar x^\infty, \qquad x^{T_k}(T_k-\varphi(k))\rightarrow\bar x^\infty, \\
& p_x^{T_k}(\varphi(k))\rightarrow\bar p_x^\infty, \qquad p_x^{T_k}(T_k-\varphi(k))\rightarrow\bar p_x^\infty ,
\end{split}
\end{equation*}
as $k\rightarrow+\infty$.
\end{lemma}

\begin{proof}
By Lemma \ref{lem_strict1}, there exist $s_0,s_1\in(0,1/4)$ such that
\begin{eqnarray*}
&&
x^{T_k}(\varphi(k)s_0)\rightarrow\bar x^\infty, \qquad
\frac{y^{T_k}(\varphi(k)s_0)}{\varphi(k)s_0}\rightarrow \bar d^\infty , \\
&& x^{T_k}(T_k-\varphi(k)s_1)\rightarrow\bar x^\infty, \qquad
\frac{y^{T_k}(T_k-\varphi(k)s_1)}{T_k-\varphi(k)s_1}\rightarrow \bar d^\infty 
\end{eqnarray*}
as $k\rightarrow +\infty$. Given any $k\in\N^*$, setting
\begin{equation*}
\begin{split}
& t_0^k=\varphi(k)s_0, \qquad t_1^k = T_k-\varphi(k)s_1, \\
& x_0^k=x^{T_k}(t_0), \qquad x_1^k=x^{T_k}(t_1),
\qquad \eta_0^k=\frac{y^{T_k}(t_0)}{t_0}, \qquad \eta_1^k=\frac{y^{T_k}(t_1)}{t_1} ,
\end{split}
\end{equation*}
by the above convergence properties, for $k$ large enough we have $x_0^k\simeq\bar x^\infty$, $x_1^k\simeq\bar x^\infty$, $\eta_0^k\simeq\bar d^\infty$ and $\eta_1^k\simeq\bar d^\infty$, and we also note that $t_1^k-t_0^k = T_k-\varphi(k)(s_0+s_1) \rightarrow+\infty$ as $k\rightarrow +\infty$.
Hence, for $k$ large enough we have $x_0^k,x_1^k\in B(\bar x^\infty,\varepsilon_0)$ and $\eta_0^k,\eta_1^k\in B(\bar d^\infty,\varepsilon_0)$, with the notations of Proposition \ref{prop_sensitivity_dyn}: this is at this step that we use in a crucial way the fact, established in Proposition \ref{prop_sensitivity_dyn}, that $\varepsilon_0$ does not depend on $t_0,t_1$ whenever $t_1-t_0$ is large enough.

Being the restriction of an optimal solution, by the dynamic programming principle the triple $(x^{T_k}(\cdot),y^{T_k}(\cdot),u^{T_k}(\cdot))$ restricted to the time interval $[t_0^k,t_1^k]$ must therefore be the unique locally optimal solution of the auxiliary optimal control problem $\mathcal{P}(t_0^k,t_1^k,x_0^k,x_1^k,\eta_0^k,\eta_1^k)$ considered in Section \ref{sec_auxiliary}.

 It follows from Proposition \ref{prop_sensitivity_dyn} that, on the time interval $[t_0^k,t_1^k] = [\varphi(k)s_0, T_k-\varphi(k)s_1]$, the optimal solution $(x^{T_k}(\cdot),y^{T_k}(\cdot),u^{T_k}(\cdot))$ has a normal extremal lift and that $u^{T_k}(\cdot)$ is $C^1$. But, by Assumption \ref{H2_unique1}, $(x^{T_k}(\cdot),y^{T_k}(\cdot),u^{T_k}(\cdot))$ has a unique extremal lift, which is moreover normal, with adjoint vectors $p_x^{T_k}(\cdot)$ and $p_y^{T_k}$. Hence, these adjoint vectors must coincide with those given by Proposition \ref{prop_sensitivity_dyn}. Therefore, still by Proposition \ref{prop_sensitivity_dyn}, if $k$ is large enough then we have also $p_x^{T_k}(\varphi(k)s_0)\simeq\bar p_x^\infty$, $p_x^{T_k}(T_k-\varphi(k)s_1)\simeq\bar p_x^\infty$ and $p_y^{T_k}\simeq\bar p_y^\infty$. Note that the latter fact, on the adjoint vectors, is far from obvious and follows from the quite difficult Proposition \ref{prop_sensitivity_dyn}, adequately combined with the previous analysis.

At this step, 
we have therefore obtained that, for every $k$ large enough,
\begin{equation*}
\begin{split}
&\Vert x^{T_k}(\varphi(k)s_0)-\bar x^\infty\Vert\leq\varepsilon_0, \qquad
\Vert x^{T_k}(T_k-\varphi(k)s_1)-\bar x^\infty\Vert\leq\varepsilon_0, \\[1mm]
&\Vert p_x^{T_k}(\varphi(k)s_0)-\bar p_x^\infty\Vert\leq\varepsilon_0, \qquad
\Vert p_x^{T_k}(T_k-\varphi(k)s_1)-\bar p_x^\infty\Vert\leq\varepsilon_0 . 
\end{split}
\end{equation*}
Hence, applying 
Proposition \ref{prop_sensitivity_dyn} on the time interval $[\varphi(k)s_0,T_k-\varphi(k)s_1]$, we infer from \eqref{linear_turnpike_estimates_prop} (here, we bootstrap again) that
\begin{multline*}
\Vert x^{T_k}(t)-\bar x^\infty\Vert + \Vert p_x^{T_k}(t)-\bar p_x^\infty\Vert
\leq \ C\varepsilon_0\left( e^{-\nu(t-\varphi(k)s_0)} + e^{-\nu(T_k-\varphi(k)s_1-t)} \right) \\
+ \frac{C}{T_k-\varphi(k)(s_0+s_1)} \Big( \varepsilon_0 + \Vert y^{T_k}(T_k-\varphi(k)s_1) - y^{T_k}(\varphi(k)s_0) - (T_k-\varphi(k)(s_0+s_1)) \bar d^\infty\Vert \Big)
\end{multline*}
for every $t\in[\varphi(k)s_0,T_k-\varphi(k)s_1]$.
The conclusion stated in the lemma now easily follows.
\end{proof}

Of course, at the end of the above proof, we could already infer a stronger conclusion than that stated in Lemma \ref{lemsansC}. But before that, we are first going to establish the following result, providing an additional improvement. 

\begin{lemma}\label{lem_tau0tau1}
For every $\varepsilon\in(0,1)$ there exist $\tau_0,\tau_1>0$ and $k_0\in\N^*$ such that
\begin{equation*}
\begin{split}
& \Vert x^{T_k}(\tau_0)-\bar x^\infty\Vert < \varepsilon, \qquad \Vert x^{T_k}(T_k-\tau_1)-\bar x^\infty\Vert < \varepsilon , \\ 
& \Vert p_x^{T_k}(\tau_0)-\bar p_x^\infty\Vert < \varepsilon, \qquad \Vert p_x^{T_k}(T_k-\tau_1)-\bar p_x^\infty\Vert < \varepsilon , 
\end{split}
\end{equation*}
for every $k\geq k_0$. 
\end{lemma}

\begin{proof}
Let us prove the existence of $\tau_0$ (and only for the $x$ component), the property for $\tau_1$ being similar.
By contradiction, assume that there exists $\varepsilon>0$ such that, for every $\tau>0$, for every $k\in\N^*$, there exists $j(k)\geq k$ such that $\Vert x^{T_{j(k)}}(\tau)-\bar x^\infty\Vert\geq\varepsilon$. Applying this property to $\tau=k$, we obtain that, for every $k\in\N^*$, there exists $j(k)\geq k$ such that $\Vert x^{T_{j(k)}}(k)-\bar x^\infty\Vert\geq\varepsilon$. Without loss of generality, we assume that $k\mapsto j(k)$ is increasing. Then, there exists a function $\varphi:\N^*\rightarrow\N^*$ such that $\phi(j(k))=k$ for every $k\in\N^*$. By Lemma \ref{lemsansC}, we have $x^{T_{j(k)}}(k)=x^{T_{j(k)}}(\phi(j(k)))\rightarrow\bar x^\infty$. This is a contradiction.
\end{proof}

\paragraph{Second (and final) improvement by bootstrap.}
With respect to Lemma \ref{lemsansC}, in Lemma \ref{lem_tau0tau1} we have obtained that the trajectory $(x^{T_k}(\cdot),p^{T_k}(\cdot))$, starting at $t=0$ (resp., starting backward in time at $t=T_k$), enters an $\varepsilon_0$-neighborhood of $(\bar x^\infty,\bar p_x^\infty)$ in finite time $\tau_0$ (resp., $\tau_1$).
This is instrumental to conclude, what we now do.

Applying again Proposition \ref{prop_sensitivity_dyn}, but now, on the time interval $[\tau_0,T_k-\tau_1]$, we infer from \eqref{linear_turnpike_estimates_prop} (again a bootstrap) that, for $k$ large enough,
\begin{equation*}
\begin{split}
& \Vert x^{T_k}(t)-\bar x^\infty\Vert + \Vert p^{T_k}_x(t)-\bar p_x^\infty\Vert \leq C\varepsilon_0 \left( e^{-\nu (t-\tau_0)} + e^{-\nu(T_k-\tau_1-t)} \right) \\
& \qquad\qquad\qquad\qquad + \frac{C}{T_k-\tau_0-\tau_1} \left( \varepsilon_0 + \Vert y^{T_k}(T_k-\tau_1)-y^{T_k}(\tau_0)-(T_k-\tau_0-\tau_1)\bar d^\infty\Vert \right) ,  \\
& \Vert y^{T_k}(t)-t\bar d^\infty\Vert \leq C\left( \varepsilon_0 + \Vert y^{T_k}(\tau_0)-\tau_0\bar d^\infty\Vert + \Vert y^{T_k}(T_k-\tau_1)-y^{T_k}(\tau_0)-(T_k-\tau_0-\tau_1)\bar d^\infty\Vert \right), \\
& \Vert p_y^{T_k}-\bar p_y^\infty\Vert \leq \frac{C}{T_k-\tau_0-\tau_1} \left( \varepsilon_0 + \Vert y^{T_k}(T_k-\tau_1)-y^{T_k}(\tau_0)-(T_k-\tau_0-\tau_1)\bar d^\infty\Vert \right) ,
\end{split}
\end{equation*}
for every $t\in[\tau_0,T_k-\tau_1]$.
Since $y^{T_k}(0)=y_0$ and $\dot y^{T_k}(t)=g(x^{T_k}(t),u^{T_k}(t))$ is bounded (see Assumption \ref{H2_unique1}), the sequence $(y^{T_k}(\tau_0)-\tau_0\bar d^\infty)_{k\in\N^*}$ is bounded. Besides, since
$$
y^{T_k}(T_k-\tau_1) = y^{T_k}(T_k) - \int_{T_k-\tau_1}^{T_k} g(x^{T_k}(t),u^{T_k}(t))\, dt ,
$$
since $y^{T_k}(T_k) = y_1^{T_k}$ and since $y_1^{T_k} = y_0+T_k\bar d^\infty$ (see \eqref{y1T}), we infer that the sequence $(y^{T_k}(T_k-\tau_1)-y^{T_k}(\tau_0)-(T_k-\tau_0-\tau_1)\bar d^\infty)_{k\in\N^*}$ is bounded.
Therefore, taking a larger constant $C$ if necessary, we obtain
\begin{equation}\label{estimlinTk}
\begin{split}
& \Vert x^{T_k}(t)-\bar x^\infty\Vert + \Vert p^{T_k}_x(t)-\bar p_x^\infty\Vert \leq C \left( \frac{1}{T_k} + e^{-\nu t} + e^{-\nu(T_k-t)} \right) \\
& \Vert y^{T_k}(t)-t\bar d^\infty\Vert \leq C, \qquad
\Vert p_y^{T_k}-\bar p_y^\infty\Vert \leq \frac{C}{T_k}  \qquad\qquad \forall  t\in[0,T_k] .
\end{split}
\end{equation}

\paragraph{Conclusion.}
Let us summarize what we have done. Given any closure point $(\bar x^\infty,\bar d^\infty,\bar u^\infty,\bar p_x^\infty,\bar p_y^\infty)$ of the bounded family $(\bar x^T,\frac{y_1^T}{T},\bar u^T,\bar p_x^T,\bar p_y^T)_{T\geq T_0}$ and any sequence $T_k\rightarrow+\infty$ of positive real numbers such that $(\bar x^{T_k},\frac{y_1^{T_k}}{T_k},\bar u^{T_k},\bar p_x^{T_k},\bar p_y^{T_k})$ converges to $(\bar x^\infty,\bar d^\infty,\bar u^\infty,\bar p_x^\infty,\bar p_y^\infty)$ (see \eqref{limit_tuple}), we have established the estimates \eqref{estimlinTk}, and thus, using Lemma \ref{CV_rate_static} and the triangular inequality, and taking a larger $C$ if necessary,
\begin{equation*}
\begin{split}
& \Vert x^{T_k}(t)-\bar x^{T_k}\Vert + \Vert p^{T_k}_x(t)-\bar p_x^{T_k}\Vert \leq C \left( \frac{1}{T_k} + e^{-\nu t} + e^{-\nu(T_k-t)} \right) \\
& \Vert y^{T_k}(t)-t \frac{y_1^{T_k}}{T_k}\Vert \leq C, \qquad
\Vert p_y^{T_k}-\bar p_y^{T_k}\Vert \leq \frac{C}{T_k}  \qquad\qquad \forall  t\in[0,T_k] .
\end{split}
\end{equation*}
It follows that
\begin{equation*}
\begin{split}
& \Vert x^T(t)-\bar x^T\Vert + \Vert p^T_x(t)-\bar p_x^T\Vert \leq C \left( \frac{1}{T} + e^{-\nu t} + e^{-\nu(T-t)} \right) \\
& \Vert y^T(t)-t \frac{y_1^T}{T}\Vert \leq C, \qquad
\Vert p_y^T-\bar p_y^T\Vert \leq \frac{C}{T}  \qquad\qquad \forall  t\in[0,T] .
\end{split}
\end{equation*}
Indeed, here, we use the following obvious fact%
\footnote{It suffices to prove that $0$ is the only possible closure point of the (bounded) function $h_1-h_2$ at $+\infty$. Let $a\in\R$ be such a closure point: there exists a sequence $x_k\rightarrow+\infty$ such that $h_1(x_k)-h_2(x_k)\rightarrow a$. Since the sequence $(h_2(x_k))_{k\in\N}$ is bounded, there exists a subsequence such that $h_2(x_{k_p})\rightarrow 0$. Hence, by assumption, $h_1(x_{k_p})-h_2(x_{k_p})\rightarrow 0$, and thus $a=0$.}: 
let $h_1$ and $h_2$ be two bounded functions on $\R$ such that, for every sequence $x_k\rightarrow+\infty$ for which $h_2(x_k)$ converges, we have $h_1(x_k)-h_2(x_k)\rightarrow 0$ as $k\rightarrow+\infty$; then $h_1(x)-h_2(x)\rightarrow 0$ as $x\rightarrow+\infty$.  

This gives \eqref{turnpikelin} and concludes the proof of Theorem \ref{turnpike_thm}.

\section{Examples}\label{sec_examples}
In this section, we illustrate the linear turnpike phenomenon on several practical examples, with numerical simulations.
Hereafter, we compute numerically the optimal solutions either with a direct method, using \texttt{AMPL} (see \cite{AMPL}) combined with \texttt{IpOpt} (see \cite{IPOPT}), following a full discretization of the problem, or with the variant of the shooting method introduced in \cite{TZ}. In both cases, we initialize the numerical method at the turnpike, thus guaranteeing its convergence. 

\subsection{The Zermelo problem}
The Zermelo problem is a famous optimal control problem, often used as a simple model or exercise to illustrate theory or practice of optimal control (see \cite{trelat_book}).
In this problem, we consider a boat moving with constant speed along a river of constant width $\ell$, in which there is a current $c(x)$. The movement of the center of mass of the boat is governed by the control system
\begin{equation*}
\begin{split}
\dot{x}(t)&=v(t)\sin u(t),\qquad\qquad\qquad x(0)=0,\\
\dot{y}(t)&=v(t)\cos u(t)+c(x(t)),\qquad y(0)=0,
\end{split}
\end{equation*}
where the controls are the modulus of the speed $v(t)\in[0,v_{\mathrm{max}}]$ and the angle $u(t)\in\R$ of the axis of the boat with respect to the axis $(0y)$ (see Figure \ref{fig_zermelo}). 
\begin{figure}[h]
\begin{center}
\includegraphics[width=9cm]{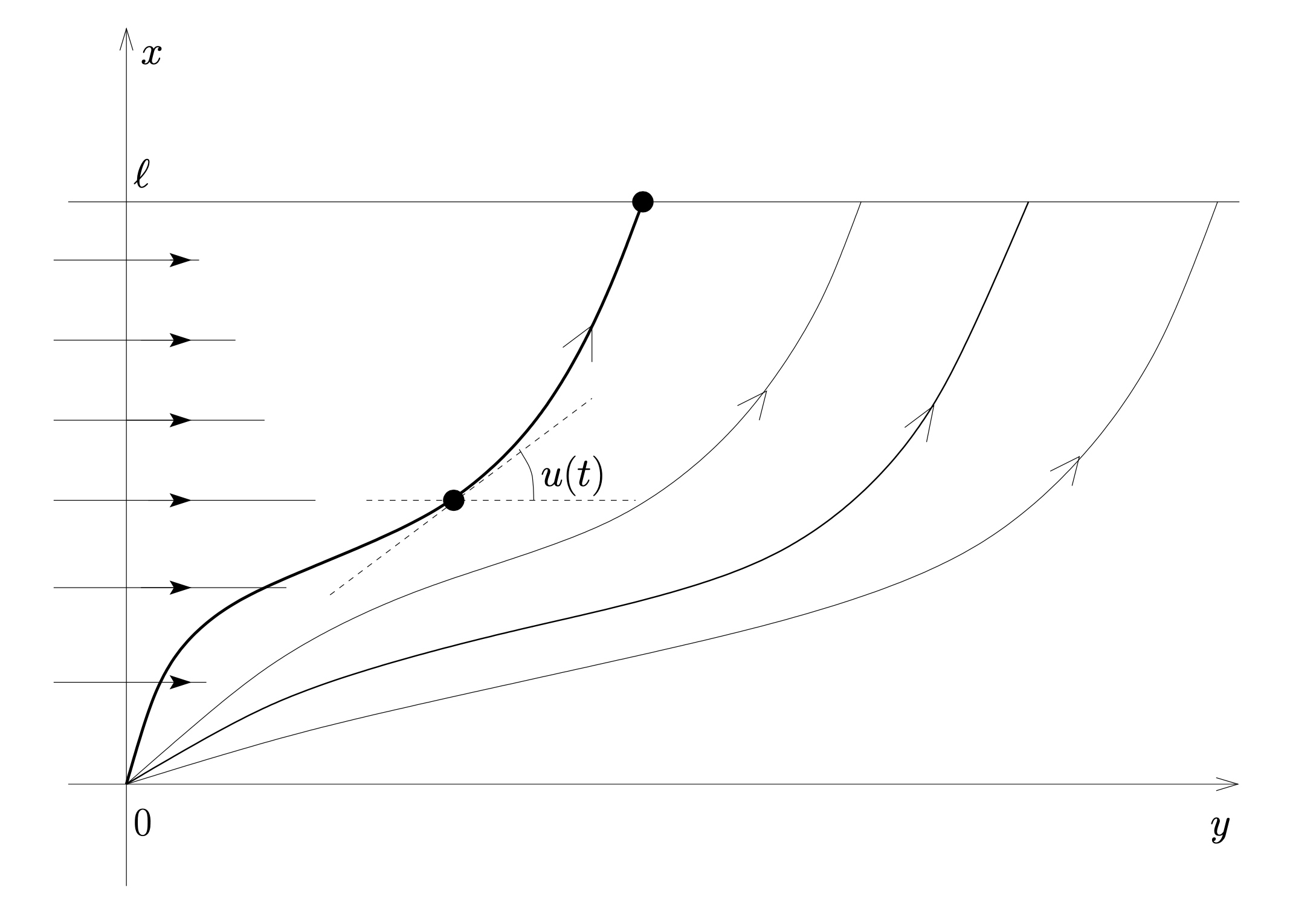}
\end{center}
\caption{Zermelo problem.}\label{fig_zermelo}
\end{figure}

Many variants of terminal conditions and of cost functionals can be considered. Here, in order to illustrate the linear turnpike phenomenon, given some $T>0$ fixed (to be chosen large enough), we consider the optimal control problem of steering in time $T$ this control system from the initial point $(0,0)$ to the final point $(\ell,5T)$ (on the opposite river), by minimizing the cost functional $\int_0^T v(t)^2\, dt$. 

We take $L=2$, $v_{\mathrm{max}}=1.1$ and $c(x)=3+x(L-x)$. The current in the river is so strong that we always have $\dot y(t)>0$, whatever the control may be; hence $y(\cdot)$ is bound to be increasing. 

The turnpike-static optimal control problem consists of minimizing $v^2$ under the constraints $v\sin u=0$ and $T(v\cos u+c(y))=5T$. It has the unique solution $\bar x^T=1$, $\bar y^T(t)=5t$, $\bar x^T=1$, $\bar u^T=0$, $\bar v^T=1$, $\bar p_x^T=0$, $\bar p_y^T=2$.

\begin{figure}[H]
\centerline{\includegraphics[width=17cm]{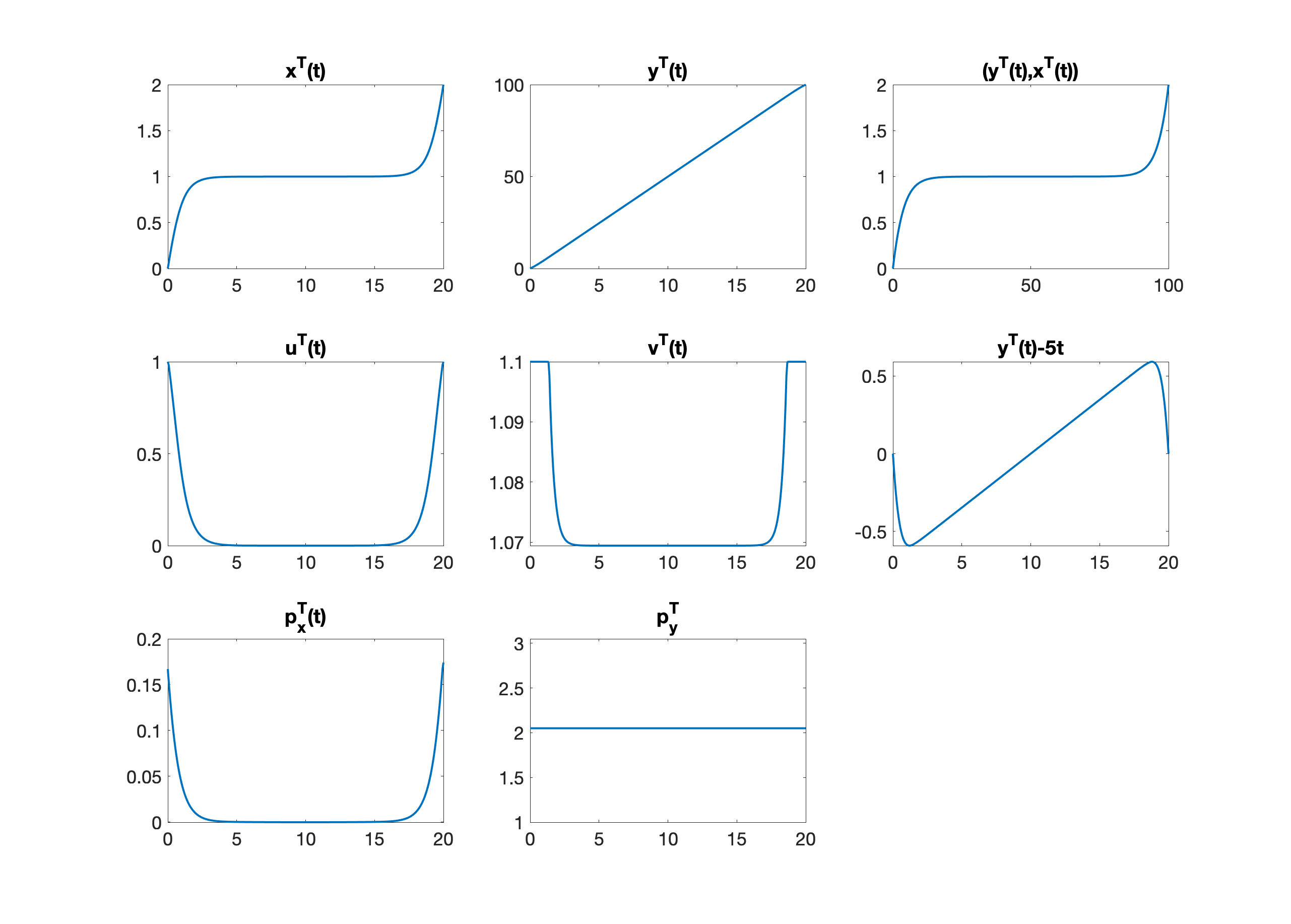}}
\caption{Numerical results for the Zermelo problem}\label{fig_zermelo}
\end{figure}

The numerical results for $T=20$ are reported on Figure \ref{fig_zermelo}.
We observe, as expected, that $\vert y^T(t)-\bar y^T(t)\vert$ is bounded by a constant. Besides, we observe that, in the long middle part of the time interval, $(x^T(t),u^T(t),p_x^T(\cdot))$ is exponentially close to $(\bar x^T,\bar u^T,\bar p_x^T)$ while $(v^T(t),p_y^T)$ is linearly close to $(\bar v^T,\bar p_y^T)$.

Note that the optimal control $v^T(\cdot)$ has two bang arcs along which $v^T(t)=v_{\mathrm{max}}=1.1$: a short one at the beginning, and another short one at the end.

\subsection{Optimal control model of a runner}
Consider the runner model developed in \cite{aft, aftalion_bonnans, AftalionTrelat_RSOS, AftalionTrelat_JOMB}, originating from \cite{keller1974optimal}:
\begin{equation}\label{equ}
\begin{array}{rlll}
\displaystyle\dot x(t) &=& v(t)  & x(0)=0, \quad x(t_f)=d \\[2mm]
\displaystyle\dot v(t) &=& \displaystyle-\frac{v(t)}{\tau}+f(t) & v(0)=v^0 \\[2mm]
\displaystyle\dot e(t) &=& \sigma 
-f(t)v(t) & e(0)=e^0,\quad e(t_f)=0,\quad e(t)\geq 0 \\[2mm]
\displaystyle\dot f(t) &=& \gamma \left( u(t) (\Fmax-f(t)) - f(t) \right) \quad & 0\leq f(t)\leq \Fmax 
\end{array} 
\end{equation}
where $d>0$ is the prescribed distance to run, $e^0>0$ is the initial energy, $\tau>0$ is the friction coefficient related to the runner's economy, $\sigma>0$ is a constant standing for the energetic equivalent of the oxygen uptake $VO_2$, $\Fmax>0$ is a threshold upper bound for the force $f(t)$, $\gamma>0$ is the time constant of motor activation and $u(t)\in [-M,M]$ is the neural drive which is the control, where $M>0$ is some control bound. 
Here, $x(t)$ is the distance travelled at time $t$ by the runner, $v(t)$ the velocity, $e(t)$ the anaerobic energy, and $f(t)$ the propulsive force per unit mass. 
Actually, the typical values of the initial energy $e_0$ are such that the energy variable $e(t)$ is decreasing.
The minimization criterion is
\begin{equation}\label{cost} 
\min \left( t_f+ \frac{\alpha}{2} \int_0^{t_f} u(t)^2\, dt \right)
\end{equation} 
where $\alpha>0$ is a fixed constant and the final time $t_f$ is free.

Although, in the optimal control problem \eqref{equ}-\eqref{cost}, $t_f$, $v(t_f)$, $f(0)$ and $f(t_f)$ are let free, the problem can be reparametrized by the distance $s$, with the change of variable $t'(s)=\frac{dt}{ds}=\frac{1}{v(s)}$. In terms of $s$, the optimal control problem is
\begin{equation*}
\begin{array}{ll}
\displaystyle v'(s) = -\frac{1}{\tau}+\frac{f(s)}{v(s)} & v(0)=v^0 \\[4mm]
\displaystyle e'(s) = \frac{\sigma}{v(s)} 
- f(s) & e(0)=e^0,\quad e(d)=0,\quad e(s)\geq 0 \\[4mm]
\displaystyle f'(s) = \frac{\gamma}{v(s)} \left( u(s) (\Fmax-f(s)) - f(s) \right)  \quad & 0\leq f(s)\leq\Fmax   \\[4mm]
\displaystyle \min  \int_0^d \frac{1}{v(s)} \left( 1 + \frac{\alpha}{2} u(s)^2 \right) ds &
\end{array}
\end{equation*}
and now fits in the general framework developed in this paper.

Solving the ``static" problem leads to $\bar e(t)=e_0+(\sigma-\bar f\bar v)t$ (we re-express it in function of $t$) where
$$
\bar v = \frac{\tau}{2d} \left( e_0 + \sqrt{ e_0^2+4\frac{\sigma d^2}{\tau}} \right) , \qquad
\bar f = \frac{\bar v}{\tau} = \frac{1}{2d} \left( e_0 + \sqrt{ e_0^2+4\frac{\sigma d^2}{\tau}} \right) , \qquad
\bar u = \frac{\bar f}{\Fmax-\bar f} .
$$
For the numerical simulations, we take $d=1500$, $\tau=0.932$, $\sigma=22$, $\alpha=10^{-5}$, $\Fmax=8$, $\gamma=0.0025$, $v_0=3$, $e_0=4651$. 
We obtain $\bar v=6.2$, $\bar f=6.65$, $\bar u=4.92$.
The distance run ($1500$ meters) is large enough so that we observe the linear turnpike phenomenon.
The numerical results are reported on Figure \ref{fig_runner}.

\begin{figure}[H]
\centerline{\includegraphics[width=16cm]{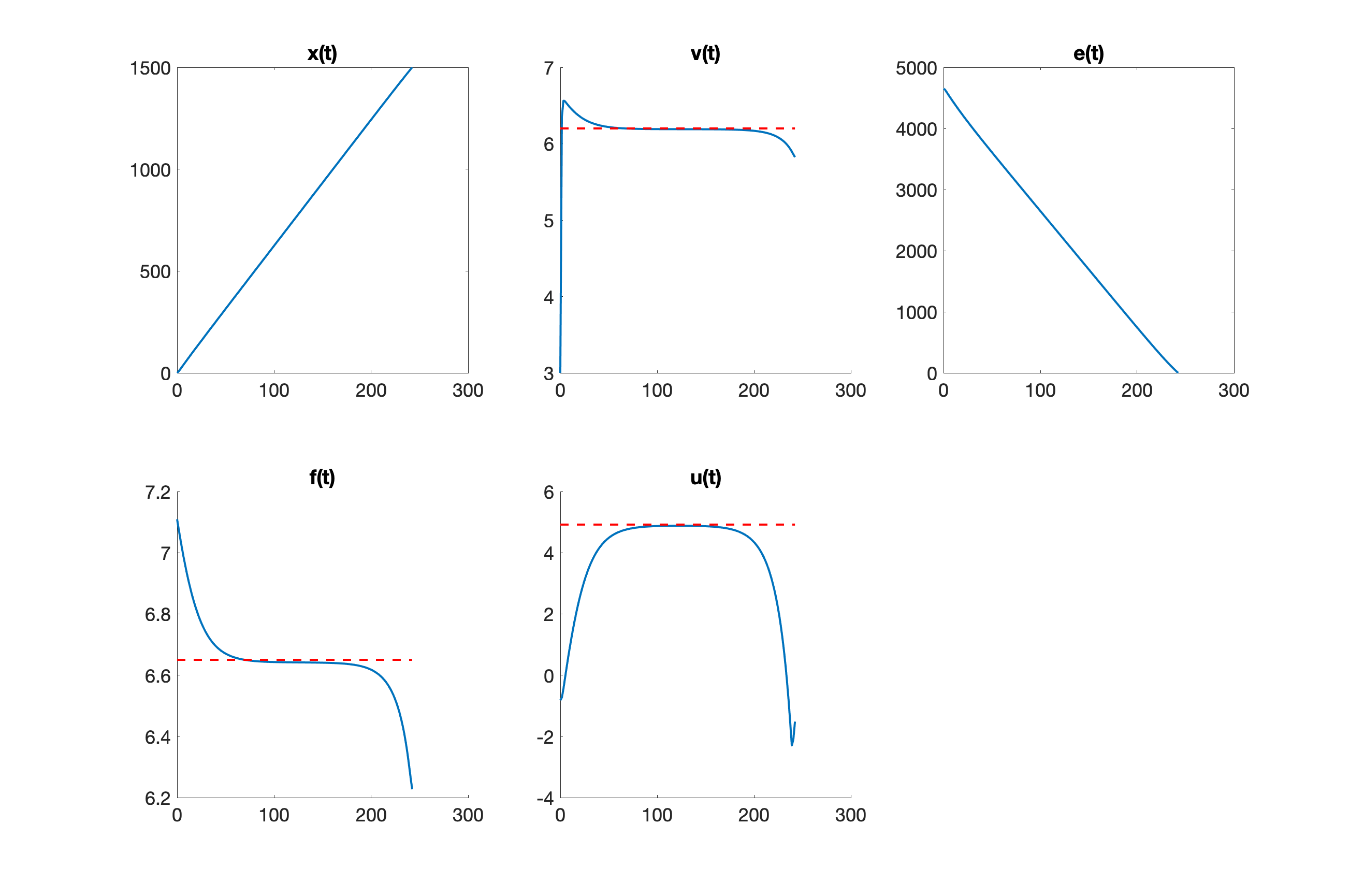}}
\caption{Numerical results for the runner problem}\label{fig_runner}
\end{figure}

This runner optimal control problem has actually been the initial point for the present paper, and the author warmly thanks Amandine Aftalion for having raised such an interesting problem. 
The linear turnpike phenomenon has been recently exploited in \cite{AftalionTrelat_JOMB}.

\appendix

\section{Appendix}\label{sec_app}
This appendix is devoted to illustrating the comments done in Remark \ref{rem_local}. We take a very simple example and we give numerical simulations showing the competition between two global turnpikes, or, at the level of the initialization of numerical methods, between local and global turnpikes.

We consider the one-dimensional optimal control problem 
\begin{equation}\label{xcube}
\begin{split}
& \dot x = -3x + 3x^3 + u, \qquad x(0)=x_0,\quad x(T)=x_f\textrm{ or free} \\
& \min \int_0^T \left(  (x(t)-x_d)^2 + (u(t)-u_d)^2   \right) dt
\end{split}
\end{equation}
The corresponding static problem is 
\begin{equation}\label{xcube_static}
\min_{(x,u)\ \mid\ -3x + 3x^3 + u=0} \left( (x-x_d)^2 + (u-u_d)^2\right) .
\end{equation}

\subsection{Competition between two (global) turnpikes}
We take $x_d=1$ and $u_d=3.47197$. The choice of $u_d$ is done so that the static problem \eqref{xcube_static} has two global minima, at $\bar x_1=-1.347372066$ and $\bar x_2=0.5939615956$, see Figure \ref{fig1}.

\begin{figure}[H]
\centerline{\includegraphics[width=8cm]{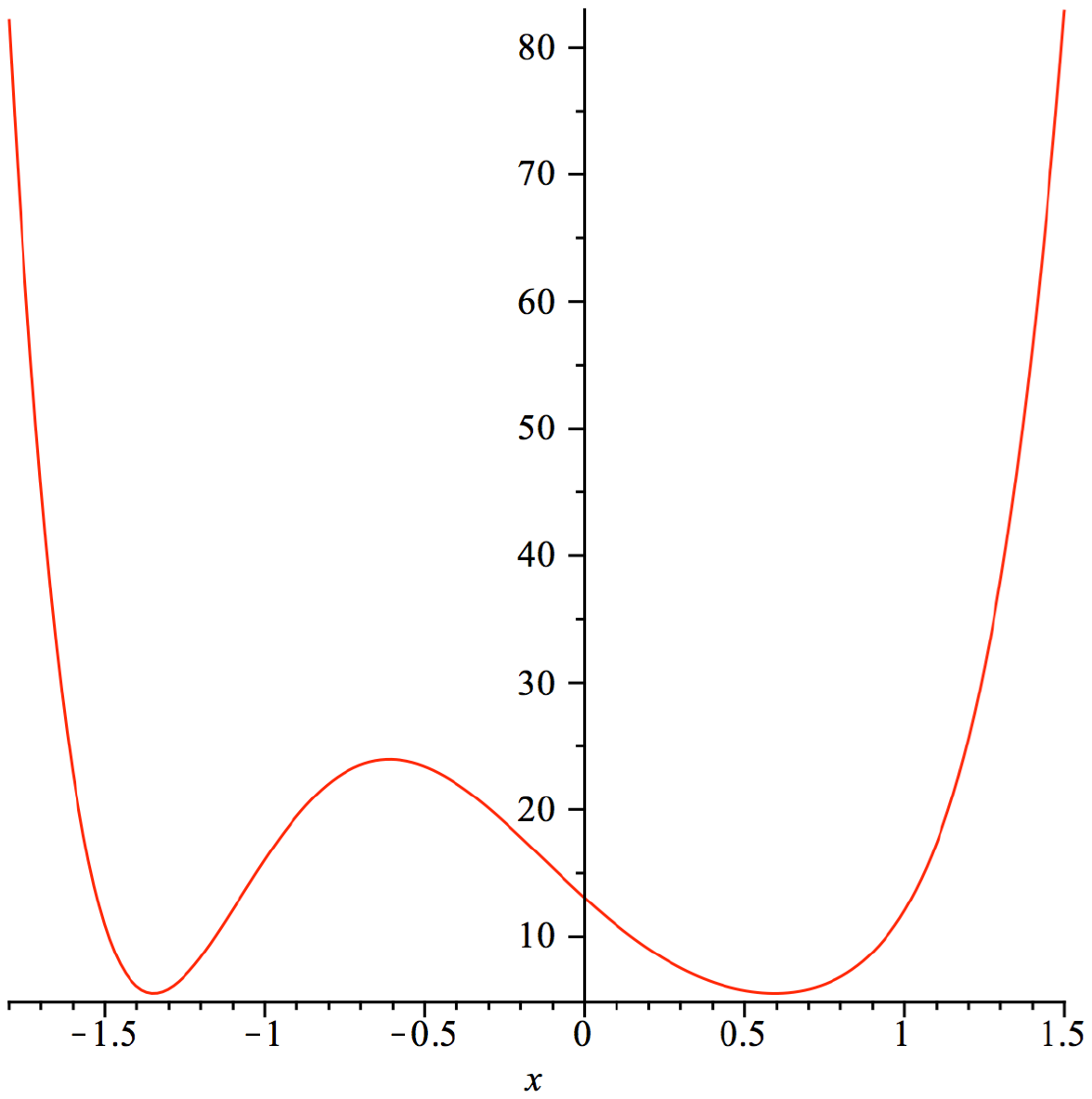}}
\caption{Plot of the function $x\mapsto (x-1)^2 + (3x - 3x^3-3.47197)^2$}
\label{fig1}
\end{figure}

On Figure \ref{fig23}, in dashed blue, we have computed the optimal trajectory with $x_0=-5$, $x_f=-1$, $T=10$: we observe a turnpike phenomenon around $\bar x_1$.
In solid red, we have computed the optimal trajectory with $x_0=2$, $x_f=1$, $T=10$: we observe a turnpike phenomenon around $\bar x_2$.

\begin{figure}[H]
\begin{center}
\includegraphics[width=11cm]{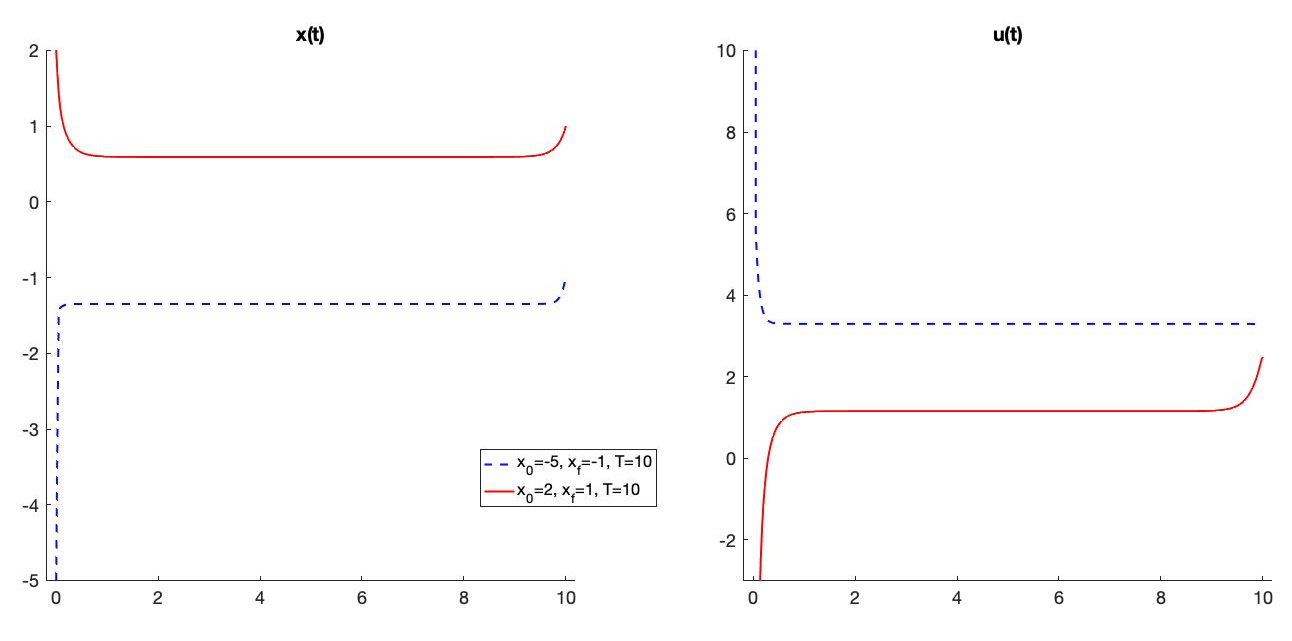}
\end{center}
\caption{Global turnpikes around $\bar x_1$ and $\bar x_2$.}
\label{fig23}
\end{figure}

The fact that the turnpike phenomenon is either around $\bar x_1$ or around $\bar x_2$ depends on the terminal conditions. For instance, if $x_0$ and $x_f$ are close to $\bar x_1$ (resp., $\bar x_2$) then the optimal trajectory will make a turnpike around $\bar x_1$ (resp., $\bar x_2$). But when the terminal conditions are farther, it is not clear to predict the behavior of the optimal trajectory.

Since the minima are global, we expect to observe a competition between both turnpikes, depending on the terminal conditions (see \cite{Rapaport2}).
Let us provide some numerical simulations illustrating this competition.
To facilitate the understanding, we consider the problem with $x_f$ free.
It is interesting to note that the numerical result strongly depends on the initialization of the numerical method.
Here, to compute numerically the optimal solutions of \eqref{xcube}, we use \texttt{AMPL} (see \cite{AMPL}) combined with \texttt{IpOpt} (see \cite{IPOPT}): the trajectory and the controls are discretized (the control is piecewise constant and the trajectory is piecewise linear, on a given subdivision that is chosen fine enough), and we initialize the trajectory with the same constant value over all the subdivision.
Then according to the value of this initialization, we can make emerge such or such turnpike, and all solutions are anyway optimal.

\begin{figure}[H]
\begin{center}
\includegraphics[width=11cm]{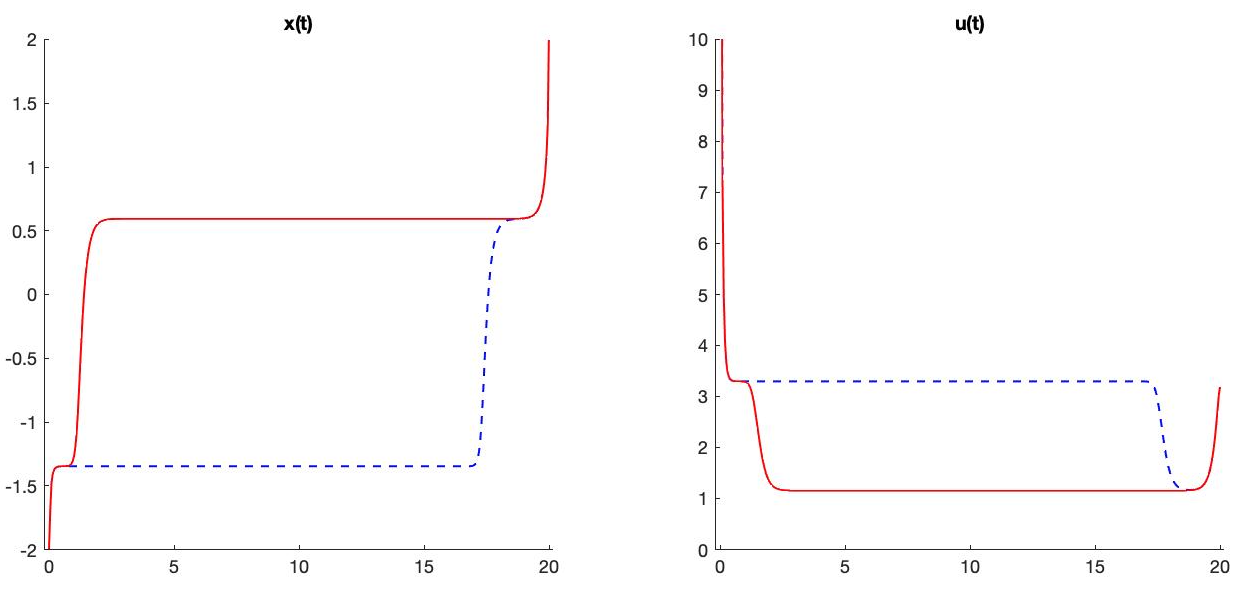}
\end{center}
\caption{$x_0=-2$, $x_f$ free, $T=20$.}
\label{fig45}
\end{figure}

On Figure \ref{fig45}, we have taken $T=20$, $x_0=-2$ (and $x_f$ is let free). In dashed blue, we have initialized the trajectory to the constant trajectory $\bar x_1$, and we then obtain an optimal trajectory which stays essentially near $\bar x_1=-1.347372066$, with a kind of ``hesitation" towards $\bar x_2=0.5939615956$ near the end.
In solid red, we have initialized the trajectory to the constant trajectory $\bar x_2$, and we obtain a trajectory staying essentially near $\bar x_2$, with a kind of ``hesitation" towards $\bar x_1$ near the beginning.
We stress that the two solutions are optimal: both have a cost $C\simeq 6.2822$.
We could make emerge other similar trajectories, which ``hesitate" between the two turnpikes.  All of them are optimal, or, at least, ``quasi-optimal" (there is a small error due to switches from one turnpike to the other).

\subsection{Local versus global turnpike}\label{app_local_versus_global}
We take $x_d=1$ and $u_d=1$. The choice of $u_d$ is now such that the static problem \eqref{xcube_static} has a unique global solution $\bar x = 0.781538640850898$, see Figure \ref{fig6}.
But it has also a local solution $\bar x_{loc}=-1.10551208794920$.

\begin{figure}[H]
\centerline{\includegraphics[width=8cm]{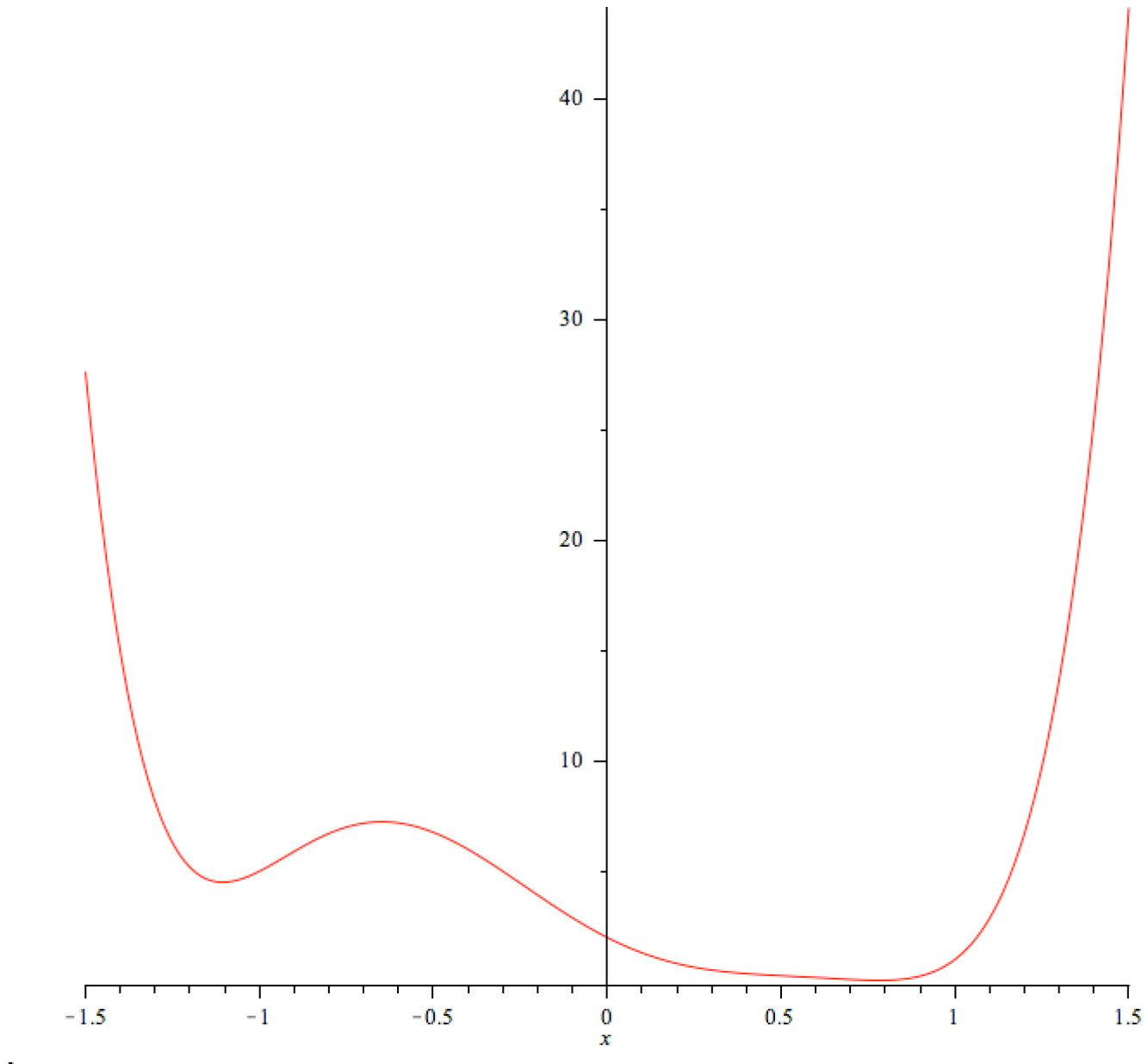}}
\caption{Plot of the function $x\mapsto (x-1)^2 + (3x - 3x^3-1)^2$}
\label{fig6}
\end{figure}

The global minimum $\bar x$ is a global turnpike, while the local minimum $\bar x$ is a local turnpike. In the numerical simulations, when performing either an optimization or a Newton method (shooting), we compute local solutions. Hence, we must expect that the numerical results depend on the initialization. To check global optimality, we have to compare the costs.

\begin{figure}[H]
\begin{center}
\includegraphics[width=12cm]{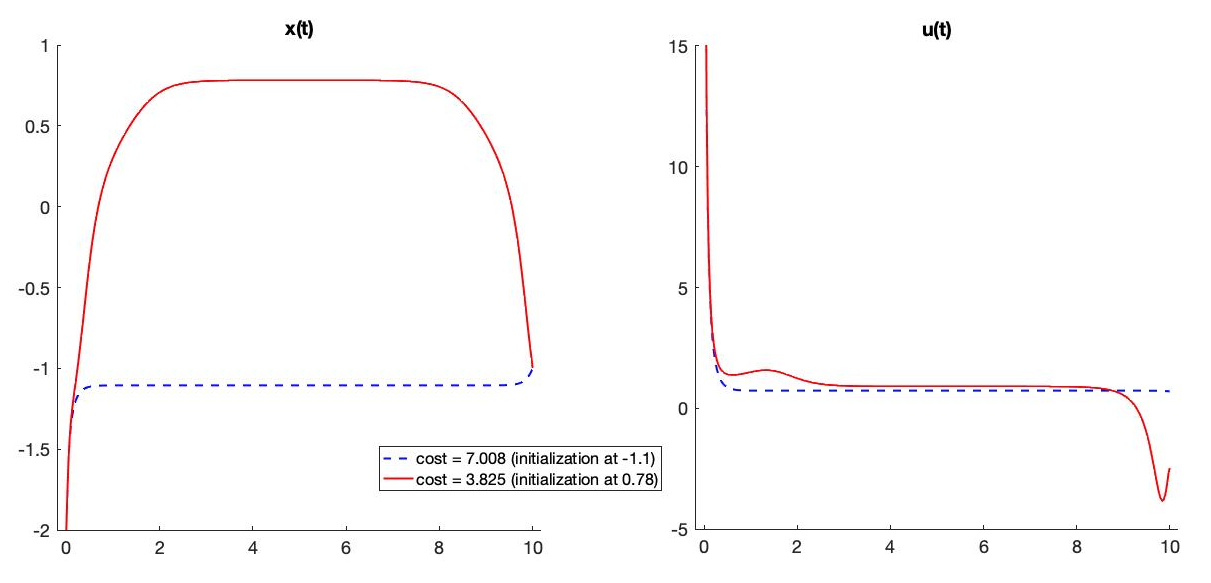}
\end{center}
\caption{$x_0=-2$, $x_f=-1$, $T=10$.}
\label{fig78}
\end{figure}

On Figure \ref{fig78}, in dashed blue, we have initialized the code with the constant trajectory $\bar x_{loc}$. We observe a turnpike around $\bar x_{loc}=-1.10551208794920$. The cost is $C\simeq 7.008$.
But this is a local turnpike only. This trajectory that we obtain is only locally optimal, and is not globally optimal.
In solid red, we have initialized the code with the constant trajectory $\bar x$. We observe a turnpike around $\bar x = 0.781538640850898$. The cost is $C\simeq 3.825$ and is lower than the one of the previous one.
Here, we have actually computed the globally optimal trajectory. This is the global turnpike.

\medskip

It is also interesting to see what happens if we take $T$ much smaller. Let us take $T=2$.
On Figure \ref{fig910}, in dashed blue, we have initialized the code with the constant trajectory $\bar x$. We observe a trend to the turnpike around $\bar x = 0.781538640850898$. The cost is $C\simeq 17.792$.
But this trajectory, now, is not globally optimal.
In solid red, we have initialized the code with the constant trajectory $\bar x_{loc}$. We observe a turnpike around $\bar x_{loc}=-1.347372066$. The cost is $C\simeq 16.322$.
It is the globally optimal trajectory.

\begin{figure}[H]
\begin{center}
\includegraphics[width=10cm]{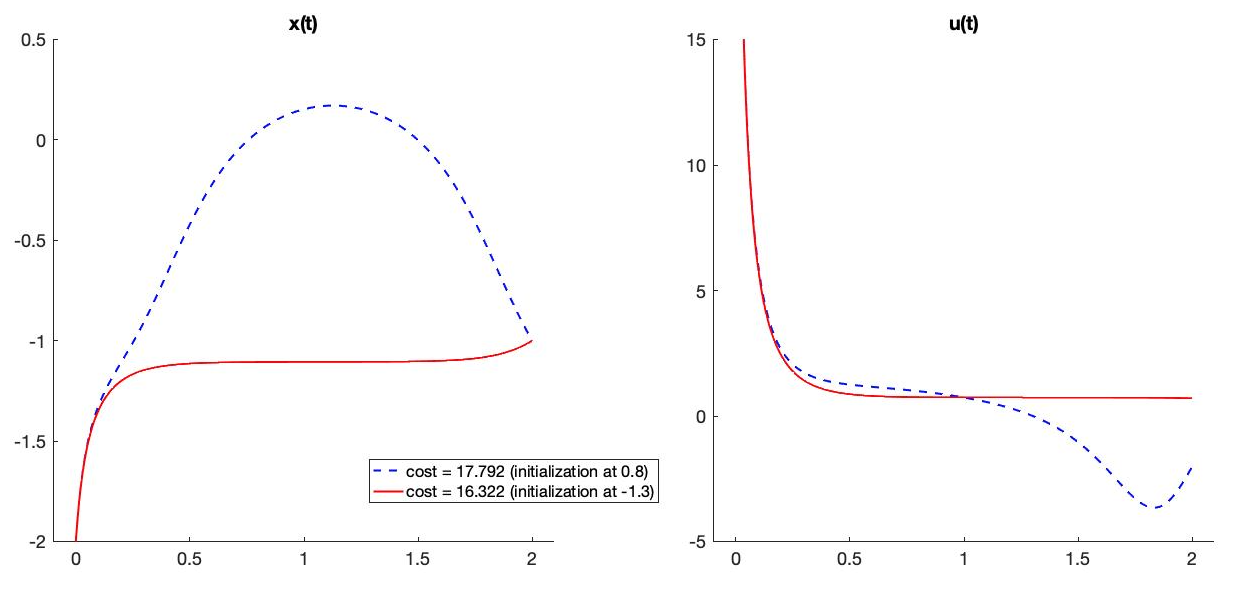} 
\end{center}
\caption{$x_0=-2$, $x_f=-1$, $T=2$.}
\label{fig910}
\end{figure}

We can search a time $2<T_c<10$ for which both previous initializations give equivalent turnpikes around $\bar x_{loc}$ and $\bar x$ (with same cost). This is done hereafter.
What is important is that, in large time $T$, we have indeed the global turnpike around $\bar x$.

\begin{figure}[H]
\begin{center}
\includegraphics[width=15cm]{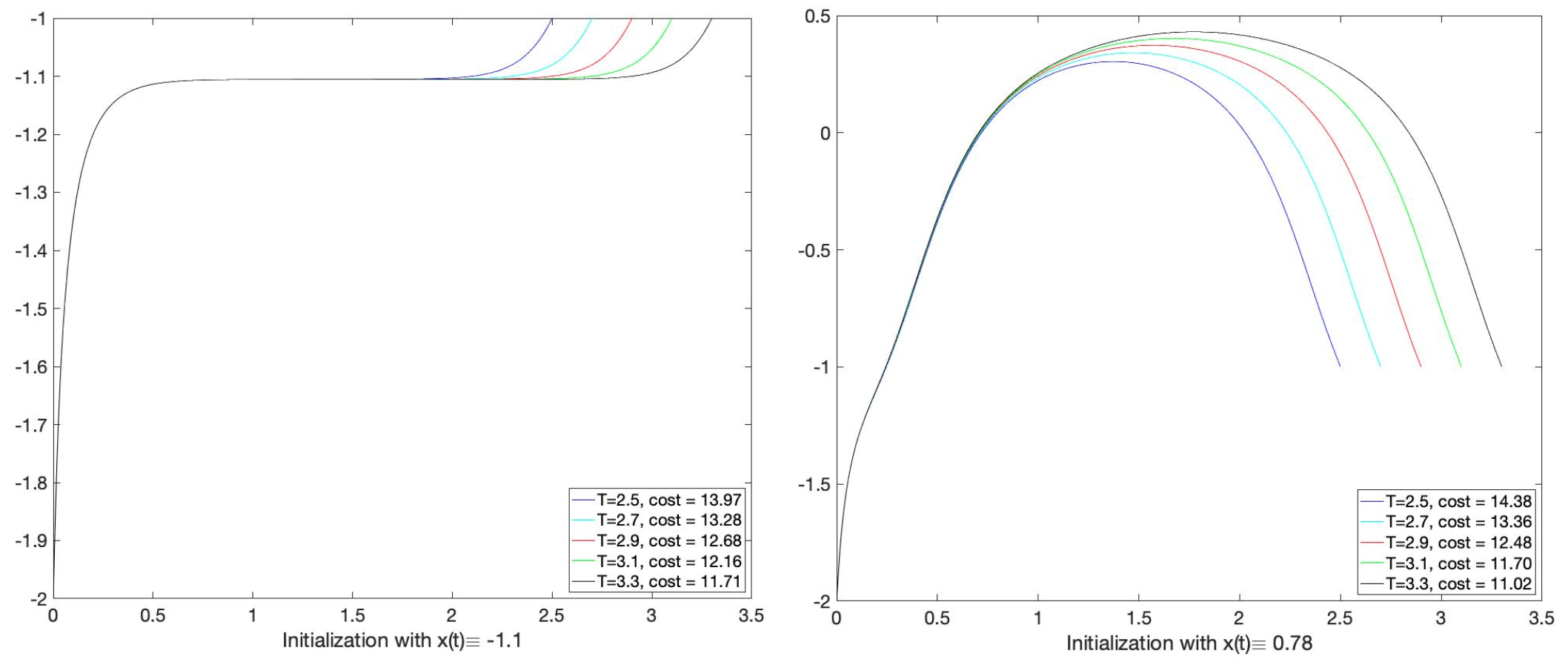} 
\end{center}
\caption{$x_0=-2$, $x_f=-1$}
\label{fig1112}
\end{figure}

On Figure \ref{fig1112}, at the left, the code is initialized with $x(t)\equiv -1.1$, in order to promote the turnpike around $\bar x_{loc}=-1.347372066$.
At the right, the code is initialized with $x(t)\equiv 0.78$, in order to promote the turnpike around $\bar x = 0.781538640850898$.
We have computed trajectories for the following successives values of $T$:
$2.5$, $2.7$, $2.9$, $3.1$, $3.3$.
We observe that the blue and cyan trajectories at the top are optimal (see the value of their cost); and that the red, green and black trajectories at the bottom are optimal.
The bifurcation occurs around $T=2.9$.

Finally, on Figure \ref{figlast}, we represent the global optimal trajectory.
For $T\lesssim 2.9$, the global optimal trajectory makes a turnpike around $\bar x_{loc}=-1.347372066$, which is a local minimizer of the optimal static problem. For $T\geq 2.9$ we have a bifurcation and the global optimal trajectory makes a turnpike around $\bar x = 0.781538640850898$, which is the global minimizer of the optimal static problem.

\begin{figure}[H]
\centerline{\includegraphics[width=8cm]{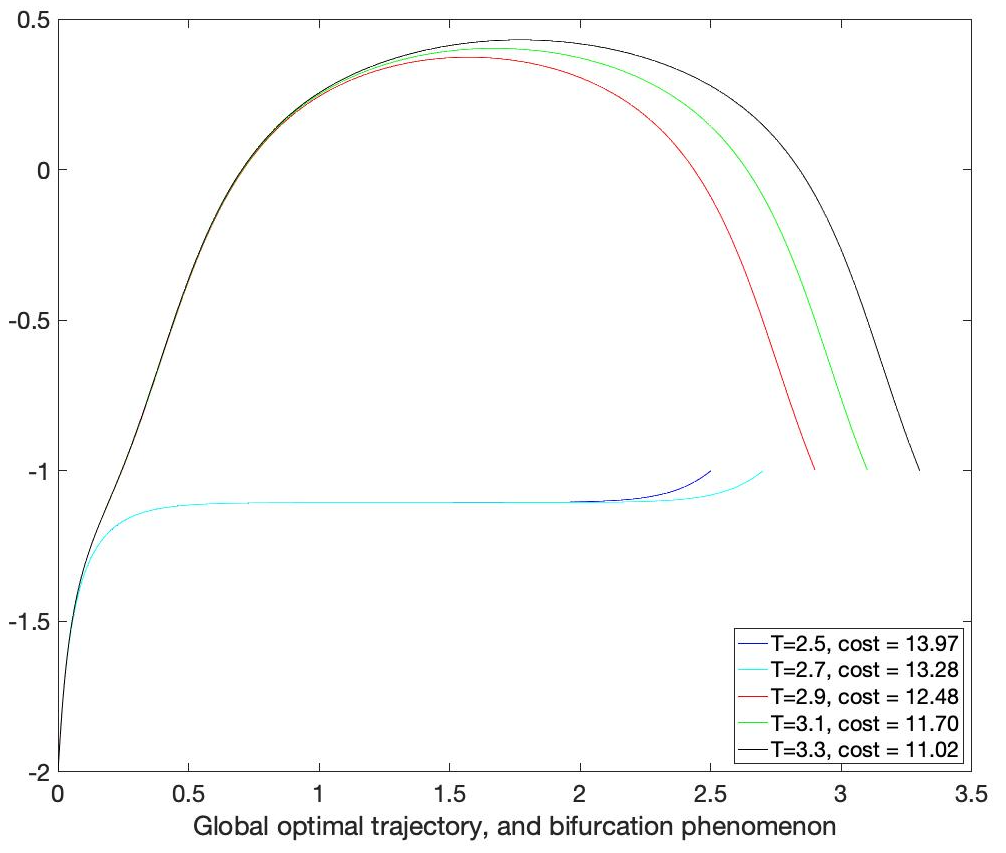}}
\caption{$x_0=-2$, $x_f=-1$}
\label{figlast}
\end{figure}

\end{document}